\documentclass{amsart}
\usepackage{amssymb,amsmath,amsfonts,amsthm}
\usepackage[dvips]{graphicx}
\usepackage{psfrag}

\usepackage{color}

\newtheorem{theorem}{Theorem}[section]
\newtheorem{lemma}[theorem]{Lemma}
 \newtheorem{proposition}[theorem]{Proposition}
\newtheorem{corollary}[theorem]{Corollary}
\theoremstyle{definition}

\newtheorem{example}[theorem]{Example}

\newtheorem{remark}[theorem]{Remark}

\newtheorem{ltheorem}{Theorem}

\def\real{\mathbb{R}}
\def\field{\mathbb{K}}
\def\complex{\mathbb{C}}
\def\integer{\mathbb{Z}}
\def\natural{\mathbb{N}}

\def\projective{\mathbb{P}}
\def\Hyp{\mathfrak{H}}
\def\cB{\mathcal{B}}
\def\cF{\mathcal{F}}
\def\cW{\mathcal{W}}
\def\cH{\mathcal{H}}
\def\cS{\mathcal{S}}

\def\cP{\mathcal{P}}
\def\cM{\mathcal{M}}

\def\cV{\mathcal{V}}

\def\fB{\mathfrak B}

\def\dist{\operatorname{dist}}
\def\dim{\operatorname{dim}}
\def\id{\operatorname{id}}
\def\quand{\quad\text{and}\quad}

\def\hf{\hat{f}}
\def\hF{\hat{F}}
\def\hA{\hat{A}}

\def\hsigma{\hat{\sigma}}
\def\hSigma{\hat{\Sigma}}

\def\hM{\hat{M}}
\def\hx{\hat{x}}
\def\hz{\hat{z}}

\def\hq{\hat{q}}
\def\hp{\hat{p}}
\def\hy{\hat{y}}
\def\hw{\hat{w}}
\def\hm{\hat{m}}
\def\hnu{\hat{\nu}}
\def\hmu{\hat{\mu}}
\def\tmu{\tilde{\mu}}
\def\hmuSigma{\hnu}
\def\muSigma{\nu}

\def\hphi{\hat{\phi}}

\def\hrho{\hat{\rho}}
\def\hvro{\hat{\varrho}}

\def\hxt{(\hx,t)}
\def\hys{(\hy,s)}
\def\hzr{(\hz,r)}
\def\grass{\operatorname{Grass}}
\def\grassl{\grass(l,d)}
\def\grassd{\grass(d-l,d)}
\def\sectl{\operatorname{sec}(K,\grassl)}
\def\sectd{\operatorname{sec}(K,\grassd)}
\def\projfield{\projective\mathbb{K}^{d}}

\def\graf{\operatorname{graph}\Hyp}
\def\exteriorl{\Lambda^{l}(\field^d)}
\def\exteriord{\Lambda^{(d-l)}(\field^d)}
\def\L1{L^1_{\mu}(M,N)}
\def\0{\hat{0}}

\def\GL{\operatorname{GL}(d,\field)}

\def\SLthreereal{\operatorname{SL}(3,\real)}

\def\limlog{\lim_{n\to \infty}\frac{1}{n}\log}
\def\adj{\hat{A}_{*}}
\def\loc{{\operatorname{loc}}}
\def\spann{{\operatorname{span}}}
\newcommand{\ip}[1]{{\left\langle #1 \right\rangle}}
\newcommand{\norm}[1]{{\left\lVert #1 \right\rVert}}

\title[Simple Lyapunov Spectrum]
{Simple Lyapunov spectrum for certain linear cocycles over partially hyperbolic maps}
\author{Mauricio Poletti and Marcelo Viana}
\date{\today}



\thanks{Work partially supported by Fondation Louis D. -- Institut de France (project coordinated by M. Viana).
The authors were also supported by CNPq, FAPERJ and CAPES}
\address{IMPA -- Estrada D. Castorina 110, Jardim Bot\^anico, 22460-320 Rio de Janeiro, Brazil.}
\email{mpoletti@impa.br, viana@impa.br}

\begin{document}

\begin{abstract}
Criteria for the simplicity of the Lyapunov spectra of linear cocycles have been found by
Furstenberg, Guivarc'h-Raugi, Gol'dsheid-Margulis and, more recently, Bonatti-Viana and Avila-Viana.
In all the cases, the authors consider cocycles over hyperbolic systems, such as shift maps
or Axiom A diffeomorphisms.

In this paper we propose to extend such criteria to situations where the base map is just
partially hyperbolic. This raises several new issues concerning, among others, the recurrence
of the holonomy maps and the (lack of) continuity of the Rokhlin disintegrations of $u$-states.

Our main results are stated for certain partially hyperbolic skew-products whose iterates
have bounded derivatives along center leaves. They allow us, in particular, to exhibit non-trivial examples
of stable simplicity in the partially hyperbolic setting.
\end{abstract}

\maketitle

\section{Introduction}
The theory of linear cocycles is now a classical field of dynamical systems and ergodic
theory, grounded on the pioneer works of Furstenberg, Kesten~\cite{FK60,Fur63} and
Oseledets~\cite{Ose68}.
The derivatives of smooth dynamical systems are the first examples that come to mind,
but the notion of linear cocycle is a lot more broad, and arises naturally in many other
situations, e.g., in the spectral theory of Schr\"odinger operators.

Among the outstanding issues is the problem of simplicity: \emph{when is it the case that the
dimension of all Oseledets subspaces is equal to 1?}
This was first studied by Furstenberg~\cite{Fur63}, Guivarc'h-Raugi~\cite{GR86} and
Gol'dsheid-Margulis~\cite{GM89}, who obtained explicit simplicity criteria for random i.i.d.
products of matrices.
Recently, Bonatti-Viana~\cite{BoV04} and Avila-Viana~\cite{AvV1} extended the theory to include
a much broader class of (H\"older continuous) cocycles over hyperbolic maps.
There is also much progress in the quasi-periodic case, that is, for linear cocycles over rotations:
see Duarte-Klein \cite{DuK} and references therein.

Our purpose in this paper is to initiate the study of the simplicity problem in the context of linear
cocycles over partially hyperbolic maps, that combine features from both the hyperbolic and the
quasi-periodic cases.

The theory of partially hyperbolic diffeomorphisms and flows was initiated by Brin-Pesin~\cite{BP74}
and Hirsch-Pugh-Shub~\cite{HPS77} and has been at the heart of much recent progress in dynamical
systems. While boasting many of the important properties of uniformly hyperbolic (Axiom A) systems,
partially hyperbolic maps are a lot more flexible and encompass several interesting new phenomena.

Linear cocycles over volume-preserving partially hyperbolic maps were studied previously by
Avila-Santamaria-Viana~\cite{ASV13}. The issue of simplicity is much better understood when the
base map is non-uniformly hyperbolic, meaning that all the center Lyapunov exponents are non-zero.
Indeed, Viana~\cite{Almost} proved that simplicity is generic, in a very strong sense, among
$2$-dimensional cocycles. Backes-Poletti-Varandas~\cite{BPV} extended that conclusion to any
dimension $d\ge 2$, under additional assumptions such as fiber-bunching.

For this reason, here we focus on the opposite case, namely, we take the partially hyperbolic map
to be \emph{mostly neutral along the center direction}, meaning that its iterates have bounded
derivatives along the leaves of the center foliation. The following simple example illustrates
some of the systems we have in mind.

Let $\omega_0, \omega_1$ be real numbers and $f_0, f_1:S^1\to S^1$ be the corresponding rotations,
that is, $f_i(t)=t+\omega_i \mod \integer$ for every $t\in S^1$. Take $\omega_0$ to be irrational.
Let $A_0:S^1\to \SLthreereal$ and $A_1:S^1\to\SLthreereal$ be given by
\begin{equation*}
\begin{aligned}
A_0(t) = \left(\begin{array}{ccc} 2 & 0 & 0 \\ 0 & 1 & 0 \\ 0 & 0 & 2^{-1} \end{array}\right)
\text{ and }
A_1(t) = R_1(t) R_2(t) R_3(t)
\end{aligned}
\end{equation*}
where $R_i(t)$ denotes the rotation of angle $2 \pi t$ around the $i$-th axis.
Each $A_i$ defines a linear cocycle $F_i$ over the transformation $f_i$. We want to consider the random
combination $\hF$ of these two cocycles: at each step one applies either $F_1$ or $F_2$, at random.

The results in this paper (see Theorem~\ref{teo.B} and Example~\ref{ex.main}) ensure that the Lyapunov
spectrum of $\hF$ is simple, and the same is true for any small perturbation in the uniform topology.
Concerning this last point, it should be noted that simplicity of the Lyapunov spectrum is usually not
an open property, cf. Wang, You~\cite{WaY}.

\subsubsection*{Acknowledgements.} We are grateful to Lucas Backes, Fernando Lenarduzzi, Enrique Pujals
and Jiagang Yang for numerous discussions, and to the anonymous referee for a thoughtful review of the
manuscript and many suggestions that helped improve the text.

\section{Definitions and statements}\label{s.2}

Here we state our main result. Beforehand, we must give the precise definitions of the notions involved
in the statement. In what follows, $\field$ denotes either the real field $\real$ or the complex field
$\complex$, indifferently.

\subsection{Linear cocycles and Lyapunov exponents}

The \emph{linear cocycle} defined by a measurable matrix-valued function $\hA:\hM\to \GL$
over an invertible measurable map $\hf:\hM\to \hM$ is the (invertible) map
$\hF_A:\hM\times \field^d\to \hM\times \field^d$ given by
\begin{equation*}
\hF_A\big(\hp,v\big)=\big(\hf(\hp),\hA(\hp)v\big).
\end{equation*}
Its iterates are given by $\hF^n_A\big(\hp,v\big)=\big(\hf^n(\hp),\hA^n(\hp)v\big)$ where
\begin{equation*}
\hA^n(\hp)=
\left\{
	\begin{array}{ll}
		\hA(\hf^{n-1}(\hp))\ldots \hA(\hf(\hp))\hA(\hp)  & \text{if } n>0 \\
		\id & \text{if } n=0 \\
		\hA(\hf^{n}(\hp))^{-1}\ldots \hA(\hf^{-1}(\hp))^{-1}& \text{if } n<0. \\
	\end{array}
\right.
\end{equation*}

Let $\hmu$ be an $\hf$-invariant probability measure on $\hM$ such that $\log\|\hA^{\pm 1}\|$
are integrable. By Oseledets~\cite{Ose68}, at $\hmu$-almost every point
$\hp\in \hM$ there exist real numbers $\lambda_1 \left(\hp\right)>\cdots >\lambda_k\left(\hp\right)$
and a decomposition $\field^d=E^1_{\hp} \oplus \cdots \oplus E^k_{\hp}$ into vector subspaces such that
\begin{equation*}
\hA(\hp)E^i_{\hp}=E^i_{\hf\left(\hp\right)}
\text{ and } \lambda_i(\hp)=\lim_{\mid n\mid \to \infty}\frac{1}{n}\log\|\hA^n(\hp)v\|
\end{equation*}
for every non-zero $v\in E^i_{\hp}$ and $1\leq i \leq k$. The dimension of $E^i_{\hp}$ is called the
\emph{multiplicity} of $\lambda_i(\hp)$.

In this work we assume that the invariant measure $\hmu$ is ergodic.  Then the Lyapunov exponents and the
dimensions of the subspaces $E^i_{\hp}$ are constant almost everywhere.
The \emph{Lyapunov spectrum} of the cocycle is the set of all Lyapunov exponents. The following notion is
central to the whole paper: the Lyapunov spectrum is \emph{simple} if it contains exactly $d$ distinct Lyapunov
exponents or, equivalently, if every Lyapunov exponent has multiplicity equal to $1$.

\subsection{Partially hyperbolic skew-products}

Let $\hsigma:\hSigma\to \hSigma$ be any two-sided finite or countable shift.
By this we mean that $\hSigma$ is the set of two-sided sequences $(x_n)_{n\in\integer}$
in some set $X\subset\natural$ with $\# X>1$, and the map $\hsigma$ is given by
\begin{equation*}
\hsigma\left(\left(x_n\right)_{n\in \integer}\right)=\left(x_{n+1}\right)_{n\in \integer}.
\end{equation*}
Let $\dist_{\hSigma}:\hSigma\times \hSigma\to \real$ be the distance defined by
\begin{equation}\label{eq.distanceSigma}
\dist_{\hSigma}(\hx,\hy)=\sum_{k=-\infty}^{\infty} 2^{-\lvert k\rvert} \delta(x_k,y_k),
\quad\text{with $\hx=\left(x_k\right)_{k\in \integer}$ and $\hy=\left(y_k\right)_{k\in \integer}$,}
\end{equation}
where $\delta(x,y)=1$ if $x=y$ and $\delta(x,y)=0$ otherwise.
Then $\hsigma$ is a \emph{hyperbolic homeomorphism} (in the sense of \cite{Almost}),
as we are going to explain.

Given any $\hx\in \hSigma$, we define the
\emph{local stable} and \emph{unstable sets} of $\hx$ with respect to $\hsigma$ by
\begin{align*}
W^{s}_{\loc}\left(\hy\right) & =\lbrace \hx: x_k=y_k\text{ for every }k\geq 0\rbrace \text{ and }\\
W^{u}_{\loc}\left(\hy\right) & =\lbrace \hx: x_k=y_k\text{ for every }k\leq 0\rbrace.
\end{align*}
From now on we fix $\lambda=1/2$ and $\tau = 1/2$. Then,
\begin{itemize}
\item[(i)] $\dist_{\hSigma}(\hsigma^n(\hy_1),\hsigma^n(\hy_2)) \leq \lambda^n \dist_{\hSigma}(\hy_1,\hy_2)$ for any
      $\hy \in \hSigma$, $\hy_1, \hy_2 \in W^s_{\loc} (\hy)$ and $n \geq 0$;
\item[(ii)] $\dist_{\hSigma}(\hsigma^{-n}(\hy_1), \hat{\hsigma}^{-n}(\hy_2)) \leq \lambda^n \dist_{\hSigma}(\hy_1,\hy_2)$ for any
      $\hy \in \hSigma$, $\hy_1, \hy_2 \in W^u_{\loc} (\hy)$ and $n \geq 0$;
\item[(iii)] if $\dist_{\hSigma}(\hx, \hy)\leq\tau$, then $W^s_{\loc}(\hx)$ and $W^u_{\loc}(\hy)$ intersect in a unique point, which is denoted by $[x,y]$ and depends continuously on $\hx$ and $\hy$.
\end{itemize}

By a \emph{partially hyperbolic skew-product} over the shift map $\hsigma$ we mean a homeo\-morphism
$\hf:\hSigma\times K\to \hSigma \times K$ of the form
\begin{equation*}
\hf(\hx,t)=\big(\hsigma(\hx),\hf_{\hx}(t)\big)
\end{equation*}
where $K$ is a compact Riemannian manifold and the maps $\hf_{\hx}:K\to K$ are diffeomorphisms satisfying
\begin{equation}\label{eq.domination}
\lambda \|D\hf_{\hx}(t)\| < 1 \text{ and } \lambda \|D\hf_{\hx}^{- 1} (t)\| < 1 \text{ for every } (\hx,t)\in\hSigma\times K,
\end{equation}
where $\lambda$ is a constant as in (i) - (ii). We also assume the following H\"older condition:
there exist $C>0$ and $\alpha>0$ such that the $C^ 1$-distance between $\hf_{\hx}$ and $\hf_{\hy}$
is bounded by $C\dist_{\hSigma}(\hx,\hy)^\alpha$ for every $\hx, \hy \in \hSigma$.

We say that $\hf$ has \emph{mostly neutral center direction} if the maps $\hf^n_{\hx}:K\to K$
defined for $n \in \integer$ and $\hx\in \hSigma$ by
\begin{equation*}
\hf_{\hx}^n =
\left\{
	\begin{array}{ll}
		\hf_{\hsigma^{n-1}(\hx)} \circ \cdots \circ \hf_{\hx}  & \text{if } n>0 \\
		\id & \text{if } n=0 \\
		\hf_{\hsigma^n(x)}^{-1} \circ \cdots \circ \hf_{\hsigma^{-1}(\hx)}^{-1} & \text{if } n<0. \\
	\end{array}
\right.
\end{equation*}
have bounded derivatives, that is, if there exists $C>0$ such that
$$
\|D\hf_{\hx}^n\| \leq C \text{ for every $\hx\in\hSigma$ and  $n\in\integer$.}
$$

\begin{remark}\label{r.equicontinuity}
Clearly, this implies that the $\{\hf^{n}_{\hx}: j \in \integer \text{ and } \hx \in \hSigma\}$ is equicontinuous.
When the maps $\hf_{\hy}$ are $C^{1+\epsilon}$, equicontinuity alone suffices for all our purposes
(see Remark~\ref{equicontinuitysuffices}).
\end{remark}

In the definition of partially hyperbolic skew-product, one may replace the shift
$\hsigma:\hSigma\to\hSigma$ with a sub-shift $\hsigma_T:\hSigma_T\to\hSigma_T$ associated to
a transition matrix $T=(T_{i,j})_{i,j\in X}$. By this we mean that $T_{i,j}\in\{0,1\}$ for every $i, j \in X$
and $\hsigma_T$ is the restriction of the shift map $\hsigma$ to the subset $\hSigma_T$ of
sequences $(x_n)_{n\in\integer}$ such that $T_{x_n,x_{n+1}}=1$ for every $n\in\integer$.

One way to reduce the sub-shift case to the full shift case is through inducing.
Namely, fix any $1$-cylinder $[0;i]=\{(x_n)_{n\in\integer}\in\hSigma_T: x_0 = i\}$ with positive measure
and consider the first return map $g:[i] \to [i]$ of $\hsigma_T$ to $[i]$.
This is conjugate to a full countable shift (with the return times as symbols) and it preserves the
normalized restriction to the cylinder of the $\hsigma_T$-invariant measure. All the conditions that
follow are not affected by this procedure. Moreover, every linear cocycle $F$ over $\hsigma_T$
gives rise, also through inducing, to a linear cocycle over $g$ whose Lyapunov spectrum is just
a rescaling of the Lyapunov spectrum of $F$.
In particular, simplicity may also be read out from the induced cocycle.

%{\color{red}Moreover, although we choose to formulate our approach in a symbolic set-up,
% for skew-products over shifts, it is clear that it extends to other situations that are more
%geometric in nature.
%For example, take $g:N\to N$ to be a partially hyperbolic diffeomorphism on a compact
%$3$-dimensional manifold $N$ and assume that there exists an embedded closed curve
%$\gamma\subset N$ such that $g(\gamma)=\gamma$ and some connected component of
%$W^s_{\loc}(\gamma)\cap W^{u}_{\loc}(\gamma)\setminus \gamma$ is a closed curve.
%By \cite{BW05}, $g$ is conjugate up to finite covering to a skew-product over a
%linear Anosov diffeomorphism of the 2-torus.
%Thus, using a Markov partition for the Anosov map, one can semi-conjugate $g$ to a
%partially hyperbolic skew-product over a sub-shift of finite type.
%In this way, the conclusions of this paper can be adapted to linear cocycles over such
%a diffeomorphism.}

\subsection{Stable and unstable linear holonomies}

Property \eqref{eq.domination} is a condition of \emph{domination} (or \emph{normal hyperbolicity},
in the spirit of~\cite{HPS77}): it means that any expansion and contraction of
$\hf_{\hx}$ along the fibers $\{\hx\}\times K$ are dominated by the hyperbolicity of the base map
$\hsigma$. For our purposes, its main relevance is that it ensures the existence of strong-stable
and strong-unstable ``foliations'' for $\hf$, as we explain next.

Let the product $\hM=\hSigma\times K$ be endowed with the distance defined by
\begin{equation*}
\dist_{\hM}((\hx_1,t_1),(\hx_2,t_2)) = \dist_{\hSigma}(\hx_1,\hx_2)+\dist_{K}(t_1,t_2),
\end{equation*}
where $\dist_{\hSigma}$ denotes the distance \eqref{eq.distanceSigma} on $\hSigma$ and
$\dist_K$ is the distance induced by the Riemannian metric on $K$.

We consider the \emph{stable holonomies}
\begin{equation*}
h^{s}_{\hx,\hy}:K\to K, \quad
h^{s}_{\hx,\hy}=\lim_{n\to \infty}\big(\hf^{n}_{\hy}\big)^{-1} \circ \hf^{n}_{\hx},
\end{equation*}
defined for every $\hx$ and $\hy$ with $\hx\in W^{s}_{\loc}(\hy)$,
and \emph{unstable holonomies}
\begin{equation*}
h^{u}_{\hx,\hy}:K\to K, \quad
h^{u}_{\hx,\hy}=\lim_{n\to\infty}\big(\hf^{-n}_{\hy}\big)^{-1} \circ\hf^{-n}_{\hx}
\end{equation*}
defined for every $\hx$ and $\hy$ with $\hx\in W^{u}_{\loc}(\hy)$. That these families of maps
exist follows from the assumption \eqref{eq.domination}, using arguments from \cite{BGV03}.
See for instance \cite{BaK16}, which deals with a similar setting.

We define the \emph{local strong-stable set} and the \emph{local strong-unstable set} of each
$(\hx,t)\in \hM$ to be
\begin{align*}
W^{ss}_{\loc}\left(\hx,t\right)&=\lbrace \left(\hy,s\right)\in \hM: \hy\in W^s_{\loc}(\hx)
\text{ and }s=h^{s}_{\hx,\hy}(t)\rbrace \text{ and}\\
W^{uu}_{\loc}\left(\hx,t\right)&=\lbrace \left(\hy,s\right)\in \hM: \hy\in W^u_{\loc}(\hx)
\text{ and }s=h^{u}_{\hx,\hy}(t)\rbrace,
\end{align*}
respectively. It is easy to check that
\begin{equation*}
(\hy,s)\in W^{ss}_{\loc}(\hx,t) \quad\Rightarrow\quad
\lim_{n\to +\infty}\dist_{\hM}(\hf^n(\hy,s),\hf^n(\hx,t))=0
\end{equation*}
and analogously on strong-unstable sets for time $n\to -\infty$.

\subsection{Measures with partial product structure}\label{ss.partialproduct}

Recall that $\hM=\hSigma\times K$.
Throughout, we take $\hmu$ to be an $\hf$-invariant measure with \emph{partial product structure}, that is,
a probability measure of the form $\hmu=\hrho \, \mu^{s} \times \mu^{u} \times \mu^{c}$ where:
\begin{itemize}
\item $\hrho:\hM\to (0,+\infty)$ is a continuous function  bounded away from zero and infinity;
\item $\mu^{s}$ is a probability measure supported on $\Sigma^{-}=X^{\integer_{<0}}$;
\item $\mu^{u}$ is a probability measure supported on $\Sigma^{+}=X^{\integer_{\ge 0}}$;
\item $\mu^{c}$ is a probability measure on the manifold $K$.
\end{itemize}

For notational convenience, we formulate the boundedness condition as follows: there exists $ \kappa>0$ such that
\begin{equation}\label{eq.boundedness}
 \frac{1}{\kappa}\leq \frac{\tilde{\rho}(x^s,x^u)}{\tilde{\rho}(x^s,z^u)}\leq \kappa \quand  \frac{1}{\kappa}\leq \frac{\tilde{\rho}(x^s,x^u)}{\tilde{\rho}(z^s,x^u)}\leq \kappa
\end{equation}
for every $x^s$, $z^s\in \Sigma^{-}$ and  $x^u$, $z^u\in \Sigma^{-}$, where
$\tilde{\rho}:\hSigma\to \real$ is defined by
\begin{equation}\label{eq.rhotil}
\tilde\rho(\hx)=\int \hrho(\hx,t)d\mu^c(t).
\end{equation}
Observe that when $\hSigma$ is a finite shift space this is an immediate consequence of
compactness and the continuity of $\hrho$.

Now define, for $\hx\in\hSigma$,
\begin{equation}\label{eq.hatvarrho}
\hvro(\hx,\cdotp)=\frac{\hrho(\hx,\cdotp)}{\tilde\rho(\hx)}
\quad\text{and}\quad
\hmu^c_{\hx}=\hvro(\hx,\cdot) \, \mu^{c}.
\end{equation}
In other words, $\hmu_{\hx}^c$ is the normalization of $\hrho(\hx,\cdot)\mu^c$.
Note that $\{\hmu^{c}_{\hx}: \hx\in\hSigma\}$ is a (continuous) disintegration of
$\hmu$ along vertical fibers, that is, with respect to the partition
$\hat{\cP}=\lbrace \lbrace \hx\rbrace \times K: \hx\in \hSigma \rbrace$.

The assumption that $\hmu$ is invariant under $\hf$, together with the fact that
$\hmu^c_{\hx}$ depends continuously on $\hx$, implies that
\begin{equation}\label{eq.disintegration}
(\hf_{\hx})_*\hmu^c_{\hx} = \hmu^c_{\hsigma(\hx)}
\quad\text{for every $\hx\in\hSigma$.}
\end{equation}
We will also see in Section~\ref{s.ustates} that this disintegration is \emph{holonomy invariant}:
\begin{equation}\label{eq.abs_cont}
\begin{aligned}
(h^s_{\hx,\hy})_*\hmu^c_{\hx} & = \hmu^c_{\hy} \text{ whenever } \hy \in W^s(\hx)\quand\\
(h^u_{\hx,\hy})_*\hmu^c_{\hx} & = \hmu^c_{\hy} \text{ whenever } \hy \in W^u(\hx).
\end{aligned}
\end{equation}

\begin{remark}\label{r.homoclinic}
In particular, if $\hx$ is a fixed point of the shift map then $\hmu^c_{\hx}$ is invariant under
$\hf_{\hx}$. Clearly, it is equivalent to $\mu^c$. Moreover, if $\hy$ is a homoclinic point of $\hx$,
that is, a point in $W^s(\hx) \cap W^u(\hx)$, then
$(h_{\hy,\hx}^s \circ h_{\hx,\hy}^u)_*\hmu^c_{\hx} = (h_{\hy,\hx}^u \circ h_{\hx,\hy}^s)_*\hmu^c_{\hx} = \hmu^c_{\hx}$.
\end{remark}

\subsection{Linear cocycles with holonomies}

Let $\hA:\hM \to \GL$ be a $\alpha$-H\"older continuous map for some $\alpha>0$.
By this we mean that there exists $C>0$ such that
\begin{equation*}
\|\hA(\hp)-\hA(\hq)\| \leq C \dist_{\hM}(\hp,\hq)^{\alpha} \quad\text{for any } \hp, \hq\in \hM.
\end{equation*}
The linear cocycle defined by $\hA$ over the transformation $\hf:\hM\to\hM$ is the map
$\hF:\hM\times \field^d\to \hM\times \field^d$ defined by
\begin{equation*}
\hF(\hp,v)=\big(\hf(\hp),\hA(\hp)v\big).
\end{equation*}
In what follows we take the cocycle to admit \emph{stable} and \emph{unstable linear holonomies}.
Let us explain this.

By \emph{stable linear holonomies} we mean a family of linear maps
$H^{s}_{\hp,\hq}:\field^d\to\field^d$,
defined for each $\hp, \hq \in \hM$ with $\hq \in W^{ss}_{\loc}(\hp)$ and such that, for some constant $L>0$,
\begin{enumerate}
\item[(a)] $H^{s}_{\hf^j(\hp),\hf^j(\hq)}=\hA^j(\hq) \circ H^s_{\hp,\hq} \circ \hA^j(\hp)^{-1}$ for every $j\geq 1$;
\item[(b)] $H^{s}_{\hp,\hp}=\id$ and $H^s_{\hp,\hq}=H^s_{\hz,\hq}\circ H^s_{\hp,\hz}$ for any $\hz\in W^{ss}_{\loc}(\hp)$;
\item[(c)] $\|H^s_{\hp,\hq}-\id\| \leq L \dist_{\hM}(\hp,\hq)^{\alpha}$;
\item[(d)] $(\hp,\hq) \mapsto H^s_{\hp,\hq}$ is uniformly continuous on
           $\lbrace (\hp,\hq): \hq \in W^{ss}_{\loc}(\hp)\rbrace \subset \hM\times \hM$.
\end{enumerate}
\emph{Unstable linear holonomies} $H^{u}_{\hp,\hq}:\field^d\to\field^d$ are defined analogously,
for the pairs $(\hp,\hq)\in\hM\times \hM$ with $\hq \in W^{uu}_{\loc}(\hp)$.

These notions were introduced in \cite{BGV03,ASV13}, where they were called simply stable and unstable holonomies.
We add the adjective \emph{linear} to avoid any confusion with the holonomies $h^s$ and $h^u$ in the previous paragraph,
that concern only the base dynamics, whereas $H^s$ and $H^u$ pertain to the linear cocycle.

It was shown in \cite{ASV13} that stable and unstable linear holonomies do exist, in particular,
when the cocycle is \emph{fiber-bunched}. By the latter we mean that there exist $C>0$ and $\theta <1$ such that
\begin{equation*}
\Vert \hA^n(\hp)\Vert \Vert \hA^n(\hp)^{-1}\Vert \lambda ^{n \alpha} \leq C \theta ^n
\quad\text{for every $\hp\in\hM$ and $n\geq 0$,}
\end{equation*}
where $\lambda$ is a hyperbolicity constant for $\hf$ as in conditions (i)-(ii) above. Then stable and unstable
linear holonomies are given by
$$
H^{s}_{\hp,\hq}=\lim_{n\to \infty}\hA^{n}(\hq)^{-1} \circ \hA^{n}(\hp),
\quad\text{and}\quad
H^{u}_{\hp,\hq}=\lim_{n\to\infty} \hA^{-n}(\hq)^{-1} \circ\hA^{-n}(\hp)
$$

%ACHO QUE NÃO PRECISA
%For each $\hx, \hy\in\hSigma$ with $\hy\in W^{s}(\hx)$, define
%$\rH^{s}_{\hx,\hy} : K \times \complex^{d} \to K\times \complex^{d}$ by
%\begin{equation*}
%\rH^{s}_{\hx,\hy}(s,v)=(t,H^{s}_{(\hx,s),(\hy,t)} v), \quad\text{where } t= h^{s}_{\hx,\hy}(s).
%\end{equation*}
%When $\hx$ and $\hy$ can be guessed from the context, we allow ourselves to write
%$\rH^s_{\hx,\hy}(s,v)=(h^s(s),H^s_{s} v)$.
%We define $\rH^{u}_{\hx,\hy} : K \times \complex^{d} \to K\times \complex^{d}$ analogously,
%for $\hy\in W^u(\hx)$.

\subsection{Pinching and twisting}\label{ss.pinching_twisting}

Now we state our criterion for simplicity of the Lyapunov spectrum. It is assumed that the cocycle admits
stable and unstable linear holonomies.

We call $\hF$ \emph{pinching} if there exists some fixed (or periodic) vertical leaf $\ell=\{\hx\}\times K$
such that the restriction to $\ell$ of every exterior power $\Lambda^k\hF$ has simple Lyapunov spectrum,
relative to the $\hf_{\hx}$-invariant measure $\hmu^c_{\hx}$ (recall Remark~\ref{r.homoclinic}).
In other words, the Lyapunov exponents $\lambda_1, \cdots, \lambda_d$ are such that,
for each $1 \leq k \le d-1$ and $\hmu^c_{\hx}$-almost every $t\in K$, the sums
$$
\lambda_{i_1}(\hx,t)+ \cdots +\lambda_{i_k}(\hx,t),
\qquad 1 \le i_1 < \cdots < i_k \le d
$$
are all distinct.

Next, take $\hF$ to be pinching and let $\field^d = E^1(t) \oplus \cdots \oplus E^d(t)$ be the Oseledets
decomposition at each point $(\hx,t)\in\ell$. This is defined on a full $\hmu^c_{\hx}$-measure set.
Choose (measurably) unit vectors $e^i(t)\in E^i(t)$. Let $\hy$ be a homoclinic point of $\hx$.
Given $t\in S^1$, denote $t_1=h^u_{\hx,\hy}(t)$ and $t_2=h^s_{\hy,\hx}(t_1)$.
Then define $\cB(t)$ to be the matrix of the linear map
\begin{equation}\label{eq.HH}
H^u_{(\hy,t_1),(\hx,t)} \circ H^s_{(\hx,t_2),(\hy,t_1)}:\field^d \to \field^d
\end{equation}
relative to the bases $\{e^1(t_2), \dots, e^d(t_2)\}$ and $\{e^1(t), \dots, e^d(t)\}$,
respectively. Observe that $t_2$ also varies on a full $\hmu^c_{\hx}$-measure set,
since the composition of the holonomies preserves $\hmu^c_{\hx}$ (Remark~\ref{r.homoclinic}).

We call the cocycle $\hF$ \emph{twisting} if, for some choice of the homoclinic point $\hy$,
all the algebraic minors $m_{I,J}(t)$ of $\cB(t)$ are non-zero for $\hmu^c_{\hx}$-almost every
$t\in K$ and they decay sub-exponentially along the orbits of $\hf_{\hx}$, meaning that
\begin{equation}\label{eq.suffar}
\limlog \vert m_{I,J}(\hf_{\hx}^{n}(t))\vert=0
\quad\text{for $\hmu^c_{\hx}$-almost every $t\in K$}
\end{equation}
and any proper subsets $I$ and $J$ of $\{1, \dots, d\}$.

\begin{remark}
It is well-known (see~\cite[Proposition~2.2]{LeS82} or \cite[Corollary~3.11]{LLE}) that the
property \eqref{eq.suffar} holds whenever the function
$\log |m_{I,J}|\circ \hf_{\hx}-\log |m_{I,J}|$ is $\hmu^c_{\hx}$-integrable.
In Example~\ref{ex.main} we show how to check the twisting condition in a specific case,
using this observation.
\end{remark}

%It is clear that this holds if $\hF$ is uniformly twisting, because in this case the algebraic minors are
%uniformly bounded away from zero.
Finally, we say that the cocycle $\hF$ is \emph{simple} if it is both pinching and twisting (in addition to
admitting stable and unstable linear holonomies).
%That is the case, in particular, if $\hF$ is uniformly pinching and uniformly twisting.

\subsection{Main statement}

Let $H^{\alpha}(\hM)$ denote the space of all $\alpha$-H\"older continuous maps $\hA:\hM\to\GL$.
The norm
\begin{equation*}
\| \hA \|_{\alpha}
= \sup _{\hp\in \hM} \|\hA(\hp)\| + \sup _{\hp\neq\hq} \frac{\| \hA(\hp)-\hA(\hq)\|}{\dist_{\hM}(\hp,\hq)^{\alpha}}
\end{equation*}
defines a topology in $H^{\alpha}(\hM)$ that we call \emph{$\alpha$-H\"older topology}.

We say that $A$ is a \emph{continuity point for the Lyapunov exponents} if, for every $1\leq i\leq d$,
the function $\lambda_i:H^\alpha(\hM)\to \real$ is continuous in $A$.

\begin{ltheorem}\label{teo.B}
Let $\hf:\hM\to\hM$ be a partially hyperbolic skew-product with mostly neutral center direction and
$\hmu$ be a $\hf$-invariant measure with partial product structure.
Suppose that $\hA\in H^{\alpha}(\hM)$ is such that the corresponding linear cocycle
$\hF:\hM\times \field^d\to \hM\times \field^d$ over $\hf$ is simple.
Then $\hA$ is a continuity point for the Lyapunov exponents, and the Lyapunov spectrum of $\hF$ is simple.
\end{ltheorem}

By continuity, the Lyapunov spectrum remains simple for every perturbation of $\hF$, that is, for the linear cocycle
over $\hf$ corresponding to every element of $H^\alpha(\hM)$ sufficiently close to $\hA$.

\begin{example}\label{ex.main}
The example presented in the Introduction satisfies all the conditions in Theorem~\ref{teo.B}, and so the conclusion
applies to it. In order to explain this, let us formalize the example as follows.

Let $\hsigma:\hSigma\to\hSigma$ be the shift map on $\hSigma=\{0,1\}^\integer$ and let $\hnu$ be the Bernoulli measure
$(\delta_0/2 + \delta_1/2)^\integer$. Let $\hf:\hSigma\times S^1 \to \hSigma \times S^1$ be defined by
$\hf(\hx,t)=(\hsigma(\hx),f_{x_0}(t))$ and $\hmu$ be the product of $\hnu$ by the Haar measure on $S^1$.
Finally, let $\hM=\hSigma\times S^1$ and $\hF:\hM\times\real^3\to\hM\times\real^3$ be given by
$\hF((\hx,t),v)=(\hf(\hx,t),\hA(\hx,t)v)$ with $\hA(\hx,t)=A_{x_0}(t)$.

It is clear that $\hf$ is a skew-product with mostly neutral central direction, and $\hmu$ has partial product structure.
Moreover, $\hmu$ is ergodic. Indeed, let $\zeta$ be any ergodic component. Since $\hmu$ projects down to $\hnu$,
which is ergodic, $\zeta$ must project to $\hnu$. The Lyapunov exponent of $\hf$ along the vertical $S^1$ fibers is zero
and so, by the Invariance Principle of~\cite{Extremal}, there exists a disintegration $\{\zeta_{\hz}:\hz\in\hSigma\}$ of
$\zeta$ along the $S^1$ fibers which is invariant under stable and unstable holonomies and is continuous.
Consider the fixed point $\hx=(\dots, 0, 0, 0, \dots)$ of $\hsigma$. Since $f_0$ is uniquely ergodic, because we took
$\omega_0$ to be irrational, $\zeta_{\hx}$ must coincide with the Haar measure on $S^1$. Then, by holonomy invariance,
$\zeta_{\hz}$ is uniquely determined at every point, which proves that the ergodic component is unique.

Now consider the homoclinic point $\hy=(\dots, 0, 1, 0, \dots)$ of $\hx$, where the sole non-zero entry is in
position $0$. The corresponding stable and unstable holonomies are given by
$$
\begin{aligned}
h^{s}_{\hx,\hy}
& % = \lim_{n\to \infty} \big(\hf^{n}_{\hy}\big)^{-1} \circ \hf^{n}_{\hx}
  = \lim_{n\to \infty} \big(f_0 \circ \cdots f_0 \circ f_1\big)^{-1} \circ \big(f_0 \circ \cdots \circ f_0\big)
  = f_1^{-1} \circ f_0 \quad\text{and}\\
h^{u}_{\hy,\hx}
& % = \lim_{n\to\infty} \big(\hf^{-n}_{\hx}\big)^{-1} \circ \hf^{-n}_{\hy}
  = \lim_{n\to\infty} \big(f_0^{-1} \circ \cdots f_0^{-1}\big)^{-1} \circ \big(f_0^{-1} \circ \cdots f_0^{-1}\big)
  = \id.
\end{aligned}
$$
Similarly, the stable and unstable linear holonomies are given by
$$
\begin{aligned}
H^{s}_{(\hx,s),(\hy,t)}
& % = \lim_{n\to \infty} \hA^{n}(\hy,t)^{-1} \circ \hA^{n}(\hx,s)\\
% &  = \lim_{n\to \infty} \big(A_0 \circ \cdots A_0 \circ A_1(t)\big)^{-1} \circ \big(A_0 \circ \cdots \circ A_0\big)
  = A_1(t)^{-1} \circ A_0 \quad\text{and}\quad
H^{u}_{(\hy,t),(\hx,t)}
& % = \lim_{n\to\infty} \big(\hA^{-n}(\hx,t)\big)^{-1} \circ \hA^{-n}(\hx,t)\\
% &  = \lim_{n\to\infty} \big(A_0^{-1} \circ \cdots A_0^{-1}\big)^{-1} \circ \big(A_0^{-1} \circ \cdots A_0^{-1}\big)
  = \id
\end{aligned}
$$
where $s = h^{s}_{\hy,\hx}(t) = f_0^{-1}(f_1(t)) = t+\omega_1-\omega_0$.

It is also clear that $\hF$ is pinching: its restriction to $\ell=\{\hx\}\times S^1$ corresponds to the constant cocycle
$$
A_0(t)=\left(\begin{array}{ccc} 2 & 0 & 0 \\ 0 & 1 & 0 \\ 0 & 0 & 2^{-1} \end{array}\right),
$$
whose Lyapunov spectrum is obviously simple. We are left to check the twisting condition.

Since $A_0$ is constant, so is its Oseledets decomposition $E^1(t)\oplus E^2(t) \oplus E^3(t)$,
with $E^1=\spann\{(1,0,0)\}$, $E^2=\spann\{(0,1,0)\}$ and $E^3=\spann\{(0,0,1)\}$.
This shows that $\cB(t)$ is just the matrix of
$$
H^u_{(\hy,t),(\hx,t)} \circ H^s_{(\hx,s),(\hy,t)}
= A_1(t)^{-1} \circ A_0
= R_3(-t)R_2(-t)R_1(-t)A_0,
$$
relative to the canonical basis of $\real^3$. %In other words,
%$$
%\begin{aligned}
%\cB(t) =
%\left(\begin{array}{ccc} \cos 2\pi t & \sin 2\pi t & 0   \\ -\sin 2\pi t & \cos 2 \pi t & 0 \\ 0 & 0 & 1\end{array}\right)
%& \left(\begin{array}{ccc} \cos 2\pi t & 0 & \sin 2\pi t \\ 0 & 1 & 0                       \\ -\sin 2\pi t & 0  & \cos 2 \pi t \end{array}\right)\\
%& \left(\begin{array}{ccc} 1 & 0 & 0                     \\ 0 & \cos 2\pi t & \sin 2\pi t   \\ 0 & -\sin 2\pi t & \cos 2 \pi t \end{array}\right)
%\left(\begin{array}{ccc} 2 & 0 & 0                       \\ 0 & 1 & 0                       \\ 0 & 0 & 2^{-1}\end{array}\right)
%\end{aligned}
%$$
It is straightforward to check that all the minors $m_{I,J}(t)$ of this matrix are analytic functions of $t$ not identically zero.
In particular, all their zeros have finite order and, consequently, the functions $\log|m_{I,J}|$ are integrable.
It follows that
$$
\lim_{n\to\infty} \frac 1n \log|m_{I,J}(f_0^n(t))| = 0 \text{ for Lebesgue almost every $t\in S^1$,}
$$
which shows that $\hF$ is twisting.
\end{example}

In many contexts of linear cocycles over \emph{hyperbolic} systems, simplicity turns out to be a
generic condition: it contains an open and dense subset of cocycles (precise statements can
be found in Viana~\cite{LLE}).  This is related to the fact that in the hyperbolic setting
pinching and twisting are just transversality conditions, and so they clearly hold on the
complement of suitable submanifolds with positive codimension.

It would be interesting to find whether this extends to the present partially hyperbolic setting.
In dimension $d=2$, simplicity is equivalent to positivity of the largest Lyapunov exponent and
that has been shown to hold for an open and dense subset of linear cocycles over partially hyperbolic
skew-products with mostly neutral center direction, by Poletti~\cite{Pol}.
In general, by Theorem~\ref{teo.B}, it would suffice to prove density of our pinching and twisting
conditions. Density of pinching corresponds, roughly, to density of simplicity for linear cocycles
over quasi-periodic transformations, a subject that does not seem to have been much investigated
beyond the 2-dimensional case (but see~\cite{DuK}). On the other hand, the arguments in
Example~\ref{ex.main} suggest that twisting is probably a rather mild requirement on the cocycle.

\subsection{Outline of the proof}

%Let us outline the overall strategy of the proof of Theorem~\ref{teo.B}.
For every $1\leq \ell<d$, we want to find complementary $\hF-$invariant measurable sections
\begin{equation}\label{eq.sections}
\xi:\hM\to \grassl \quand \eta:\hM\to \grassd
\end{equation}
such that the Lyapunov exponents of $\hF$ along $\xi$ are strictly larger than the Lyapunov exponents
along $\eta$.

The starting point is to reduce the problem to the case when the maps $\hf_{\hx}$ and the matrices
$\hA(\hx,t)$ depend on $\hx$ only through its positive part $x^u$.
This we do in Section~\ref{s.martingale2}, using the stable linear holonomies to conjugate the original
dynamics to others with these properties.
Then $\hf:\hM\to\hM$ projects to a transformation $f:M\to M$ on $M=\Sigma^+\times K$ which
is a skew-product over the one-sided shift $\sigma:\Sigma^+\to\Sigma^+$ and, similarly,
the linear cocycle $\hF:\hM\times\field^d\to\hM\times\field^d$ projects to a linear cocycle
$F:M\times\field^d\to M\times\field^d$ over the transformation $\hf$.

We also denote by $\hF$ and $F$ the actions
$$
\begin{aligned}
& \hF: \hM\times\grassl\to\hM\times\grassl \text{ and }\\
& F: M\times\grassl\to M\times\grassl
\end{aligned}
$$
induced by the two linear cocycles on the Grassmannian bundles.
Still in Section~\ref{s.martingale2}, using very classical arguments, we relate the invariant measures
of $\hf$ and $\hF$ with those of $f$ and $F$, respectively.

In Section~\ref{s.ustates} we study \emph{$u$-states}, that is, $\hF$-invariant probability measures
$\hm$ whose Rokhlin disintegrations $\{\hm_{\hx}: \hx\in\hSigma\}$ are invariant under unstable
holonomies, as well as the corresponding $F$-invariant probability measures $m$.
Here we meet the first important new difficulty arising from the fact that $\hf$ is only partially
hyperbolic.
Indeed, in the hyperbolic setting such measures $m$ are known to admit \emph{continuous}
disintegrations $\{m_x: x\in M\}$ along the fibers $\{x\}\times\grassl$ and this fact plays a
key part in the arguments of Bonatti-Viana~\cite{BoV04} and Avila-Viana~\cite{AvV1}.

In the partially hyperbolic setting, the situation is far more subtle:  the disintegration
$\{m_x: x\in\Sigma\}$ along the sets $\{x\} \times K\times\grassl$ is still continuous,
but there is no reason why this should extend to the disintegration $\{m_{x,t}: (x,t) \in M\}$
along the sets $\{(x,t)\}\times\grassl$, which is what one really needs.
The way we make up for this is by proving a kind of $L^1$-continuity:
if $(x_i)_i\to x$ in $\Sigma$ then $(m_{x_i,t})_i\to m_{x,t}$ in $L^1(\mu^c)$.
See Proposition~\ref{p.L1} for the precise statement.

This also leads to our formulating the arguments in terms of measurable sections  $K\to \grassl$
of the Grassmannian bundle, which is perhaps another significant novelty in this paper.
The properties of such sections are studied in Section~\ref{s.zeromes}.
The key result (Proposition~\ref{p.medidazero}) is that, under pinching and twisting, the graph of
every invariant Grassmannian section has zero $m_x$-measure, for every $x\in M$ and any
$u$-state $\hm$.

These results build up to Section~\ref{s.dirac}, where we prove that every $u$-state $\hm$ has an
atomic disintegration. More precisely (Theorem~\ref{t.Dirac}), there exists a measurable
section $\xi:\hM\to\grassl$ such that, given any $u$-state $\hm$ on
$\hM\times\grassl$, we have
\begin{equation}\label{eq.convDirac}
\hm_{\hx,t}=\delta_{\xi(\hx,t)}
\quad\text{for $\hmu$-almost every $(\hx,t)\in\hM$.}
\end{equation}
Thus we construct the invariant section $\xi:\hM\to\grassl$ in \eqref{eq.sections}.

To find the complementary invariant section $\eta:\hM\to\grassd$, in Section~\ref{s.orthogonal}
we apply the same procedure to the adjoint cocycle $\hF^*$, that is, the linear cocycle defined
over $\hf^{-1}:\hM\to\hM$ by the function
$$
(\hx,t) \mapsto \hA^*(x,t) = \text{ adjoint of } \hA(\hf^{-1}(\hx,t)).
$$
We check (Proposition~\ref{p.twist.adj}) that this cocycle $\hF^*$ is pinching and twisting if and only if $\hF$ is.
So, the previous arguments yield a $\hF^*$-invariant section  $\xi^{\ast}:\hM\to \grassl$ related
to the $u$-states of $\hF^*$. Then we just take $\eta=\big(\xi^{\ast}\big)^{\perp}$.

Finally, in Section~\ref{s.teoB} we check that the eccentricity, or lack of conformality, of the iterates
$\hA^n$ goes to infinity $\hmu$-almost everywhere (see Proposition~\ref{p.infinito}) and we use
this fact to deduce that every Lyapunov exponent of $\hF$ along $\xi$ is strictly larger than any of
the Lyapunov exponents of $\hF$ along $\eta$. At this stage the  arguments are again very classical.
This concludes the proof of Theorem~\ref{teo.B}.

%The two appendices contain material that seems to be folklore, but for which we could not find explicit references.
%In Appendix~\ref{a.adjoint} we check that the Lyapunov spectra of a linear cocycle and its adjoint
%coincide.
In Appendix~\ref{a.denseL1} we show that continuous maps are dense in the corresponding $L^1$ space,
whenever the target space is geodesically convex. This is probably well known,
but we could not find explicit references.

%%%%%%%%%%%%%%%%%%%%%%%%%%%%%%%%%%%%%%%%%%%%%%%%%%%%%%%%%%%%%%%%%%%%%%%%%%%%%%%%%%%%%%%%%%%%%%%%%%%%%%%%%%%%%%%%%%%%%%%%%%%%%%%%%%%%%%%%%%%%%%%%%%%%%%%%%%%%%%%%%%%%%%%%%%%%%%%%%%%%

\section{Disintegration along center leaves}\label{s.martingale1}

Let us start by fixing some terminology. We use $\id_Y$ to denote the identity transformation in a set $Y$.
Similarly, $\dist_Y$ will always denote the distance in a metric space $Y$.

Let $\hSigma=\Sigma^- \times \Sigma^+$, where $\Sigma^-=X^{\integer_{<0}}$ and $\Sigma^+=X^{\integer_{\geq 0}}$.
Thus we write every $\hx\in\hSigma$ as $(x^s, x^u)$ with $x^s\in\Sigma^{-}$ and $x^u\in\Sigma^{+}$.
For simplicity, we also write $\Sigma=\Sigma^+$ and $x=x^u$.
Let $P:\hSigma \to \Sigma$ be the canonical projection given by $P(\hx)=x$ and let $\sigma:\Sigma \to \Sigma$ be the one-sided shift.
Given points $\hx\in\hSigma$ and $\hq\in\hM$, denote
$$
x_n=P(\hsigma^{-n}(\hx)) \quand q_n = (P\times\id_K)(\hf^{-n}(\hq))
$$
for each $n\in\natural$.

We also consider $M=\Sigma \times K$ and the projection $\mu=(P\times\id_K)_*\hmu$.
In other words,
\begin{equation}\label{eq.mu}
\mu = \rho(x,t) \, \mu^u \times \mu^c
\quad\text{where}\quad
\rho(x,t) = \int\hrho(x^s,x,t) \, d\mu^s(x^s).
\end{equation}
Similarly to \eqref{eq.hatvarrho}, for each $x\in M$ let
\begin{equation}\label{eq.varrho}
\varrho(x,\cdot) = \frac{\rho(x,\cdot)}{\int \rho(x,t) \, d\mu^{c}(t)}
\quad\text{and}\quad
\mu^c_{x}=\varrho(x,\cdot) \mu^c.
\end{equation}
Note that $\{\mu^{c}_{x}: x\in \Sigma\}$ is a continuous disintegration of $\mu$
with respect to the partition $\cP=\lbrace \lbrace x\rbrace \times K: x\in \Sigma \rbrace$.

In this section we derive some useful properties of these disintegrations
\eqref{eq.hatvarrho} and \eqref{eq.varrho}.
For this, we assume that the base dynamics is such that each $\hf_{\hx}:K\to K$
along the center direction depends only on $x=P(\hx)$.
This is no restriction in our setting, as we will see in Section~\ref{s.martingale2}.
Then there exists $f:M\to M$ of the form $f(x,t)=(\sigma(x),f_x(t))$ such that
\begin{equation*}
(P \times \id_K) \circ \hf = f \circ (P \times \id_K)
\end{equation*}

\subsection{Holonomy invariance}\label{ss.jacobians}

We call the \emph{extremal center Lyapunov exponents} of $\hf$ the limits
$$
\lambda^{c+}(\hx,t)=\lim_n \frac{1}{n}\log \big\| D\hf^{n}_{\hx}(t)\big\|
\quand
\lambda^{c-}(\hx,t)=\lim_n -\frac{1}{n}\log \big\| D\hf^{n}_{\hx}(t)^{-1}\big\|.
$$
for $\hxt\in \hSigma\times K$. The Oseledets theorem \cite{Ose68} ensures that
these numbers are well defined at $\hmu$-almost every point. In our situation,
since the maps $\hf^{n}_{\hx}$ have uniformly bounded derivatives:

\begin{lemma}\label{l.lyapunovcentral}
$\lambda^{c+}=\lambda^{c-}=0$.
\end{lemma}

\begin{remark}\label{equicontinuitysuffices}
When the maps $\hf^{n}_{\hx}$ are $C^{1+\epsilon}$, equicontinuity alone suffices to get the conclusion of
Lemma~\ref{l.lyapunovcentral}. This can be shown using Pesin theory, as follows.

Suppose that $\lambda^{c+}>0$. Then we have a Pesin unstable manifold defined $\hmu$-almost everywhere.
This implies that there exist $\hx\in \hSigma$ and $t\neq s \in K$ such that
\begin{equation*}
\dist_K\big(\hf^{-n}_{\hx}(t),\hf^{-n}_{\hx}(s)\big)\to 0.
\end{equation*}
Then, given points $t$ and $s$ in the unstable manifold and given any $\delta>0$, there exists $n$ such
that $\dist_K(\hf^{-n}_{\hx}(t),\hf^{-n}_{\hx}(s))<\delta$. This implies that the family is not equicontinuous.
The proof for $\lambda^{c-}$ is analogous.
\end{remark}

Let $\pi_1:\hM\to \hSigma$ be the projection $\pi_1(\hx,t)=\hx$, recall that we also assume that $\hf:\hM\to \hM$
admits s-holonomies and u-holonomies and $\hmu$ has
partial product structure. That implies that $(\pi_1)_*\hmu$ has local product
structure in the sense of \cite{Extremal}.
%
% \begin{proposition}\label{p.inv.princip}
% The measure $\hmu$ admits a disintegration $\lbrace \hmu^c_{\hx}:\hx \in \hSigma\rbrace$ which is $s$-invariant,
% $u$-invariant and continuous: the conditional probabilities $\hmu^c_x$ vary continuously with $\hx$ on
% the support of $\tilde{\mu}$.
% \end{proposition}
%
%\begin{lemma}\label{l.inv.muc}
% The disintegration $\{\mu^{c}_{x}: x\in\Sigma\}$ is $f$-invariant, in the sense that
% ${f_{x}}_{\ast}\mu^c_x=\mu_{\sigma(x)}$ for every $x\in\Sigma$.
%\end{lemma}
%
%\begin{proof}
% For every $x\in \Sigma$ we have
% \begin{equation*}
%  \left(f_{x}^{-1}\right)_{\ast}\mu^{c}_{\sigma(x)}=J(x,\cdot)\mu^{c}_{x}+\eta_{x}
% \end{equation*}
%where $\eta_x$ is singular respect to $\mu^{c}_{\sigma(x)}$. Define $h=\int -\log J d\mu$.
%By \cite[Proposition~3.1]{Extremal} we have that zero Lyapunov exponent implies $h=0$ this
%implies $J=1$ at $\mu$-almost every point.
%Then $\left(f_{x}^{-1}\right)_{\ast}\mu^{c}_{\sigma(x)}=\mu^{c}_{x}$, which is the same as
%${f_{x}}_{\ast}\mu^c_x=\mu_{\sigma(x)}$, almost everywhere.
%The continuity of the disintegration implies that the equality is true everywhere.
%\end{proof}

\begin{lemma}
The map $\hx \mapsto \hmu^{c}_{\hx}$ is continuous, and the disintegration $\{\hmu^c_{\hx}:\hx\in\hSigma\}$
is both $u$-invariant and $s$-invariant:
  \begin{itemize}
  \item[(a)] $\left(h^{u}_{\hx,\hy}\right)_{\ast}\hmu^{c}_{\hx}=\hmu^c_{\hy}$ for every $\hx\in W^{u}(\hy)$ and
  \item[(b)] $\left(h^{s}_{\hx,\hz}\right)_{\ast}\hmu^{c}_{\hx}=\hmu^c_{\hz}$ for every $\hx\in W^{s}(\hz)$.
  \end{itemize}
\end{lemma}

\begin{proof}
By Theorem~D in \cite{Extremal}, there exists a disintegration $\{\tmu^c_{\hx}:\hx\in\hSigma\}$ which is
continuous, $u$-invariant and $s$-invariant. By essential uniqueness, $\tmu^c_{\hx} = \hmu^c_{\hx}$ for $\hmu$-almost
every $x$. Since both disintegrations are continuous, it follows that they coincide, and so
$\{\hmu^c_{\hx}:\hx\in\hSigma\}$ is continuous, $u$-invariant and $s$-invariant, as claimed.
\end{proof}

\begin{corollary}\label{c.muc}
$\mu^{c}_{x}=\hmu^{c}_{\hx}$ for every $\hx\in\hSigma$, where $x=P(\hx)$.
\end{corollary}

\begin{proof}
The assumption that $\hf_{\hx}$ only depends on $x=P(\hx)$ implies that $h^s_{\hx,\hy}=\id_K$ for every $\hx$
and $\hy$ in the same stable set. By the previous lemma, this implies that $\hmu^c_{\hx}=\hmu^c_{\hy}$
whenever $\hx$ and $\hy$ are in the same stable set. Then,
$$
\mu^c_x=\int \hmu^c_{\hy} \, d\hmu^s_{x}(\hy) = \hmu^c_{\hx}
$$
for any $\hx$ with $P(\hx)=x$.
\end{proof}

We also have
\begin{corollary}\label{c.inv.muc}
 The disintegration $\{\mu^{c}_{x}: x\in\Sigma\}$ is $f$-invariant, in the sense that
 $\big(f_{x}\big)_{\ast}\mu^c_x=\mu_{\sigma(x)}$ for every $x\in\Sigma$.
\end{corollary}

\begin{proof}
We have that $(\hf_{\hx})_* \hmu^c_{\hx}=\hmu^c_{\hsigma(\hx)}$ for $\hmu$-almost every $\hx$,
because $\hmu$ is $\hf$-invariant. Since $\hx\mapsto\hmu^c_{\hx}$ is continuous,
the identity extends to \emph{every} $\hx\in\hSigma$. By Corollary~\ref{c.muc} this implies that
$(f_{x})_* \mu^c_{x}=\mu^c_{\sigma(x)}$ for every $x\in\Sigma$.
\end{proof}

\subsection{Jacobians}\label{ss.jacobians2}

Denote $\hnu=(\pi_1)_{\ast}\hmu$ and $\nu=(\pi_1)_{\ast}\mu$ where $\pi_1$ denotes
both canonical projections $\hM\to\hSigma$ and $M\to\Sigma$. Recall the functions
$\hvro$ and $\varrho$ defined in \eqref{eq.hatvarrho} and \eqref{eq.varrho}.
Note that $\{\hvro(x) \hmu^c: \hx\in\hSigma\}$ is a disintegration of $\hmu$
with respect to the partition $\{\pi_1^{-1}(\hx):\hx\in\hSigma\}$ of $\hM$ and
$\{\varrho(x) \mu^c: x\in\Sigma\}$ is a disintegration of $\mu$
with respect to the partition $\{\pi_1^{-1}(x):x\in\Sigma\}$ of $\hM$.

Given a measurable map $g:N\to N$ and a measure $\eta$ on $N$ we call \emph{Jacobian}
of $g$ with respect to $\eta$ the essentially unique function $J_{\eta}g:N \to \real$
such that $\eta(g(C))=\int_C J_\eta g \, d\eta$ for every measurable set $C\subset N$
where $g$ is invertible. This is well defined whenever $N$ can be covered with countably
many domains of invertibility of $g$. See \cite[Section~9.7]{FET} for a detailed discussion.

\begin{remark}\label{r.jacobianos_muc}
Let $Jf^{j}_{x}:K\to \real$ be the Jacobian of $f^{j}_{x}$ with respect to $\mu^c$.
Using the observation that $\{\varrho(x) \mu^c: x\in\Sigma\}$ is a disintegration
of the $f$-invariant measure $\mu$, one easily gets that
$$
Jf^j_{x}(t)=\frac{\varrho(f^j(x,t))}{\varrho(x,t)}.
$$
In particular, these Jacobians are uniformly bounded from above and below.
Analogously, the fact that $\{\hmu^c_{\hx}:\hx\in\hSigma\}$ is invariant
under stable and unstable holonomies ensures that the Jacobians
$Jh^{\ast}_{\hx,\hy}$ of those holonomies  with respect to $\mu^c$ are
uniformly bounded from above and below.
\end{remark}

\begin{lemma}\label{l.jacobian1}
$J_{\mu}f^k(x,t)=J_{\nu}\sigma^k(x)$ for every $(x,t)\in M$ and $k\ge 1$.
\end{lemma}

\begin{proof}
Fix some $k$-cylinder $I=\left[ 0; x_0,\dots,x_{k-1}\right]$ and let $J\subset I$ and $C\subset K$. Noting that $\sigma^k \mid I$
is injective,
$$
\begin{aligned}
\mu\big(f^k(J\times C)\big)
& = \int_{y\in \sigma^k(J)} \mu^c_{y}(f^k_{z(y)}(C)) \, d\nu(y) \\
& = \int_{z\in J} J_\nu\sigma^k(z) \mu^c_{\sigma^k(z)}(f^k_{z}(C)) \, d\nu(z),
\end{aligned}
$$
where $z(y)$ is the unique point in $\sigma^{-k}(y) \cap I$ and we use the
change of variables $z=z(y)$. Using Lemma~\ref{c.inv.muc}, it follows that
$$
\mu\big(f^k(J\times C)\big)
= \int_J J_\nu\sigma^k(z) \mu^c_z(C) \, d\nu(z)
= \int_{J\times C} (J_{\nu}\sigma^k \circ \pi_1) \, d\mu,
$$
which concludes the proof.
\end{proof}

Now we find the Jacobian of $\sigma^k$:

\begin{lemma}\label{l.jacobian2}
$J_{\nu}\sigma^k(x)=1/\nu^s_{\sigma^k(x)}(I)$, where $I=[-k; x_0,\dots,x_{k-1}]$.
Consequently, the Jacobians $J_\mu f^k=J_\nu \sigma^k \circ \pi_1$ are continuous
and bounded from away zero and infinity on every $k$-cylinder.
\end{lemma}
\begin{proof}
Given $x\in\Sigma$ and $n\ge 1$, let $J_n=[x_0,\dots,x_{n}]$ be the $n$-cylinder
that contains $x$. Then,
$$
J_{\nu}\sigma^k(x)=\lim_{n \to \infty} \frac{\nu(\sigma^k(J_n))}{\nu(J_n)}
$$
Since $\hnu$ is invariant under $\hsigma$,
$$
\nu(J_n)
= \hnu(\Sigma^-\times J_n)
= \hnu(\hsigma^k(\Sigma^-\times J_n))
= \int_{y\in \sigma^k(J_n)} \nu^s_y(I)d\nu(y).
$$
It follows that
$$
\frac{1}{J_{\nu}\sigma^k(x)}
= \lim_{n\to\infty} \frac{\int_{y\in \sigma^k(J_n)} \nu^s_y(I)d\nu(y)}{\nu(\sigma^k(J_n))}
= \nu^s_{\sigma^k(x)}(I).
$$
This proves the first part of the conclusion. The second part is a consequence,
since the local product structure implies that $x \mapsto \nu^s_{\sigma^k(x)}(I)$ is continuous
for every cylinder $I$.
\end{proof}

\section{Convergence of conditional measures}\label{s.martingale2}

For each $1 \le l <d$, the linear cocycle $\hF:\hM\times \field^d\to \hM\times \field^d$
induces a projective cocycle $\hF:\hM\times \grassl\to \hM\times \grassl$ through
\begin{equation}
\hF(\hq,v)=(\hf(\hq),\hA(\hq)v).
\end{equation}
Let $\mu=\left(P\times \id_K\right)_{\ast}\hmu$ and, for any Borel probability measure $\hm$
on $\hM\times \grassl$,
\begin{equation}\label{eq.mhm}
m=\left(P\times \id_K\times \id_{\grassl}\right)_{\ast}\hm.
\end{equation}
We will be especially interested in the case when $\hm$ is a $\hF$-invariant probability measure
that projects down to $\hmu$ under the canonical projection $\pi:\hM\times \grassl \to \hM$ on
the first coordinate.

\subsection{Reduction to the one-sided case}\label{ss.reduction}

Our first step is to show that, up to conjugating the cocycle in a suitable way, we may suppose that:
\begin{enumerate}
\item[(A)] the base dynamics $\hf_{\hx}$ along the center direction depends only on $x$;
\item[(B)] the matrix $\hA(\hx,t)$ depends only on $(x,t)$.
\end{enumerate}
Next, let us explain how such a conjugacy may be defined using the stable linear holonomies.

Let $x^s\in\Sigma^-$ be fixed. For any $\hy\in\hSigma$, let $\phi(\hy) = (x^s,y)$ and then define
\begin{equation*}
h(\hy,t)=\big(\hy,h^{s}_{\varphi(\hy),\hy}(t)\big).
\end{equation*}
Then $\tilde{f}=h^{-1}\circ f\circ h$ is given by
\begin{equation*}
\tilde{f}(\hy,t)=\big(\hsigma(\hy), \tilde{f}_{\hy}(t)\big),
\quad\text{with } \tilde{f}_{\hy}(t) = h^{s}_{\hsigma(\phi(\hy)), \phi(\hsigma(\hy))}f_{\phi(\hy)}(t).
\end{equation*}
Notice that $\tilde{f}_{\hy}$ does depend only on $y$ (because $\phi$ does).

Assume that (A) is satisfied. Define $\hphi(\hy,t)=(\phi(\hy),t)$ and then let
\begin{equation*}
H(\hy,t) = H^s_{\hphi(\hy,t),(\hy,t)}
\end{equation*}
Define $\tilde{A}(\hy,t) = H(\hf(\hy,t))^{-1} \circ A(\hy,t) \circ H(\hy,t)$. Then
\begin{equation*}
\tilde{A}(\hy,t)
= H^s_{\hf(\hphi(\hy,t)),\hphi(\hf(\hy,t))} \circ \hA(\hphi(\hy,t)),
\end{equation*}
which only depends on $(y,t)$. Clearly, this procedure does not affect the
Lyapunov exponents.

Observe that this new cocycle is continuous but not necessarily H\"older continuous,
since the definition involves the holonomies which are not assumed to be  H\"older
continuous. However, $\tilde{A}$ does clearly have well defined holonomies,
which is what we actually need for the following.

From now on, we assume that both (A) and (B) are satisfied. Then, there exist
\begin{equation*}
f:M\to M, \  f(x,t)=(\sigma(x),f_x(t))
\quand A: M \to \GL
\end{equation*}
such that
\begin{equation*}
(P \times \id_K) \circ \hf = f \circ (P \times \id_K)
\quand \hA = A \circ (P \times \id_K).
\end{equation*}
Consequently, the map
\begin{equation*}
F:M\times \grassl\to M\times \grassl, \quad F(p,V)=(f(p),A(p)V)
\end{equation*}
satisfies
\begin{equation*}
\left(P\times \id_K\times \id_{\grassl}\right)\circ \hF=F \circ \left(P\times \id_K \times \id_{\grassl}\right).
\end{equation*}

The following well known basic fact will be used to characterize the $F$-invariant probability measures:

\begin{proposition}\label{p.invariante}
Let $(N,\fB,\eta)$ be a Lebesgue probability space and $g:N\to N$ be a measurable map that
preserves $\eta$. Let $\{\eta_{y}: y\in N\}$ be the disintegration of $\eta$ with respect
to the partition into pre-images $\cP=\{g^{-1}(y): y \in N\}$. Let
\begin{equation*}
G:N\times L \to N\times L,
\quad G(x,v)=(g(x),G_x(v))
\end{equation*}
be a measurable skew-product over $g$ and, given any probability measure $m$ on $N\times L$
that projects down to $\eta$, let $\{m_x: x\in N\}$ be its disintegration with respect to the
partition into vertical fibers $\{x\} \times L$, $x\in N$. Then $m$ is invariant under $G$ if and only if
\begin{equation*}
m_{x}=\int (G_{z})_{\ast} m_{z} \, d\eta_{x}(z)
\quad\text{for $\eta$-almost every $x\in N$.}
\end{equation*}
\end{proposition}
As an immediate consequence, we get:

\begin{corollary}\label{c.invariante}
In the conditions of Proposition~\ref{p.invariante}, if $g$ is invertible then $m$ is invariant under
$G$ if and only if $m_x = (G_{g^{-1}(x)})_* m_{g^{-1}(x)}$ for $\eta$-almost every $x\in N$.
\end{corollary}

\begin{proof}
Each $\eta_y$ must coincide with the Dirac mass at $g^{-1}(y)$.
\end{proof}

\subsection{Lifting of measures}\label{ss.lifting}

The next proposition shows that every $\hF$-invariant measure $\hm$ that projects down to $\hm$
may be recovered from the corresponding $F$-invariant measure $m$, defined by \eqref{eq.mhm}.
Recall that we write $q_n=(P\times \id_K)(\hf^{-n}(\hq))$ for each $\hq\in \hM$ and $n\ge 0$.

The following proposition is borrowed from \cite[Section~3]{BoV04}:

\begin{proposition}\label{p.limmesa}
Take $\hm$ to be $\hF$-invariant. Then, for $\hmu$-almost every $\hq \in \hM$, the sequence
$(A^{n}(q_n)_{\ast} m_{q_n})_n$ converges to $\hm_{\hq}$ in the weak$^*$ topology.

Moreover, for any $k\geq 1$ and any choice of points $y_{n,k}$ such that $f^k(y_{n,k})=q_n$ and $\{y_{n,k}:n\ge 0\}$
is contained in some set of the form $I_k\times K$, where $I_k$ is a $k$-cylinder in $\Sigma$,
\begin{equation}\label{eq.limitnk}
\lim_{n\to \infty} A^{n}(q_n)_{\ast} m_{q_n}=
\lim_{n\to \infty} A^{n+k}(y_{n,k})_{\ast} m_{y_{n,k}}.
\end{equation}
\end{proposition}

Adapting the proof to the present setting is straightforward, so we only make a couple of observations.
Firstly, the proof uses the fact that the Jacobians $J_\mu f^k(y_{n,k})$ are bounded away from zero.
In \cite{BoV04} that is automatic, by continuity and compactness of the ambient space,
whereas in the present setting it is ensured by the assumption that $y_{n,k}\in I_k\times K$ for every $n\geq 0$,
through Lemma~\ref{l.jacobian2}. Finally, note that the identity \eqref{eq.limitnk} corresponds to the
conclusion (b) in \cite[Proposition~3.1]{BoV04}.

\section{Properties of $u$-states}\label{s.ustates}

A probability measure $\hm$ on $\hM\times \grassl$ is called a \emph{$u$-state} of $\hF$ if
there exists a disintegration $\{\hm_{\hq}: \hq\in\hM\}$ along the partition
$\big\{\{\hq\} \times \grassl: \hq\in\hM\big\}$ which is invariant under unstable linear holonomy:
there exists $\tilde{M}\subset \hM$ with $\hmu(\tilde{M})=1$ such that
\begin{equation}\label{eq.ustate}
\hm_{\hq}={H^u_{\hp,\hq}}_{\ast}\hm_{\hp}
\text{ for every $\hp, \hq \in \tilde{M}$ with $\hq \in W^{uu}_{\loc}(\hp)$.}
\end{equation}
Let $\pi:\hM\times \grassl \to \hM$
be the canonical projection.

\begin{proposition}
There is some $\hF$-invariant $u$-state $\hm$ that projects down to $\hmu$ under $\pi$.
\end{proposition}

This is analogous to \cite[Proposition~4.2]{AvV1}. In very brief terms, the idea is to fix some $\hx\in \hSigma$
and to construct a homeomorphism between the space of measures in $\lbrace \hx\rbrace\times W^{s}_{\loc}(\hx) \times K$
that project down to $\mu^s$ and the space of all $u$-states. Using that the former is weak$^*$ compact, we get that
the space of $u$-states measures is also compact. Moreover, it is $\hF_*$-invariant.
That ensures that any accumulation point of ${n}^{-1}\sum_{j=0}^{n-1}\hF^j_*\hm$ is also a $u$-state.

In the remainder of this section, $\hm$ denotes any $\hF$-invariant $u$-state that projects down to $\hmu$ under $\pi$,
and $\{\hm_{\hq}: \hq\in\hM\}$ is taken to be a disintegration as in \eqref{eq.ustate}.

\subsection{Bounded distortion}

Let $\pi_1:\hM\to \hSigma$ be the canonical projection $\pi_1(\hx,t)=\hx$ and denote $\hmuSigma={\pi_1}_{\ast}\hmu$.
Equivalently,
$$
\hnu(E) = \int_{E \times K} \hrho(x^s,x,t) \, d\mu^s(x^s) \, d\mu^u(x) \, d\mu^c(t)
$$
for any measurable set $E\subset\hSigma$. For each $x\in\Sigma$, define $\hmuSigma_x$
to be the normalization of
$$
\mu^s \int\hrho(\cdot,x,t)\, d\mu^c(t).
$$
Then $\{\hmuSigma_x: x\in\Sigma\}$ is a continuous disintegration of $\hmuSigma$
with respect to the partition into local stable sets $W^s_{\loc}(\hx)$.

The measure $\hmuSigma$ satisfies the properties of local product structure, boundedness and continuity in \cite[Section~1.2]{AvV1}.
In what follows, we recall a few results about this type of measures that we will use later.
For each $x^u \in \Sigma^+$ and $k\geq 1$ let the \emph{backward average} measure $\mu^u_{k,x^u}$ of the map $\sigma$ be defined by
\begin{equation*}
 \mu^u_{k,x^u}=\sum_{\sigma^k(z)=x^u}\frac{1}{J\sigma^k(z)}\delta_z
\end{equation*}
where $J\sigma^k:\Sigma^+\to \real$ is the Jacobian of $\mu^u$ with respect to $\sigma^k$.

\begin{lemma}[Lemma~2.6 in \cite{AvV1}]\label{l.todos1}
For any cylinder  $I^u=[0;\iota_0,\dots,\iota_{k-1}]\subset \Sigma^+$ and $z^u\in I^u$,
\begin{equation*}
\hsigma^k_{\ast}\hmuSigma_{z^u}=J\sigma^k (z^u) ( \hmuSigma_{\sigma^k(z^u)} \mid I^s)
\end{equation*}
where $\{\hmuSigma_{z^u}: z^u\in \Sigma^+\}$ is the disintegration of $\hmuSigma$ with respect to the partition
$\{\Sigma^{-}\times \lbrace z^u\rbrace : z^u \in \Sigma^+\}$.
\end{lemma}

\begin{lemma}[Lemma~2.7 in \cite{AvV1}]
For every cylinder $\left[J\right]\subset \Sigma^+$ and $x^u\in \Sigma^+$,
$$
\begin{aligned}
\kappa \mu^u(\left[J\right])
& \geq \limsup_n \frac{1}{n}\sum_{k=0}^{n-1}\mu^u_{k,x^u}(\left[J\right])\\
& \geq \liminf_n \frac{1}{n}\sum_{k=0}^{n-1}\mu^u_{k,x^u}(\left[J\right])
\geq \frac{1}{\kappa}\mu^u(\left[J\right]),
\end{aligned}
$$
where $\kappa$ is the bound given in \eqref{eq.boundedness}.
\end{lemma}

As a direct consequence, for every cylinder $\left[J\right]\subset \Sigma^+$ and $x^u\in \Sigma^+$,
\begin{equation}\label{eq.distortion}
 \limsup_k \mu^u_{k,x^u}(\left[J\right])\geq \frac{1}{\kappa}\mu^u(\left[J\right]).
\end{equation}

%\subsection{Estimating the Jacobians}
\subsection{$L^1$-continuity of conditional probabilities}\label{s.L1continuity}

For every $x\in \Sigma$ let
\begin{equation*}
F_x:K\times \grassl\to K\times\grassl, \quad F_x(t,V)=(f_x(t),A(x,t)V)
\end{equation*}
and for every $\hx,\hy \in \hSigma$ in the same unstable set let
$$
H_{\hx,\hy}:K\times \grassl\to K\times \grassl
$$
be defined by
$$
H_{\hx,\hy}(t,V)=(h^u_{\hx,\hy}(t),H^u_{(\hx,t)(\hy,h^u_{\hx,\hy}(t))}V).
$$
Observe that $F_{\hy}\circ H_{\hx,\hy}=H_{\hsigma(\hx),\hsigma(\hy)}\circ F_{\hx}$.
Now define $\lbrace \hm_{\hx}:\hx\in \hSigma \rbrace$ as
$$
\hm_{\hx}=\int \hm_{\hx,t} \,d\hmu^c_{\hx}(t).
$$
Observe that for any $\varphi:K\times \grassl \to \real$,

$$
\begin{aligned}
\int \varphi d\big({H_{\hx,\hy}}_* \hm_{\hx}\big)
&=\int \varphi \big(h^u_{\hx,\hy}(t),H^u_{(\hx,t)(\hy,h^u_{\hx,\hy}(t))}V\big)d\,\hm_{\hx,t}(V)d\,\hmu^c_{\hx}(t)\\
&=\int \varphi \big(h^u_{\hx,\hy}(t),V\big)d\,\big({H^u_{(\hx,t)(\hy,h^u_{\hx,\hy}(t))}}_*\hm_{\hx,t}\big)(V)d\,\hmu^c_{\hx}(t)\\
&=\int \varphi \big(h^u_{\hx,\hy}(t),V\big)d\,\hm_{\hy,h^u_{\hx,\hy}(t)}(V)d\,\hmu^c_{\hx}(t)\\
&=\int \varphi \big(t,V\big)d\,\hm_{\hy,t}(V)d\,({h^u_{\hx,\hy}}_*\hmu^c_{\hx})(t)\\
&=\int \varphi \big(t,V\big)d\,\hm_{\hy,t}(V)d\,\hmu^c_{\hy}(t),
\end{aligned}
$$
because $\hm$ is a $u$-state and $\lbrace \hmu^c_{\hx}:\hx\in\hSigma\rbrace$ is $h^u$-invariant.
It is also easy to see that $\lbrace \hm_{\hx}:\hx\in \hSigma \rbrace$ is a disintegration of
$\hm$ with respect to the partition $\lbrace \hx\times K \times \grassl:\hx\in\hSigma \rbrace$.

The main point with the next corollary is that the conclusion is for \emph{every} $x\in\Sigma$.

\begin{corollary}\label{c.invcont2}
For every $u$-state $\hm$, there exists a disintegration $\{\hm_{\hx} : \hx\in\hSigma\}$ of $\hm$
such that
$$
\hm_{\hsigma^n(\hx)} = {F_x^n}_* \hm_{\hx}
$$
for every $n \geq 1$, every $x\in\Sigma$, and $\hmuSigma_x$-almost every $\hx \in W^s_{\loc}(x)$.
\end{corollary}

\begin{proof}
Up to modifying a $0$ measure subset we can assume that de disintegration $\{\hm_{\hx} : \hx\in\hSigma\}$ has the property that ${ H_{\hx,\hy} }_\ast \hm_{\hx}=\hm_{\hy}$ for every $\hx\in W^u_{loc}(\hy)$.
Since $\hm$ is $\hF$-invariant, the equality is true for all $n\geq
1$ and $\hmuSigma$-almost all $\hz\in\hSigma$ or, equivalently, for
$\hmuSigma_z$-almost every $\hz\in W^s_{\loc}(z)$ and $\muSigma$-almost every
$z\in\Sigma$. Consider an arbitrary point $x\in\Sigma$. Since $\muSigma$
is positive on open sets, $x$ may be approximated by points $z$ such
that
$$
\hm_{\hsigma^n(\hz)} = {F^n_z}_* \hm_{\hz}
$$
for every $n\geq 1$ and $\hmu_z$-almost every $\hz\in W^s_{\loc}(z)$. Since the conditional
probabilities of $\hm$ are invariant under unstable linear holonomies, it follows that
$$
\hm_{\hsigma^n(\hx)}
% = (H_{\hsigma^n(z),\hsigma^n(x)})_* \hm_{\hsigma^n(\hz)}
 = (H_{\hsigma^n(z),\hsigma^n(x)})_* {F_{z}^n}_* \hm_{\hz}
 = {F^n_x}_* (H_{\hz,\hx})_* \hm_{\hz}
 = {F^n_x}_* \hm_{\hx}$$
for $\hmu_z$-almost every $\hz\in W^s_{\loc}(z)$, where $\hx$ is the unique point in
$W^s_{\loc}(x)\cap W^u_{\loc}(\hz)$. Since the measures $\hmu_x$ and $\hmu_z$ are equivalent,
this is the same as saying that the last equality holds for $\hmu_x$-almost every
$\hx\in W^s_{\loc}(x)$, as claimed.
\end{proof}

%\subsection{$L^1$-continuity of conditional probabilities}\label{s.L1continuity}

Recall that
\begin{equation*}
\mu=\left(P\times \id_K\right)_{\ast}\hmu \quand m=\left(P\times \id_K\times \id_{\grassl}\right)_{\ast}\hm.
\end{equation*}
Let $\{m_{x} : x \in \Sigma\}$ and $\{m_{x,t}:(x,t)\in M\}$ be disintegrations of $m$ with respect to
the partitions $\lbrace \lbrace x\rbrace \times K\times \grassl,x\in \Sigma \rbrace$ and
$\lbrace \lbrace (x, t)\rbrace \times \grassl,(x,t)\in \Sigma\times K \rbrace$, respectively.
Thus each $m_{x}$ is a probability measure on $K\times \grassl$ and each $m_{x,t}$
is a probability measure on $\grassl$.

It is easy to check that $x \mapsto m_x$ may be chosen to be continuous with respect to the weak$^*$ topology
(see Corollary~\ref{continuidad}). The corresponding statement for $x \mapsto m_{x,t}$ is false, in general.
However, the main goal in this section is to show that the family $\{m_{x,t}: (x,t)\in M\}$ does have some
continuity property:

\begin{proposition}\label{p.conv}
Let $(x_n)_n$ be a sequence in $\Sigma$ converging to some $x\in\Sigma$.
Then there exists a sub-sequence $(x_{n_{k}})_k$ such that
$$
m_{x_{n_{k}},t} \to m_{x,t} \text{ as } k\to\infty
$$
in the weak$^*$ topology, for $\mu^c$-almost every $t \in K$.
\end{proposition}

We will deduce this from a somewhat stronger $L^1$-continuity result, whose precise statement
will be given in Proposition~\ref{p.L1}. The key ingredient in the proofs is a result about maps
on geodesically convex metric spaces that we are going to state in Lemma~\ref{l.keylemma} and
which will also be useful at latter stages of our arguments.

A metric space $N$ is \emph{geodesically convex} if there exists $\tau\ge 1$
such that for every $u,v\in N$ there exist a continuous path $\lambda:\left[0,1\right]\to N$
with $\lambda(0)=u$, $\lambda(1)=v$ and
\begin{equation}\label{eq.geodesicallyconvex}
\dist_N(\lambda(t),\lambda(s))\leq\tau\dist_N(u,v)
\text{ for every $s,t\in [0,1]$.}
\end{equation}
Geodesically convex metric spaces include convex subsets of a Banach space,
path connected compact metric spaces and complete connected Riemannian manifolds,
among other examples. The spaces of maps with values in a geodesically convex
metric space are analyzed in Appendix~\ref{a.denseL1}.

\begin{lemma}\label{l.keylemma}
Let $L$ be a geodesically convex metric space and take $(K,\fB_K,\mu_K)$ to be a probability
space such that $K$ is a normal topological space, $\fB_K$ is the Borel $\sigma$-algebra
of $K$ and $\mu_K$ is a regular measure.

Let $H_{j,t}:L\to L$ and $h_j:K\to K$, with $j\in \natural$ and $t\in K$, be such that
$$
\big(H_{j,t}(x)\big)_j \to x \quand \big(h_j(t)\big)_j \to t,
$$
uniformly in $t\in K$ and $x\in L$ and, moreover, the Jacobian $J h_j(t)$ of each $h_j$ with
respect to $\mu_K$ is uniformly bounded. Then
$$
\lim_j \int \dist_L \big(\psi(t), H_{j,t} \circ \psi \circ h_j(t)\big) \, d\mu_K(t) = 0
$$
for every bounded measurable map $\psi:K\to L$.
\end{lemma}

\begin{proof}
Take $j\in\natural$ to be sufficiently large that $d_L(H_{j,t}(x),x)<{\epsilon}/{4}$ for
every $t$ and $x$. Then,
\begin{equation*}
\begin{aligned}
\int \dist_L (\psi, H_{j,t} \circ \psi \circ h_j) \, d\mu_K
& \leq \int \big(\dist_L (\psi,\psi \circ h_j) \\
& \hspace{1cm} + \dist_L (\psi \circ h_j,H_{j,t} \circ \psi \circ h_j)\big) \, d\mu_K\\
& \leq \int \dist_L (\psi,\psi \circ h_j)\, d\mu_K + \frac{\epsilon}{4}.
\end{aligned}
\end{equation*}
Let $C>1$ be a uniform bound for $Jh_j(t)$. By Proposition~\ref{p.denseL1},
given $\epsilon>0$ there exists a continuous map $\tilde{\psi}:K\to L$ such that
$$
\int \dist_L(\tilde{\psi},\psi) \, d\mu_K<\frac{\epsilon}{4C}.
$$
Then, by change of variables,
$$
\int \dist_L (\tilde{\psi} \circ h_j,\psi \circ h_j ) \, d\mu_K
\le C \int \dist_L(\tilde{\psi},\psi) \, d\mu_K
< \frac{\epsilon}{4}.
$$
Then
\begin{equation*}
\begin{aligned}
& \int \dist_L (\psi,\psi \circ h_j)\, d\mu_K \\
& \hspace{1cm} \leq \int \big( \dist_L (\psi,\tilde{\psi}) + \dist_L(\tilde{\psi},\tilde{\psi} \circ h_j ) + \dist_L (\tilde{\psi} \circ h_j,\psi \circ h_j ) \big)\, d\mu_K\\
& \hspace{1cm} \leq \int \dist_L(\tilde{\psi},\tilde{\psi} \circ h_j ) \, d\mu_K + \frac{\epsilon}{2}.
\end{aligned}
\end{equation*}
By the continuity of $\tilde{\psi}$, increasing $j$ if necessary,
$$
\dist_L(\tilde{\psi}(t),\tilde{\psi} \circ h_j(t))<\frac{\epsilon}{4}
\text{ for every } t \in K.
$$
The conclusion follows from these inequalities.
\end{proof}

\begin{lemma}\label{l.abs.cont.limit}
Let $(x_n)_n$ be a sequence in $\Sigma$ converging to some $x \in \Sigma$ and $(j_n)_n$ be a sequence
of integer numbers such that $z_n=\sigma^{-j_n}(x_n)$ converges to some $z\in\Sigma$ and
$(f^{j_n}_{z_n})_n$ converges uniformly to some $g:K\to K$. Then $g$ is absolutely continuous with
respect to $\mu^c$ with bounded Jacobian. Moreover, the Jacobians of $f^{j_n}_{z_n}$  with respect to
$\mu^c$ are uniformly bounded.
\end{lemma}

\begin{proof}
By Lemma~\ref{c.inv.muc} we have that
$
(f^{j_n}_{z_n})_*\mu^c_{z_n}=\mu^c_{x_n}.
$
Taking $n\to \infty$ we get that
$
g_*\mu^c_{z}=\mu^c_{x},
$
which implies that $J_{\mu^c}g={\varrho(x,t)}/{\varrho(z,t)}$ is uniformly bounded.
Since, $J_{\mu^c}f^{j_n}_{z_n}$ converges uniformly to $J_{\mu^c}g$, it follows that the sequence
is uniformly bounded.
\end{proof}

\begin{proposition}\label{p.L1}
Let $\varphi:\grassl\to \real$ be a continuous function, $(x_{n})_n$ be a sequence in $\Sigma$
converging to some $x\in\Sigma$ and $(j_n)_n$ be a sequence of integer numbers such that $z_n=\sigma^{-j_n}(x_{n})$ converges and
$(f^{j_n}_{z_n})_n$ converges uniformly to some $g:K\to K$.
Then  $\int \varphi \, dm_{x_{n},f^{j_n}_{z_n}(t)}$ converges to $\int \varphi \, dm_{x,g(t)}$
 in $L^{1}(\mu^{c})$.
\end{proposition}

\begin{proof}
Denote $t_n=f^{j_n}_{z_n}(t)$. Fix $x^s\in \Sigma^-$ and let
$$
h^u_{n}=h^{u}_{(x^s,x_{n}),(x^s,x)} \circ f^{j_n}_{z_n}
\quad\text{and}\quad
H^{u}_{n,t}=H^{u}_{(x^s,x,h^{u}_{n}(t)),\left(x^s,x_{n},t_n\right)}.
$$
Let $\cM$ be the space of probability measures on $\grassl$ with the distance
$$
d(\xi,\eta)=\sup\left\{\left|\int\phi\,d\xi-\int\phi\,d\eta\right|: \sup|\phi|\le 1\right\}.
$$
This generates the weak$^*$ topology, and so $\cM$ is compact. By Remark~\ref{r.jacobianos_muc} and
Lemma~\ref{l.abs.cont.limit}, the Jacobians of $g^{-1}\circ h^{u}_{j}$ with respect to $\mu^c$ are
uniformly bounded. Applying Lemma~\ref{l.keylemma} with $L=\cM$, $H_{j,t}=\left(H^{u}_{j,t}\right)_{\ast}$,
$h_j=g^{-1}\circ h^{u}_{j}$ and $\psi(t)=\hm_{x^s,x,g(t)}$, we get that
$$
\lim_{n\to\infty}\int d\left(\hm_{x^s,x,g(t)}, \left(H^{u}_{n,t}\right)_{\ast}\hm_{x^s,x,h^{u}_{n}(t)}\right) \, d\mu_K(t) = 0.
$$
Observe that $\left(H^{u}_{n,t}\right)_{\ast}\hm_{x^s,x,h^{u}_{n}(t)}=\hm_{x^s,x_{n},t_n}$ and
so the previous relation implies that the sequence $t \mapsto \int \varphi \, d\hm_{x^s,x_{n},t_n}$
converges to $t \mapsto \int \varphi \, d\hm_{x_s,x,g(t)}$ in $L^1(\mu^c)$.

Next, by the definition of the disintegration,
\begin{equation*}
m_{x,t}=\int \hrho(x^s,x,t) \hm_{x^s,x,t} \, d\mu^{s}(x^s)
\end{equation*}
and so
$$
\begin{aligned}
& \int \vert \int \varphi(v)dm_{x_{n},t_n} -\int \varphi(v)dm_{x,g(t)} \vert \, d\mu^c \\
& \hspace*{2cm}\leq \int \int\vert \int\varphi\rho(x^s,x_n,t_n) d\hm_{x^s,x_{n},t_n} \\
&\hspace*{5cm} -\int\varphi\rho(x^s,x,g(t)) d\hm_{x^s,x,g(t)}\vert \, d\mu^c \, d\mu^s.
\end{aligned}
$$
So, noting that the integrand goes to zero as $n\to\infty$, for every $x^s\in \Sigma^-$,
the dominated convergence theorem ensures that
\begin{equation*}
\lim_{n\to \infty} \int \vert \int \varphi(v) \, dm_{x_{n},t_n}-\int \varphi(v) \, dm_{x,g(t)} \vert \, d\mu^c=0,
\end{equation*}
as we wanted to prove.
\end{proof}

The case when $j_n=0$ for every $n\in \natural$ suffices for proving Proposition~\ref{p.conv}
(the full statement will be needed in Section~\ref{s.dirac}). Indeed, it gives that if $(x_{n})_n \to x$
and $\varphi:\grassl\to\real$ is continuous then $t\mapsto\int \varphi \, dm_{x_{n},t}$ converges to
$t\mapsto\int \varphi \, dm_{x,t}$  in $L^{1}(\mu^{c})$. So, there exists a sub-sequence $(n_k)_k$
such that
$$
\int \varphi \, dm_{x_{n_k},t} \to \int \varphi \, dm_{x,t}\text{ for } \mu^{c}-\text{almost every }t.
$$
Moreover, since the space of continuous functions is separable, one can use a diagonal argument
(see e.g. the proof of \cite[Proposition~2.1.6]{FET}) to construct such a sub-sequence independent
of $\varphi$. In other words,
$$
m_{x_{n_k},t} \to m_{x,t} \text{ in the weak$^*$-topology, for } \mu^{c}-\text{almost every }t.
$$
This proves Proposition~\ref{p.conv}.

\begin{corollary}\label{continuidad}
The disintegration $\{m_{x}: x\in \Sigma\}$ is continuous.
\end{corollary}

\begin{proof}
Let $\varphi:K \times \grassl\to \real$ be a continuous function. Given any $(x_{n})_n\to x$, we have that
$$
\begin{aligned}
& \vert \int \varphi dm_{x_{n}}-\int\varphi dm_{x} \vert \\
& \quad
= \vert \int\int \varphi(t,v) \, dm_{x_{n},t}(v) \rho(x_{n},t) \, d\mu^{c}(t) \\
& \qquad\qquad -\int\int \varphi(t,v) \, dm_{x,t}(v) \rho(x,t) \, d\mu^{c}(t)\vert \\
& \quad
\leq \int\vert \int \varphi(t,v) \rho(x_{n},t) \, dm_{x_{n},t}(v)-\int \varphi(t,v) \rho(x,t) \, dm_{x,t}(v)\vert \, d\mu^{c}(t).
\end{aligned}
$$
By Proposition~\ref{p.conv}, up to restricting to a subsequence, we may suppose that $(m_{x_n,t})_n$
converges to $m_{x,t}$ in the weak$^*$ sense, for $\mu^c$-almost every $t$. Then
$$
\int \varphi(t,v)\rho(x_{n},t) \, dm_{x_{n},t}(v) \to \int \varphi(t,v)\rho(x,t) \, dm_{x,t}(v)
$$
for $\mu^c$-almost every $t$. To get the conclusion it suffices to use this observation in the previous inequality,
together with dominated convergence.
\end{proof}

\begin{corollary}\label{c.outro}
We have $m_x = \int (F^k_y)_* m_y \, d\nu^k_x(y)$ for every $x\in\Sigma$ and $k\ge 1$, where $\nu^k_x$ is defined as
$$
\nu^k_x=\sum_{y\in \sigma^{-k}(x)} \frac{1}{J_{\nu}\sigma^k(y)}\delta_y.
$$
\end{corollary}

\begin{proof}
The $F$-invariance of $m$ gives that $m_x = \int (F^k_y)_* m_y \, d\nu^k_x(y)$ for $\muSigma$-almost every $x$, the continuity of the disintegration implies that this extends to every $x\in \Sigma$.
\end{proof}

%%%%%%%%%%%%%%%%%%%%%%%%%%%%%%%%%%%%%%%%%%%%%%%%%%%%%%%%%%%%%%%%%%%%%%%%%%%%%%%%%%%%%%%%%%%%%%%%%%%%%%%%%%%%%%%%%%%%%%%%%%%%%%%%%%%%%%%%%%%%%%%%%%%%%%%%%%%%%%%%%%%%%%%%%%%%%%%%%%%%

\section{Dual graphs of Grassmannian sections}\label{s.zeromes}

Fix $1 \le l < d$. Let $w_1, \dots, w_l$ be a basis of a given subspace $W\in \grassl$.
The exterior product $w_1 \wedge \cdots \wedge w_l$ depends on the choice of the basis,
but its projective class does not. Thus we have a well defined map
\begin{equation}\label{eq.plucker}
\grassl\hookrightarrow \projective\exteriorl,
\quad W \mapsto [w_1 \wedge \cdots \wedge w_l],
\end{equation}
which can be checked to be an embedding: it is called the \emph{Pl\"ucker embedding} of $\grassl$.
The image is the projectivization of the \emph{space of $l$-vectors}
$$
\Lambda^{l}_v (\field^d)
=\lbrace w_1 \wedge w_2\wedge \cdots \wedge w_{l}\in \exteriorl
: w_i\in \field^d\text{ for }1\leq i\leq l\rbrace,
$$
which we denote by $\projective\Lambda^{l}_v (\field^d)$. This is a closed subset
of $\projective\exteriorl$, and it is invariant under the action induced on
$\projective\exteriorl$ by any linear map $B:\field^d\to\field^d$.
See \cite[Section~2]{AvV1} for more information about $l$-vectors.

The \emph{geometric hyperplane} $\Hyp V\subset \grassl$ associated to each $V\in \grassd$
is the set $\Hyp V$ of all subspaces $W\in\grassl$ which are \emph{not} in general position
relative to $V$. In other words,
\begin{equation*}
\Hyp V=\lbrace W\in \grassl: W\cap V\neq \{0\} \rbrace.
\end{equation*}
This may also be formulated using the Pl\"ucker embedding \eqref{eq.plucker}: if $v$ is any $(d-l)$-vector
representing $V$, then $\Hyp V$ consists of the subspaces $W\in\grassl$ represented by $l$-vectors
$w$ such $v \wedge w = 0$.

Let $\sectl$ denote the space of measurable maps $V$ from some full $\mu^c$-measure
subset of $K$ to the Grassmannian manifold of all $l$-dimensional subspaces of $\field^d$.
Define the \emph{dual graph} of each $\cV \in \sectd$ to be
$$
\graf\cV=\lbrace (t,v)\in K\times \grassl: v\in \Hyp\cV(t) \rbrace.
$$
%In what follows, let $\hsec$ denote the space of measurable sections $\cV:K \to \grassd$.
%We define the \emph{dual graph} of each $\cV \in \hsec$ to be
%$$
%\graf\cV=\lbrace (t,v)\in K\times \grassl: v\in \Hyp\cV(t) \rbrace.
%$$

Let $\hm$ be any $u$-state on $\hM\times\grassl$, $m$ be its projection to $M\times\grassl$ and
$\{m_x:x\in\Sigma\}$ be the Rokhlin disintegration of $m$ along the fibers $K\times\grassl$
(recall Section~\ref{s.L1continuity}). The purpose of this section is to prove the following fact:

\begin{proposition}\label{p.medidazero}
We have $m_{x}(\graf\cV)=0$ every $\cV\in \sectd$, $u$-state $\hm$ and every $x\in \Sigma$.
\end{proposition}

The following terminology will be useful.
For each $\hx\in\hSigma$,  consider the following push-forward maps $\sectl \to \sectl$:
\begin{itemize}
\item[(a)] $V \mapsto \cF_{\hx}V$ given by $$\cF_{\hx}V (t)=\hA(\hx,s) V(s)
\text{ with } s = (\hf_{\hx})^{-1}(t);
$$
\item[(b)] $V \mapsto \cH^s_{\hx,\hy} V$ given, for $\hy\in W^s_{\loc}(\hx)$, by
$$
\cH^s_{\hx,\hy}V(t)=H^s_{(\hx,s),(\hy,t)}V(s)
\text{ with } s = h^s_{\hy,\hx}(t);
$$
\item[(c)] $V \mapsto \cH^u_{\hx,\hy} V$ given, for $\hy\in W^u_{\loc}(\hx)$, by
$$
\cH^u_{\hx,\hy}V(t)=H^u_{(\hx,s),(\hy,t)}V(s)
\text{ with } s = h^u_{\hy,\hx}(t).
$$
\end{itemize}
These are well defined because the $h^s$ and $h^u$ holonomy maps are absolutely continuous
with respect to $\mu^c$, as a consequence of \eqref{eq.abs_cont}.

%Given a homoclinic point $\hy$ of $\hx$,
%define $V^i:W^c(\hx)\to \projfield$ by
%\begin{equation}\label{eq.Vi1_new}
%V^i
%= (\cH^u_{\hy,\hx} \circ \cH^s_{\hx,\hy}) E^i
%%= (\cH^u_{\hy,\hx} \circ \cF_{\hy}^{-\imath} \circ \cH^s_{\hx,\hsigma^{\imath}(\hy)}) E^i
%\quad\text{for } i=1, \dots, d.
%\end{equation}
%\begin{equation}\label{eq.Vi}
%V^i = (\cH^s_{\hsigma^{\imath}(\hz),\hp} \circ \cF_{\hz}^{\imath} \circ \cH^u_{\hp,\hz}) E^i
%\quad\text{for } i=1, \dots, d.
%\end{equation}
%Equivalently,
%\begin{equation}\label{eq.Vi2}
%V^i(t) = H^u_{(\hz,t_1),(\hp,t)} \hA^{-\imath}(\hsigma^{\imath}(\hz),t_2) H^s_{(\hp,s),(\hsigma^{\imath}(\hz),t_2)} E^i(s)
%\end{equation}
%where
%\begin{equation}\label{eq.Vi3}
%t_1 = h^u_{\hp,\hz}(t) \text{ and } t_2 = \hf_{\hz}^{\imath}(t_1) \text{ and } s = h^s_{\hsigma^{\imath}(\hz),\hp}(t_2).
%\end{equation}
%Note that $V^i$ is also defined on a full $\hmu^c_{\hx}$-measure set, since the maps $\hf_{\hx}$
%and the composition of the holonomies $h^s_{\hx,\hy}$ and $h^{u}_{\hy,\hx}$ preserves $\hmu^c_{\hx}$
%(Remark~\ref{r.homoclinic}).
%
%\begin{figure}[phtb]
%    \centering
%    \includegraphics[scale=0.4]{def_HA.eps}
%    \caption{Definition of $V^i:W^c(p)\to \projfield$.}\label{figure1}
%\end{figure}

\subsection{Graphs have measure zero}

Starting the proof of Proposition~\ref{p.medidazero}, recall that each $m_{x}$ is a probability
measure on $K\times \grassl$, and
$$
m_{x} = \int m_{x,t}  \varrho(x,t) \, d\mu^{c}(t),
$$
where each $m_{x,t}$ is a probability measure on $\{(x,t)\}\times\grassl$.
Recall also that $x\mapsto m_x$ is continuous, by Corollary~\ref{continuidad}.

Let $x\in\Sigma$ be fixed for the time being, and consider the functions
$$
\begin{aligned}
&G:K\times \grassd \to \real, \quad G\left(t,V\right)=m_{x,t}(\Hyp V) \text{ and }\\
&g:K \to \real, \quad  g(t)=\sup\{m_{x,t}(\Hyp Z): Z \in \grassd\}.
\end{aligned}
$$

\begin{lemma}\label{l.gG}
$G:K\times \grassd \to \real$ and $g(t):K \to \real$ are measurable functions.
\end{lemma}

\begin{proof}
Let $\cP^1\prec\cP^2\prec \cdots$ be an increasing sequence of finite partitions of $\grassd$
such that  $\cP=\vee_{i\in \natural}\cP^i$ is the partition into points (that such a sequence
exists is clear, e.g., because the Grassmannian is compact).
Write $\cP^i=\{P^i_1,\cdots,P^i_{n_i}\}$ and then define
$$
G_n:K\times \grassd \to \real,\quad G_n(t,V)= \sum_{j=1}^{n_j} m_{x,t}(\Hyp P^n_j)\chi_{P^n_j}(V)
$$
where $\chi_B:\grassd\to \real$ denotes the characteristic function of a measurable
set $B \subset \grassd$. By the Rokhlin disintegration theorem, each $t\mapsto m_{x,t}(\Hyp P^n_j)$
is a measurable function. It follows that $G_n$ is measurable for every $n$.
Moreover, $(G_n)_n$ converges to $G$ at every point. Thus, $G$ is measurable. Analogously,
$$
g_n:K \to \real,\quad g_n(t,V)= \max\{m_{x,t}(\Hyp P^n_j): j=1, \dots, n\}
$$
is measurable for every $n$, and $(g_n)_n$ converges pointwise to $g$.
Thus the map $g$ is measurable.
\end{proof}

For each fixed $t\in K$, the function $V \mapsto G(t,V)$ is upper semicontinuous: if $(V_n)_n$
converges to $V$ then $\Hyp V_n$ is contained in a small neighborhood of $ \Hyp V$,
for every large $n$, and then $m_{x,t}(\Hyp V_n)$ can not be much larger than $m_{x,t}(\Hyp V)$.
Since $\grassd$ is compact, it follows that the set
$$
\Gamma(t)=\lbrace V \in\grassd: G(t,V)=g(t)\rbrace
$$
is compact and non-empty (the supremum in the definition of $g$ is attained) for every $t\in K$.

\begin{theorem}[Theorem~III.30 in \cite{CaV77}]\label{t.measurability}
Let $\left( X, \fB, \mu\right)$ be a complete probability space and $Y$ be a separable complete metric space.
Denote by $\fB(Y)$ the Borel $\sigma-$algebra of $Y$.
Let $\kappa(Y)$ be the space of compact subsets of $Y$, with the Hausdorff topology. The following are equivalent:
 \begin{enumerate}
  \item a map $x\to K_{x}$ from $X$ to $\kappa(Y)$ is measurable;
  \item its graph $\lbrace (x,y)\in X \times Y: y \in K_x \rbrace$ is in $\fB \otimes \fB(Y)$;
  \item $\lbrace x\in X: K_x \cap U \neq \emptyset \rbrace \in \fB$ for any open set $U\subset Y$.
 \end{enumerate}
Moreover, any of these conditions implies that there exists a measurable map $\sigma:X \to Y$
such that $\sigma(x)\in K_x$ for every $x\in X$.
\end{theorem}

\begin{lemma}\label{l.supattained}
A given $\cV\in\sectd$ realizes the supremum of
\begin{equation*}
\big\{m_{x}(\graf\cV): \cV\in\sectd \big\}
\end{equation*}
if and only if $\cV(t)\in \Gamma(t)$ for $\mu^c$-almost every $t\in K$.
Moreover, there exists some $\cV_x\in\sectd$ that does realize this supremum.
\end{lemma}

\begin{proof}
By Lemma~\ref{l.gG}, the set
$$
\{(t,V): V \in \Gamma(t)\}
= \{(t,V): G(t,V) = g(t)\}
$$
is a measurable subset of $K \times \grassd$.
Compare the second condition in Theorem~\ref{t.measurability}.
Thus, from the last claim in the theorem, there exists some measurable map $\cV_x:K \to \grassd$
such that $\cV_x(t)\in \Gamma(t)$ for every $t\in K$. In other words,
$$
m_{x,t}(\Hyp\cV_x(t))
= G\big(t,\cV_x(t)\big)
= g(t)
= \sup_Z m_{x,t}(\Hyp Z)
$$
for every $t\in K$. Given any $\cV \in \sectd$ we have
\begin{equation}\label{eq.V0V}
\begin{aligned}
m_{x}(\graf\cV)
 & = \int m_{x,t}(\Hyp \cV(t))\varrho(x,t) \, d\mu^{c}(t) \\
 & \leq \int \sup_{Z} m_{x,t}(\Hyp Z)\varrho(x,t) \, d\mu^{c}(t)\\
 & = \int m_{x,t}(\Hyp \cV_x(t))\varrho(x,t) \, d\mu^{c}(t)
  = m_{x}(\graf\cV_x).
\end{aligned}
\end{equation}
Thus, $\cV_x$ does realize the supremum. Moreover, \eqref{eq.V0V} is an equality if and only if
$G(x,\cV(t)) = g(t)$ for $\mu^{c}$-almost every $t\in K$.
\end{proof}

So far, we kept $x \in \Sigma$ fixed. The next proposition shows that the supremum in
Lemma~\ref{l.supattained} is actually independent of $x$. Denote
$$
\gamma= \sup\{m_{x}(\graf\cV): \cV\in \sectd, x\in\Sigma\}.
$$

\begin{proposition}\label{p.supmeasure}
$\sup\{m_{x}(\graf\cV):\cV\in \sectd\} = \gamma$ for every $x\in\Sigma$.
\end{proposition}

\begin{proof}
Given any cylinder $[J]\subset \Sigma$, choose a positive constant $c<\mu([J])/\kappa$,
where $\kappa>0$ is the constant in \eqref{eq.distortion}.
Consider any $\tilde{x}\in \Sigma$ and $\tilde\cV\in \sectd$.
For each $k \ge 1$ and $y\in \sigma^{-k}(\tilde{x})$, define $\cV^{k}_{y}\in\sectd$ by
$$
\cF^k_{\hat y} \cV^{k}_{y} = \tilde\cV, \text{ that is, }
A^k(t,y)\cV^{k}_{y}(t) = \tilde\cV(f_y^k(t)) \text{ for each $t\in K$.}
$$
By Corollary~\ref{c.outro},
$m_{\tilde{x}}(\graf\tilde\cV)=\int m_y(\graf\cV^{k}_{y}) \, d\mu^u_{k,\tilde{x}}(y)$ and so
\begin{equation*}
\begin{aligned}
m_{\tilde{x}}(\graf\tilde\cV)
& \leq \mu^u_{k,\tilde{x}}([J])\sup\lbrace m_x(\graf\cV): \cV\in\sectd, x\in [J] \rbrace \\
& \hspace{1cm} + (1-\mu^u_{k,\tilde{x}}([J]))\gamma.
\end{aligned}
\end{equation*}
By \eqref{eq.distortion}, there exist arbitrary large values of $k$ such that
$\mu^u_{k,\tilde{x}}([J])\geq c$. Thus
\begin{equation*}
\begin{aligned}
m_{\tilde{x}}(\graf\tilde\cV)
& \leq c \sup\lbrace m_x(\graf\cV): \cV\in\sectd, x\in [J] \rbrace\rbrace\\
&\hspace{1cm} +(1-c)\gamma.
\end{aligned}
\end{equation*}
Varying $\tilde{x}\in \Sigma$ and $\tilde\cV \in \sectd$, we can make the left-hand side arbitrarily
close to $\gamma$. It follows that
\begin{equation*}
\sup\lbrace m_x(\graf\cV): \cV\in\sectd, x\in [J] \rbrace \geq \gamma.
\end{equation*}
The converse inequality is obvious. Thus, we have shown that the supremum over any cylinder $[J]$ coincides
with $\gamma$.

So, given any $x\in \Sigma$ we may find a sequence $(x_n)_n \to x$ such that the sequence
$(m_{x_{n}}(\graf\cV_{x_{n}}))_n$ converges to $\gamma$, where (cf. Lemma~\ref{l.supattained}) each $\cV_{x_n}$
realizes the supremum at $x_n$. Moreover, by Proposition~\ref{p.conv},
up to restricting to a subsequence we may assume that $(m_{x_{n},t})_n \to m_{x,t}$ for every $t$ in some full
$\mu^{c}$-measure set $X\subset K$. Then
\begin{equation}\label{eq.gam}
\begin{aligned}
\gamma
& = \lim_n m_{x_{n}}(\graf\cV_{x_{n}}) \\
& = \lim_n\int m_{x_{n},t}(\Hyp \cV_{x_{n}}(t)) \varrho(x_{n},t) \, d\mu^{c}(t) \\
& \leq \int \limsup_n m_{x_{n},t}(\Hyp \cV_{x_{n}}(t)) \varrho(x_{n},t) \, d\mu^{c}(t).
\end{aligned}
\end{equation}
For each fixed $t\in X$, consider a sub-sequence $(x_{n_{k}})_k$ along which the $\limsup$ is realized.
It is no restriction to suppose that $(\cV_{x_{n_{k}}}(t))_n$ converges to some $V\in\grassd$ as $k\to\infty$.
For any $\epsilon>0$, let $V_\epsilon$ be the closed $\epsilon$-neighborhood of $V$.
The fact that $\cV_{x_{n_{k}}}(t) \subset V_\epsilon$ for every large $k$ implies that
\begin{equation*}
\begin{aligned}
\limsup_k m_{x_{n_k},t}(\Hyp \cV_{x_{n_k}}(t))
\leq \limsup_k m_{x_{n_k},t}(\Hyp V_\epsilon)
\leq m_{x,t}(\Hyp V_\epsilon)
\end{aligned}
\end{equation*}
(because $V_\epsilon$ is closed). Thus, making $\epsilon \to 0$ on the right-hand side,
\begin{equation*}
\begin{aligned}
\limsup_k m_{x_{n_k},t}(\Hyp \cV_{x_{n_k}}(t))
\leq m_{x,t}(\Hyp V)
\leq m_{x,t}(\Hyp \cV_{x}(t)).
\end{aligned}
\end{equation*}
Replacing this in \eqref{eq.gam}, we find that $\gamma \leq \int m_{x,t}(\Hyp \cV_{x}(t))\varrho(x,t) \, d\mu^{c}(t)$
as claimed.
\end{proof}

Having proved Proposition~\ref{p.supmeasure}, the proofs of the following two lemmas are
analogous to those of Lemmas~5.2 and~5.3 in \cite{BoV04}, and so we omit them.

\begin{lemma}\label{p.backgamma}
Given any $x\in \Sigma$ and $\cV\in \sectd$, we have that $m_x(\graf\cV)=\gamma$
 if and only if $m_y(\graf\cF^{-1}_{\hy}(\cV))=\gamma$ for every $y\in \sigma^{-1}(x)$.
\end{lemma}

%\begin{proof}
%Since $m$ is invariant under $F$, we have
%$$
%m_x=\sum_{\sigma(y)=x}\frac{1}{J\sigma(y)}{F_{y}}_{\ast}m_y,
%\quad\text{with}\quad \sum_{\sigma(y)=x}\frac{1}{J\sigma(y)}=1,
%$$
%for $\mu$-almost every $x\in\Sigma$. Moreover, since the disintegration $\{m_x: x\in\Sigma\}$
%is continuous, the identity extends to every $x\in\Sigma$. Then
%\begin{equation}
%m_x(\graf \cV)=\sum_{\sigma(y)=x}\frac{1}{J\sigma(y)}m_y(\graf \cF^{-1}_{\hy}\cV)
%\end{equation}
%Since the maximum is $\gamma$, it follows from these observations that
%$m_x(\graf \cV)=\gamma$ if and only if
%$m_y(\graf \cF^{-1}_{\hy}\cV)=\gamma$ for every $y\in\sigma^{-1}(x)$.
%\end{proof}

As introduced in Section~\ref{ss.jacobians2}, let $\{\hmuSigma_x: x\in \Sigma\}$ be the disintegration
of $\hmuSigma=(\pi_1)_* \hmu$ with respect to the partition into stable sets $\{\Sigma^-\times \{x\}:x\in \Sigma\}$.
Observe that every $\hmuSigma_x$ is equivalent to $\mu^s$.

\begin{lemma}\label{l.almostgamma}
For any $x\in \Sigma$ and any $\cV\in \sectd$ we have that
$$
\hm_{\hx}(\graf\cW)\leq \gamma
\text{ for $\hmuSigma_x$ almost every $\hx\in W^s_{\loc}(x)$.}
$$
Hence, $m_{x}\left(\graf\cW\right)=\gamma$ if and only if $\hm_{\hx}(\graf\cW)=\gamma$
for $\hmuSigma_x$-almost every $\hx\in W^s_{\loc}(x)$.
\end{lemma}

\subsection{Sections over a periodic point}\label{ss.sections}

Let $\hp$ be a fixed (or periodic) point of $\hsigma$ and $\hz$ be a homoclinic point associated to $\hp$.
More precisely, we fix $\hz$ and $\imath \ge 1$ such that $\hz \in W^u_{\loc}(\hp)$ and
$\hsigma^{\imath}(\hz) \in W^s_{\loc}(\hp)$. Denote $p=P(\hp)$ and $z=P(\hz)$.

By the pinching hypothesis in Section~\ref{ss.pinching_twisting}, the Oseledets decomposition
of $F$ restricted to $K$ has the form $E^1(t) \oplus \cdots \oplus E^d(t)$, at $\mu^c$-almost
every $t\in K$, with $\dim E^i(t)=1$ for every $i$. Fix a measurable family
$e^1(t), \dots, e^d(t)$ of bases of $\field^d$ with $e^i(t)\in E^i(t)$ for every $i$.
The matrices of the iterates $A^j(p,t)$ relative to these bases are diagonal:
$$
A^j(p,t) = \left(\begin{array}{cccc}
a^{1,j}(t) & 0 &  \vdots & 0 \\
0 & a^{2,j}(t) & \vdots & 0 \\
\cdots & \cdots & \ddots & \cdots \\
0 & 0 & \vdots & a^{d,j}(t)
\end{array}\right).
$$
We are going to use the associated linear bases of $\exteriord$ and $\exteriorl$,
defined at $\mu^c$-almost every $t\in K$ by
\begin{equation}\label{eq.based-l}
\lbrace e^I(t)=e^{i_{1}}(t) \wedge \dots \wedge e^{i_{d-l}}(t)
\text{, for }I=\lbrace i_1< \cdots < i_{d-l}\rbrace \rbrace
\end{equation}
and
\begin{equation}\label{eq.basel}
\lbrace e^J(t)=e^{j_{1}}(t) \wedge \cdots \wedge e^{j_{l}}(t)\text{, for }J=\lbrace j_1 < \cdots < j_{l} \rbrace \rbrace
\end{equation}
respectively.

%\begin{remark}
%If $\cV\in \sectl$ has a base (can be completed to a base of $\complex^d$) that is sufficiently far away
%of the invariant subspaces, then denoting by
%$$
%\cV(t)=\sum_{I=i_{1},i_{2},\dots,i_{l}}v_{I}(t)E^{i_{1}}_{t}\wedge E^{i_{2}}_{t}\wedge \dots \wedge E^{i_{l}}_{t}
%$$
%we have that for every multi-index $I$, $\lim_{n\to \infty} \frac{1}{n}\log v_I\left(f^{qn}_{p}(t)\right)=0$ for $\mu^c$-almost every $t$.
%\end{remark}

By Lemma~\ref{l.supattained} and Proposition~\ref{p.supmeasure}, we may choose $\cV^{0}\in \sectd$
such that $m_{p}\left(\graf\,\cV^{0}\right)=\gamma$. Define $\cV^j={\cF^{-j}_{p}}\cV^{0}$ for $j \ge 1$.
By Proposition~\ref{p.backgamma}, we also have $m_{p}\left(\graf\,\cV^j\right)=\gamma$ for every $j \ge 1$.
Let $V^0: K \to \exteriord$ be a representative of $\cV^0\in\sectd$, in the sense that $\cV^0(t)$
is the projective class of $V^0(t)$ for each $t$. Then denote $V^j = \cF^{-j}_{p} V^0$ for each $j \ge 1$.
Expressing $V^0$ in terms of the linear bases \eqref{eq.based-l} of $\exteriord$,
\begin{equation*}
V^0(t)=\sum_{I=i_{1},\dots,i_{d-l}} v_{I}(t) e^I(t),
\end{equation*}
we find that
\begin{equation*}
V^j(t)=\sum_{I=i_{1},\dots,i_{d-l}} \frac{v_{I}(f^{j}_{p}(t))}{a^{I,j}(t)} e^I(t),
\end{equation*}
with $a^{I,j}(t) = a^{i_1,j}(t) \cdots a^{i_{d-l},j}(t)$.
Note that $\lim_j (1/j) \log|a^{i,j}|= \lambda_i$, and so
\begin{equation}\label{eq.agrowth}
\lim_j \frac 1j \log|a^{I,j}| = \lambda_{i_1} + \dots + \lambda_{i_{d-l}}.
\end{equation}
Order the multi-indices
$$
I=\{i_1 < \cdots < i_{d-l}\}
$$
in such a way that the sums $\lambda_{i_1} + \cdots + \lambda_{i_{d-l}}$ are in increasing order
(by the pinching condition these sums are all distinct).

Let $\tilde{I}=\{\tilde{i}_1,\cdots,\tilde{i}_{d-l}\}$ be the first multi-index, in this ordering,
for which $v_{\tilde{I}}$ is not essentially zero.
In what follows we assume that $f_{p}$ is ergodic for $\mu^c$.
Then $\tilde{I}$ is the same for every $t\in K$ in a full $\mu^c$-measure set.
The non-ergodic case can be reduced to this one by ergodic decomposition.

\begin{lemma}\label{l.semperda}
The section $t\mapsto E^{\tilde I}(t)$ satisfies $m_{p}\left(\graf\,E^{\tilde I}\right)=\gamma$.
\end{lemma}

\begin{proof}
 By the Birkhoff ergodic theorem,
\begin{equation*}
\lim_n \frac{1}{n}\sum^{n-1}_{j=0}\vert v_{\tilde{I}}  \left(f^{j}_p(t)\right)\vert
= \int \vert v_{\tilde{I}}\vert d\mu^{c}_{p},
\end{equation*}
for $\mu^c_p$-almost every $t\in K$. So, there exist some $\delta>0$ such that
$$
\lim_n \frac{1}{n}\sum^{n-1}_{j=0}\vert v_{\tilde{I}}  \left(f^{j}_p(t)\right)\vert>\delta>0
$$
for every $t$ in some full $\mu^c_p$-measure set. For any $t$ in that set we may consider a
sub-sequence $(j_k)_k$ such that $\vert v_{\tilde{I}}(f_p^{j_{k}})\vert>\delta>0$.
Then
\begin{equation*}
\lim_k \frac{1}{\norm{V^{j_{k}}(t)}}V^{j_{k}}(t)= e^{\tilde{I}}(t),
\end{equation*}
and so
\begin{equation*}
\lim_k \cV^{j_{k}}(t)= E^{\tilde{i}_{1}}(t)+ E^{\tilde{i}_{2}}(t)+ \dots + E^{\tilde{i}_{d-l}}(t)=E^{\tilde{I}}(t)
\end{equation*}
for $\mu^c$-almost every $t$.
We also have that $m_{p,t}\left(\cV^{j}(t)\right)=\sup_{V}m_{p,t}\left(V\right)$,
and then Lemma~\ref{l.supattained} implies that $m_{p,t}(E^{\tilde{I}}(t))=\sup_{V}m_{p,t}\left(V\right)$
for $\mu^c$-almost every $t$, as we claimed.
\end{proof}

This means that, from the start, we may take $V^0(t)$ to coincide with one of the invariant sections
$E^{\tilde I}_t$ given by the Oseledets decomposition, for $\mu^c_p$-almost every $t\in K$.
Define $\cV'=\cF^{-\imath}_z \cV^{0}$. We have that $m_{z}\left(\graf\,\cV'\right)=\gamma$ and,
by Lemma~\ref{l.almostgamma}, $\hm_{(z^u,z)}\left(\graf\,\cV'\right)=\gamma$ for $\mu^{s}$-almost
all $(z^u,z) \in W^{s}_{\loc}(\hz)$. For each $(x^s,p)\in W^{s}_{\loc}(\hp)$,
define $\cV_{(x^s,p)}={\cH^{u}_{(x^s,z),(x^s,p)}}\left(\cV'\right)$,
where $(x^s,z)$ is the unique point in $W^{u}_{\loc}((x^s,p))\cap W^{s}_{\loc}(\hz)$.
Since $\hm$ is a $u$-state, and ${h^{u}_{\hz,\hp}}_{\ast}\mu^c_{\hz}=\mu^c_{\hp}$,
this implies that
\begin{equation}\label{eq.almostall}
\hm_{(x^s,p)}(\cV_{(x^s,p)})=\gamma
\quad\text{for $\mu^s$-almost every $(x^s,p)\in W^{s}_{\loc}(p)$.}
\end{equation}
Denote $\cV^{j}_{(x^s,p)}=\cF^{-j}_{(x^s,p)} \cV_{\hsigma^j((x^s,p))}$ for each $(x^s,p)$
and $j\ge 1$. In particular, $\cV^{j}_{\hp}=\cF^{-j}_{\hp} \cV_{\hp}$.
We are going to prove that for a large set of $j$s the $\cV^j_{\hp}$ have no intersection.

\begin{proposition}\label{p.intersection}
There exists $N\ge 1$ such that for every $M\in \natural$ and $\delta>0$ there exist
$m_1<m_2<\dots <m_M$ and $\tilde{K}\subset K$, with $\mu(\tilde{K})>1-\delta$ and
$\cV^{m_{k_1}}_{\hp}(t)\cap \cdots \cap \cV^{m_{k_N}}_{\hp}(t) = \emptyset$ for any choice
of $m_{k_1}<m_{k_2}<\dots<m_{k_N}$ and $t\in\tilde{K}$.
\end{proposition}

\begin{proof}
Let $V_{\hp}:K\to \exteriord$ be such that $V_{\hp}(t)$ is a unitary $d-l$ vector that
represents $\cV_{\hp}(t)$ in $\exteriord$. We can write it as
\begin{equation*}
V_{\hp}(t)=\sum_{I}v_{I}(t)e^{I}_{\hp}(t).
\end{equation*}
Then
\begin{equation*}
\cF^{-j}_{\hp}V_{\hp}(t)
= {A^{j}(\hp,t)}^{-1}V_{\hp}\left(f^{j}_{p}(t)\right)
= \sum_{I}\frac{v_{I}\left(f^{j}_{p}(t)\right)}{a^{I,j}(t)}e^I_{\hp}(t)
\end{equation*}
Let $N= \dim \exteriorl$. Given any $m_1<m_2<\cdots<m_N$ and $t\in K$ such that
$\cV^{m_1}_{\hp}(t)\cap \cdots \cap \cV^{m_N}_{\hp}(t) \neq \emptyset$,
there is some non-zero $W(t) \in \exteriorl$ such that
\begin{equation}\label{eq.intersectioncond}
W(t) \wedge {\cF^{-m_{k}}_{\hp}}V_{\hp}(t)=0.
\end{equation}
Write
$$
W(t)=\sum_{I}\omega_{I}(t)e^{J}_{\hp}(t)\quad \text{where }
J=\lbrace 1,2,\dots , l \rbrace \setminus I.
$$
Then \eqref{eq.intersectioncond} can be written as
\begin{equation*}
\sum_{I}\frac{v_{I}\left(f^{m_{k}}(t)\right)}{a^{J,m_{k}}(t)}\omega_{I}(t)\varpi_{I}=0,
\end{equation*}
for every $1\leq k\leq N$, where $\varpi_{I}=\pm 1$ is the sign of
$e^{i_{1}}_{\hp}(t)\wedge \dots \wedge e^{i_{d-l}}_{\hp}(t)\wedge e^{j_{1}}_{\hp}(t)\wedge \dots \wedge e^{j_{d-l}}_{\hp}(t)$.
This may be written as
\begin{equation}\label{eq.Lineareq}
B(t) x=0
\end{equation}
where
\begin{equation}\label{eq.sequations}
 B(t)=\left( \begin{array}{ccc}
\frac{v_{I_{1}}(f^{m_{1}}(t))}{a^{I_{1},m_{1}}(t)} & \dots & \frac{v_{I_{N}}(f^{m_{1}}(t))}{a^{I_{N},m_{1}}(t)} \\
\vdots & \vdots & \vdots \\
\frac{v_{I_{1}}(f^{m_{N}}(t))}{a^{I_1,m_{N}}(t)}  & \dots & \frac{v_{I_{N}}\left(f^{m_{N}}(t)\right)}{a^{I_{N},m_{N}}(t)}
\end{array} \right)
\end{equation}
\begin{equation*}
\text{ and } x=\left(\varpi_{I_{1}}\omega_{I_{1}},\dots ,\varpi_{I_{N}}\omega_{I_{N}}\right)^{T}.
\end{equation*}
So, in order to prove that the intersection is necessarily empty, it suffices to show that \eqref{eq.Lineareq}
has no non-zero solutions, in other words, that $\det B(t) \neq 0$.
We are going to use the following fact:

\begin{lemma}\label{l.nonzerodet}
Let $b^{n}_i:K\to \field$, for $1\leq i \leq d$ and $n\in \natural$,
be measurable functions and suppose there exist $\chi_1<\chi_2<\dots<\chi_d $
such that
\begin{equation}\label{eq.bj}
\lim_n \frac{1}{n}\log\vert b^{n}_i(t)\vert=\chi_i
\text{ for $\mu^c$-almost every $t$.}
\end{equation}
Then for every $M\in \natural$ and $\delta>0$ there exist $n_1<n_2<\dots <n_M$ and
$\tilde{K}\subset K$ with $\mu(\tilde{K})>1-\delta$, such that for any choice of
a set $\{k_1, \cdots, k_d\} \subset \{1, \dots, M\}$ with $k_1 < \dots < k_d$,
the matrix
$$
B(t)\in \field^{d\times d},
\quad B_{i,j}(t) = b^{n_{k_j}}_i(t),
$$
has non-zero determinant for every $t\in \tilde{K}$
\end{lemma}

For the proof we need the following simple algebraic fact:

\begin{lemma}\label{l.vandermondelike}
Let $C=(c^{j}_i)_{1 \le i, j \le d}$ be a square matrix with $c^{1}_i \neq 0$
for every $i=1, \dots, d$. Then
$$
\det C = \prod_{i=1}^d c^{1}_i \cdot \det E
$$
where $E = (e^j_i)_{2 \le i, j \le d}$ is defined by
\begin{equation}\label{eq.recorrencia}
e^j_i = \frac{c^j_i}{c^1_i} - \frac{c^j_1}{c^1_1}.
\end{equation}
\end{lemma}

\begin{proof}
The assumption ensures that we may write
\begin{equation*}
\det C= c^1_1 \dots c^1_d
\left[ \begin{array}{cccc}
1 & 1 & \dots & 1\\
\frac{c^2_1}{c^1_1} & \frac{c^2_2}{c^1_2} & \dots & \frac{c^2_d}{c^1_d} \\
\vdots & \vdots & \ddots & \vdots \\
\frac{c^d_1}{c^1_1}& \frac{c^d_2}{c^1_2}& \dots & \frac{c^d_d}{c^1_d}\end{array} \right].
\end{equation*}
Subtracting the first column from each one of the others, we end up with
\begin{equation*}
\det C=c^1_1 \dots c^1_d
\left[ \begin{array}{ccc}
e^2_2 &\dots & e^2_d\\
\vdots & \ddots & \vdots \\
e^d_2 &\dots & e^d_d\end{array} \right],
\end{equation*}
as claimed.
\end{proof}

\begin{proof}[Proof of Lemma~\ref{l.nonzerodet}]
Let us write $b^{1,n}_i = b^n_i$ for $i =1, \dots, d$ and $n\ge 1$. The hypothesis \eqref{eq.bj} implies
that there exist $n_1\ge 1$ and $K_1\subset K$ with $\mu(K_1) > 1- \delta/M$ such that
\begin{equation}\label{eq.nonzero1}
b^{1,n}_i(t) \neq 0
\text{ for $n \ge n_1$, $t\in K_1$ and $i=1, \dots, d$.}
\end{equation}
Let $n_1$ be fixed and define (compare \eqref{eq.recorrencia})
\begin{equation}\label{eq.recorrencia1}
b^{2,n,n_1}_i(t)
= \frac{b^{1,n}_i(t)}{b^{1,n_1}_i(t)} - \frac{b^{1,n}_1(t)}{b^{1,n_1}_1(t)}
\text{ for $i =2, \dots, d$ and $n > n_1$.}
\end{equation}
From \eqref{eq.bj}, and the observation that $\chi_i > \chi_1$, we get that
\begin{equation}\label{eq.bj2}
\lim_n \frac{1}{n} \log |b^{2,n,n_1}_i(t)| = \lim_n \frac{1}{n} \log |b^{1,n}_i(t)| = \chi_i.
\end{equation}
In particular, there exists $n_2 > n_1$ and $K_2 \subset K_1$ with $\mu(K_2) > 1- 2\delta/M$ such
that
\begin{equation}\label{eq.nonzero2}
b^{2,n,n_1}_i(t) \neq 0 \text{ and } b^{1,n}_i(t) \neq 0
\text{ for $n \ge n_2$, $t\in K_2$ and $i=2, \dots, d$}
\end{equation}
(the second condition follows immediately from \eqref{eq.nonzero1} and the fact that $n_2>n_1$,
but we mention it explicitly, for consistency with what follows).

Next, proceed by induction on $l\leq M$:
Suppose that we have defined an increasing sequence of numbers $n_1<\dots< n_l$, a decreasing
sequence of sets $K_l\supset \cdots \supset K_1$ with $\mu(K_l) > 1 - l\delta/M$, and
for every $1\leq j\leq \max\lbrace l,d\rbrace$
a family of measurable functions $b^{j,n_{k_j},\dots,n_{k_1}}_i:K_l\to \field\setminus \lbrace 0 \rbrace$ with
$1\leq k_1<\cdots <k_j \leq l$ satisfying the following relation:
\begin{equation}\label{eq.recorrencia2}
b^{j,n_{k_j},\dots,n_{k_1}}_i(t)=
\frac{b^{j-1,n_{k_j},n_{k_{j-2}},\dots,n_{k_1}}_i(t)}{b^{j-1, n_{k_{j-1}},\dots,n_{k_1}}_i(t)}
-\frac{b^{j-1,n_{k_j},n_{k_{j-2}},\dots,n_{k_1}}_{j-1}(t)}{b^{j-1,n_{k_{j-1}},\dots,n_{k_1}}_{j-1}(t)},
\end{equation}
for $i=j,\dots,d$ and $t\in K_l$.

Suppose $l<M$. Fix, $1<j\leq \max\lbrace l+1,d \rbrace$ and $1\leq k_1<\cdots <k_j \leq l$,
define the functions $b^{j+1,n,n_{k_j},\dots,n_{k_1}}_i:K_l\to \field$
inductively by
\begin{equation}\label{eq.recorrencia3}
b^{j+1,n,n_{k_j},\dots,n_{k_1}}_i(t)=
\frac{b^{j,n,n_{k_{j-1}},n_{k_{j-2}},\dots,n_{k_1}}_i(t)}{b^{j, n_{k_{j}},\dots,n_{k_1}}_i(t)}
-\frac{b^{j,n,n_{k_{j-1}},n_{k_{j-2}},\dots,n_{k_1}}_j(t)}{b^{j,n_{k_{j}},\dots,n_{k_1}}_j(t)}
\end{equation}
for $i=j+1,\dots,d$, $t\in K_l$ and $n\ge n_l$.
Then, arguing as in \eqref{eq.bj2} and using induction on $j$,
\begin{equation}\label{eq.bjk}
\lim_n \frac{1}{n}\log |b^{j+1,n,n_{k_j},\dots,n_{k_1}}_i(t)|=\chi_i.
\end{equation}
Hence we can find $n_{l+1}>n_l$ and $K_{l+1}\subset K_l$ with
$$
\mu(K_{l+1})>1-(l+1)\delta /M
$$
such that, for every $1<j\leq \max\{l+1,d\}$ and $1\leq k_1<\cdots <k_j \leq l$,
$$
b^{j+1,n_{l+1},n_{k_j},\dots,n_{k_1}}_i(t)\neq 0
\text{ for $i=j+1,\dots,d$ and $t\in K_{l+1}$.}
$$

Now fix $\{k_1, \cdots, k_d\} \subset \{1, \dots, M\}$ with $k_1 < \dots < k_d$, and define for $t\in K_M$
the matrix
$$
B(t)\in \field^{d\times d},
\quad B_{i,j}(t) = b^{n_{k_j}}_i(t).
$$
Then, in view of the recursive relations \eqref{eq.recorrencia2}--\eqref{eq.recorrencia3},
we may apply Lemma~\ref{l.vandermondelike} $d$-times to $C=B(t)$ to conclude that
\begin{equation*}
\det B(t)=\prod_{i=1}^d b^{1,n_{k_1}}_i(t)
\prod_{i=2}^d b^{2,n_{k_2},n_{k_1}}_i(t)
\cdots
\prod_{i=d}^d b^{d,n_{k_d},\dots,n_{k_1}}_i(t).
\end{equation*}
This completes our argument.

\end{proof}

Let us go back to proving Proposition~\ref{p.intersection}.
The twisting condition (Section~\ref{ss.pinching_twisting}), implies that
$$
\lim_n \frac 1n \log| v_{I}(f^{n}(t))|=0\text{ for $\mu^c$-almost every $t\in K$}
$$
 and $I=\lbrace i_1<\cdots<i_{d-l}\rbrace$. Then, by \eqref{eq.agrowth},
$$
\lim_n \frac 1n \log \frac{|v_I(f^{n}(t))|}{|a^{I,n}(t)|}
= -(\lambda_{i_1}+\cdots +\lambda_{i_{d-l}})\text{ for $\mu^c$-almost every $t\in K$}.
$$
The pinching condition ensures that these sums are all distinct. Then we may apply Lemma~\ref{l.nonzerodet} to the
functions
$$
b_i^n(t)=\frac{v_{I_i}(f^{n}(t))}{a^{I_i,n}(t)}.
$$
We get that there exist $m_1<\cdots<m_M$ and $\tilde{K}\subset K$ with $\mu(\tilde{K})>1-\delta$
such that for every $\{k_1, \cdots, k_d\} \subset \{1, \dots, M\}$ with $k_1 < \dots < k_d$,
the matrix $B(t)$ defined in \eqref{eq.sequations} is invertible for every $t\in K_M$.
\end{proof}

\begin{proof}[Proof of Proposition~\ref{p.medidazero}]
Assume for the sake of contradiction that $\gamma>0$.
Then let $2\delta<\gamma$ and take $C>0$ large enough that $C(\gamma-2\delta)>1$.
Consider the sequence of integers
$I=\lbrace n_1,n_2,\dots,n_{CN} \rbrace$ given by Proposition~\ref{p.intersection}.
Then there exists $\tilde{K}\subset K$ with $\mu^c(K)>1-\delta$ such that
\begin{equation}\label{eq.Nwisedisjoint}
\cV^{n_{k_1}}(t)\cap \cdots \cap \cV^{n_{k_N}}(t) = \emptyset
\end{equation}
for every $t\in\tilde{K}$ and every $\{k_1 < \cdots < k_N\} \subset \{1, \dots, CN\}$.

First, suppose that $\hm_{\hp}(\graf \cV^j_{\hp})=\gamma$. The property \eqref{eq.Nwisedisjoint}
means that the sets $\cV^{n_{k_i}}(t)$ are $N$-wise disjoint for every $t\in\tilde{K}$. Then,
\begin{equation*}
\begin{aligned}
\hm_{\hp}\Big(\bigcup_{j\in I}\graf \cV^{j}_{\hp}\Big)
& \geq \hm_{\hp}\Big(\bigcup_{j\in I}\graf \cV^{j}_{\hp} \mid \tilde{K}\Big) \\
& \geq \frac{1}{N}\sum_{I}\hm_{\hp}\Big(\graf \cV^{j}_{\hp}\mid \tilde{K} \Big)
\geq C (\gamma-\delta) > 1.
\end{aligned}
\end{equation*}
This is a contradiction because the measure $\hm_{\hp}$ is a probability.

Now we treat the general case. By \eqref{eq.almostall},
$\hm_{(x^s,p)}(\graf \cV^{n_j}_{(x^s,p)})=\gamma$ for every $j$ and
$\mu^s$-almost every $(x^s,p)\in W^{s}_{\loc}(p)$.
In particular, we may a sequence $\big((x^s_k,p)\big)_k \to \hp$ with that property.
Moreover, let $B_k(t)$ be the matrix defined by a system of equations as in \eqref{eq.sequations},
with coefficients depending on $\cV^{n_{i}}_{(x^s_k,p)}$ instead of $\cV^{n_{i}}_{\hp}$.
Keep in mind that, by definition,
$$
\cV_{(x^s_k,p)}=\cH^u_{(x_k^s,z),(x_k^s,p)}(\cV').
$$
The sequence $\cH^u_{(x^s_k,z),(x^s_k,p)}$ converges uniformly to $\cH^u_{\hz,\hp}$ when $k\to\infty$.
Let $\cV_{\hp}=\cH^u_{\hz,\hp}(\cV')$. By Lemma~\ref{l.keylemma} (together with the observation
that $L^1$ convergence implies convergence almost everywhere over some subsequence), up to
restricting to some subsequence of values of $k$ we have
$$
\lim_k \cV^{n_{i}}_{(x^s_k,p)}(t)= \cV^{n_{i}}_{\hp}(t)
\text{ for $\mu^c$-almost every $t\in K$.}
$$
This proves that $B_k$ converges almost everywhere to $B$.

Recall that $\det B(t) \neq 0$ for every $t\in\tilde{K}$, by Lemma~\ref{l.nonzerodet}.
Then, there exist $\tilde{L}\subset\tilde{K}$ with $\mu^{c}(\tilde{L})>1-2\delta$ and
$k_0\ge 1$ such that $\det B_k(t) \neq 0$ for every $t\in\tilde{L}$ and $k\ge k_0$.
Then, applying the previous argument with $(x^s,p)$ and $B_k$ instead of $\hp$ and $B$,
we get that
\begin{equation*}
\begin{aligned}
\hm_{(x^s_k,p)}\Big(\bigcup_{j\in I}\graf \cV^{j}_{(x^s_k,p)}\Big)
& \geq \frac{1}{N}\sum_{I}\hm_{(x^s_k,p)}\Big(\graf \cV^{j}_{(x^s_k,p)}\mid \tilde{L}\Big)\\
& \geq C (\gamma-2\delta) > 1
\end{aligned}
\end{equation*}
for every $k\ge k_0$. Thus, again we get a contradiction (because $\hm_{(x^s_k,p)}$ is a probability).
\end{proof}

%%%%%%%%%%%%%%%%%%%%%%%%%%%%%%%%%%%%%%%%%%%%%%%%%%%%%%%%%%%%%%%%%%%%%%%%%%%%%%%%%%%%%%%%%%%%%%%%%%%%%%%%%%%%%%%%%%%%%%%%%%%%%%%%%%%%%%%%%%%%%%%%%%%%%%%%%%%%%%%%%%%%%%%%%%%%%%%%%%%%

\section{Convergence to Dirac measures}\label{s.dirac}

The goal of this section is to prove the following theorem:
\begin{theorem}\label{t.Dirac}
There exists a measurable map $\xi:\hM \to \grassl$ such that, given any $u$-state $\hm$
on $\hM\times\grassl$, we have
$$
\hm_{\hx,t}=\delta_{\xi(\hx,t)}
\quad\text{for $\hmu$-almost every $(\hx,t)\in\hM$.}
$$
In particular, there exists a unique $u$-state.
\end{theorem}

\subsection{Quasi-projective maps}\label{ss.quasiprojective}

We begin by recalling the notion of \emph{quasi-projective map}, which was introduced by Furstenberg~\cite{Fur73}
and extended by Gol'dsheid, Margulis~\cite{GM89}. See also \cite[Section 2.3]{AvV1} for a related discussion.

Let $v\mapsto [v]$ be the canonical projection from $\field^d$ minus the origin to the projective space $\projfield$.
We call $P_\#:\projfield\to\projfield$ a \emph{projective map} if there is some $P\in GL(d,\field)$ that induces
$P_\#$ through $P_\#([v]) = [P(v)]$. The space of projective maps has a natural compactification, the space of
\emph{quasi-projective maps}, defined as follows. The quasi-projective map $Q_\#$ induced in $\projfield$ by a non-zero,
possibly non-invertible, linear map $Q:\field^d \to \field^d$ is given by
$$
Q_\#([v]) = [Q(v)].
$$
Observe that $Q_\#$ is well defined and continuous on the complement of the \emph{kernel}
$\ker Q_\# = \{[v]:v\in\ker Q\}$.

More generally, one calls $P_\#:\grassl\to\grassl$ a \emph{projective map} if there is
$P\in GL(d,\field)$ that induces $P_\#$ through $P_\#(\xi) = P(\xi)$.
Furthermore, the quasi-projective map $Q_\#$ induced in $\grassl$ by a non-zero, possibly
non-invertible, linear map $Q:\field^d \to \field^d$ is given by
$$
Q_\#\xi = Q(\xi).
$$
Observe that $Q_\#$ is well defined and continuous on the complement of the \emph{kernel}
$\ker Q_\# = \{\xi\in\grassl: \xi \cap \ker Q \neq \{0\}\}$.

The space of quasi-projective maps inherits a topology from the space of non-zero linear maps,
through the natural projection $Q\mapsto Q_\#$.
Clearly, every quasi-projective map $Q_\#$ is induced by some linear map $Q$ such that $\|Q\| = 1$.
It follows that the space of quasi-projective maps on any $\grassl$ is compact for this topology.

The following two lemmas are borrowed from Section 2.3 of \cite{AvV1}:

\begin{lemma}\label{l.kernel}
The kernel $\ker Q_\#$ of any quasi-projective map is contained in some hyperplane of $\grassl$.
\end{lemma}

\begin{lemma}\label{l.topology2}
If $(P_{n})_n$ is a sequence of projective maps converging to some
quasi-projective map $Q$ of $\grassl$, and $(\nu_n)_n$ is a
sequence of probability measures in $\grassl$ converging weakly to
some probability $\nu$ with $\nu(\ker Q)=0$, then
$(P_{n})_*\nu_n$ converges weakly to $Q_*\nu$.
\end{lemma}

\subsection{Convergence}

Recall that, given $1 \le l \le d-1$ and $1 \le i_1< \cdots < i_l \le d$, we write
$$
E^{i_1, \dots, i_l}(t)=E^{i_1}(t) \wedge \cdots \wedge E^{i_l}(t) \in \exteriorl
$$
for every $t\in K$ such that the Oseledets subspaces $E^{i}_t$ are defined.
By a slight abuse of language, we also denote by $E^{i_1, \dots, i_l}(t)$ the associated
vector subspace, that is,
$$
E^{i_1}(t) \oplus \cdots \oplus E^{i_l}(t) \in \grassl.
$$
In this way, each $E^{i_1, \dots, i_l}$ becomes an element of $\sectl$.

Let $\hp\in\hSigma$ be the fixed point of $\hsigma$ and $\hz\in\hSigma$ be a homoclinic point of $\hp$
with $\hz\in W^u_{\loc}(\hp)$. Fix $\imath\in\natural$ such that $\hsigma^{\imath}(\hz) \in W^s_{\loc}(\hp)$.
For each $k\ge 0$, denote $\hz_{k}=\hsigma^{-k}(\hz)$ and $z_{k}=P(\hz_k)$.
Observe that $\hf_{\hz_{k}} = f_{z_k}$ and, similarly, $\hA(\hp,t) = A(p,t)$.
We take advantage of this fact to simplify the notations a bit (removing ``hats´´) in the arguments that follow.
Recall that, given $\hy\in W^u_{\loc}(\hx)$, $t \in K$ and $V\in\sectl$, we take $\cH^u_{\hx,\hy}V\in\sectl$
to be defined by
$$
\cH^u_{\hx,\hy}V(t)=H^u_{(\hx,s),(\hy,t)}V(s)
\text{ with } s = h^u_{\hy,\hx}(t).
$$

\begin{proposition}\label{nova}
Let $\eta = \cH^{u}_{\hp,\hz} E^{1, \dots, l} \in \sectl$.
For every sequence $(k_j)_j\to \infty$ there exists a sub-sequence $(k'_i)_i$ such that
$$
\lim_{i \to \infty}A^{k'_i}\big(z_{k'_i},t_{k'_i}\big)_{\ast} m_{z_{k'_i},t_{k'_i}}=\delta_{\eta(t)},
\text{ where } t_k=(f^k_{z_k})^{-1}(t),
$$
for $\mu^c$-almost every $t\in K$.
\end{proposition}

\begin{proof}
To simplify our notations, let $h=h^u_{\hz,\hp}$ and $h_k=h^u_{\hz_k,\hp}$.
Observe that $f^k_{z_k}= h^u_{\hp,\hz} \circ f^k_{p} \circ h_k$ and
\begin{equation*}
A^{k}(z_{k},t_k)
= H^u_{(\hp,h(t)),(\hz,t)} \, A^{k}(p, h_k(t_k)) \, H^{u}_{(\hz_k,t_k),(\hp, h_k(t_k))}.
\end{equation*}
So $\big(A^{k}_{z_k,t_k}\big)_{\ast} m_{z_k,t_k}$ is equal to
$$
\big( H^u_{(\hp,h(t)),(\hz,t)} \, A^{k}(p, h_k(t_k))\big) \big(H^{u}_{(\hz_k,t_k),(\hp, h_k(t_k))}\big)_* m_{z_k,t_k}.
$$
Note that $H^{u}_{(\hz_k,t_k),(\hp, h_k(t_k))}$ converges uniformly to the identity map $\id$,
because $\hz_k$ converges to $\hp$.

Let $K_0\subset K$ be a full $\mu^c$-measure such that the conclusion of the Oseledets theorem holds
at $(\hp,t)$ for every $t\in K_0$. We claim that for any $t\in K_0$ and every sub-sequence of
$$
A^{k}(p,h_k(t_k))
$$
that converges, the limit is a quasi-projective transformation $Q_\#$ that maps every point outside
$\ker Q_\#$ to $E^{1, \dots, l}(h_k(t_k))\in\grassl$. This can be seen as follows.

Given $w \in \exteriorl$ and $k\ge 1$, we may write
$$
w=\bigoplus_{1 \le i_1 < \cdots < i_l \le d} w^{i_1, \dots, i_l}_k E^{i_1, \dots, i_l}\big((f_p^k)^{-1}h_{\hz,\hp}(t)\big)
$$
with coefficients $w_k^1, \dots, w_k^N \in \field$. It follows from the sub-exponential decay of angles of the Oseledets
splitting that $\vert w^i_k\vert$ grows sub-exponentially in $k$ for every $i=1, \dots, N$.
Recall that $(f^k_p)^{-1}(h(t))=h_k(t_k)$.
Then, the action of $A^{k}(p, h_k(t_k))$ in the projectivization of the exterior power is given by
$$
A^{k}(p, h_k(t_k)) w
= \bigoplus_{j=1}^N w^j_k \frac{\norm{A^k(p,(f^k_p)^{-1}(h(t)))\mid E^{I_j}_{(f^k_p)^{-1}(h(t)}}}{\norm{A^k(p,(f^k_p)^{-1}(h(t))}} E^{I_j}_{h(t)}.
$$
The quotient of the norms converges to zero for any $j>1$. Thus, we have that either
$A^{k}(p, h_k(t_k))w \to E^{I_1}_{h(t)}$ or $A^{k}(p, h_k(t_k))w\to 0$. The latter case means
that $w$ is in the kernel of the limit. Thus, any limit quasi-projective transformation does map
the complement of the kernel to $E^{I_1}_{h(t)}$, as claimed.

As an immediate consequence we get that for any $t\in K_0$ and every sub-sequence of
$$
H^u_{(\hp,h(t)),(\hz,t)} \, A^{k}(p, h_k(t_k))
$$
that converges, the limit is a quasi-projective transformation that maps every point outside the kernel
to $H^u_{(\hp,h(t)),(\hz,t)} E^{I_1}_{h(t)}$.

By Remark~\ref{r.equicontinuity}, the family $\{f^n_{z_k}: n, k \ge 1\}$ is equicontinuous.
Using Arzela-Ascoli, it follows that we can find a sub-sequence of $(k_j)_j$ along which the family
$(f^{k}_{z_{k}})^{-1}$ converges to some $g:K\to K$. Then, by Proposition~\ref{p.conv}, there exists
a further subsequence $(k'_i)_i$ and a full $\mu^c$-measure set $K_1\subset K$ such that
$$
m_{z_{k'_i},t_{k'_i}}\to m_{p,g(t)}
$$
for every $t \in K_1$.

By Proposition~\ref{p.medidazero} and Lemma~\ref{l.abs.cont.limit}, there exists a full $\mu^c$-measure
set $K_2\subset K$ such that $m_{p,g(t)}$ gives zero weight to every hyperplane of $\grassl$ for every $t\in K_2$.
Then, by Lemma~\ref{l.topology2} and the previous observations,
$$
\lim_{k\to \infty} A^{k}(z_k,t_k)_{\ast}m_{z_k,t_k}
=\delta_{\eta(t)}
$$
along any sub-sequence such that $A^k(z_k,t_k)$ converges. This yields the claim of the proposition.
%We are going to prove that the conclusion of the proposition holds for every $t\in K_0 \cap K_1 \cap K_2$.
%Let such a $t$ be fixed. Suppose, by contradiction, that there exist $\delta>0$ and a sub-sequence $(k''_l)_l$
%of $(k_i')_i$  such that
%$$
%\dist_{\cM}\big(A^{k''_l}(z_k,t_{k''_l})_{\ast} m_{z_{k''_l},t_{k''_l}}, \delta_{\eta(t)}\big)>\delta
%$$
%for every $l\ge 1$.
\end{proof}
\begin{remark}\label{r.replace}
The argument remains valid when one replaces the homoclinic point $\hz$ by any other
point in $W^{u}(\hp)$.
\end{remark}

It follows from Proposition~\ref{p.limmesa} that there is a full $\mu^s \times \mu^u$-measure subset
of points $\hx \in \hSigma$ such that
\begin{equation}\label{eq.limitexists}
\lim_{n\to \infty} A^{n}(x_n,t^{\hx}_{n})_{\ast}m_{x_n,t^{\hx}_{n}}=\hm_{\hx,t}
\end{equation}
for $\mu^{c}$-almost every $t\in K$, $x_{n}=P\big(\hsigma^{-n}(\hx)\big)$ and $t^{\hx}_{n}=({f^{n}_{x_{n}}})^{-1}(t)$.
Since the shift is ergodic with respect to the projection of $\hmu$ on $\hSigma$, one may also require that
$$
\lim_{j \to \infty} \hsigma^{-n_{j}}(\hx)=\hz.
$$
for some sub-sequence $(n_{j})_j\to\infty$.

Fix any $\hx\in\hSigma$ such that both conditions hold. Let $k\ge 1$ be fixed, for the time being.
Then \eqref{eq.limitexists} implies that
\begin{equation}\label{eq.limit}
\begin{aligned}
\lim_{j \to \infty} A^{n_j} & (x_{n_j},t^{\hx}_{n_j})_{\ast}m_{x_{n_j},t^{\hx}_{n_j}}\\
&= \lim_{j \to \infty}A^{n_j+k}(x_{n_j+k},t^{\hx}_{n_j+k})_{\ast}m_{x_{n_j+k},t^{\hx}_{n_j+k}}\\
&= \lim_{j \to \infty}A^{n_j}(x_{n_j},t^{\hx}_{n_j})_{\ast} A^{k}(x_{n_j+k},t^{\hx}_{n_j+k})_{\ast} m_{x_{n_j+k},t^{\hx}_{n_j+k}}.
\end{aligned}
\end{equation}
Note also that, by definition,
\begin{equation*}
t^{\hx}_{n_j} = f^{k}_{x_{n_j}+k}(t^{\hx}_{n_j+k}).
\end{equation*}

We use once more the fact that $\{\hf_{\hx}^n: n\in\integer \text{ and } \hx \in\hSigma\}$
is equicontinuous (Remark~\ref{equicontinuitysuffices}).
Using Ascoli-Arzela, it follows that the exists a sequence $(n_j)_j\to\infty$ such that $(f^{n_j}_{x_{n_j}})^{-1}_j$
converges to some $g:K \to K$.
Up to further restricting to a sub-sequence if necessary, Proposition~\ref{p.L1} ensures that
$$
m_{x_{n_j+k},t^{\hx}_{n_j+k}} \text{ converges to } m_{z_{k},g(t)^{\hz}_{k}}
\text{ for $\mu^{c}$-almost every $t$, }
$$
where $z_{k}=P\big(\hsigma^{-k}(\hz)\big)$ and $g(t)^{\hz}_{k} =({f^{k}_{z_{k}}})^{-1}(g(t)) $.

Fix any $t\in K$ such that the previous claims are fulfilled. Let $(n'_i)_i$ be any sub-sequence of $(n_j)_j$
such that $A^{n'_i}\big(x_{n'_i},t^{x}_{n'_i}\big)$ converges to some quasi-projective map
$Q:\grassl \to \grassl$. Then \eqref{eq.limit} may be written as
\begin{equation*}
Q_{\ast} A^{k}\big(z_{k},{g(t)}^{\hz}_{k}\big)_{\ast} m_{z_{k},{g(t)}^{\hz}_{k}}
\end{equation*}
If $\eta(g(t)) \notin \ker Q$ then, making $k\to \infty$, we may use Lemma~\ref{l.topology2} and Proposition~\ref{nova}
to conclude that $\hm_{x,t}=\delta_{Q\,\eta(g(t))}$. This gives the conclusion of Theorem~\ref{t.Dirac} under this assumption.
%Remember that $\ker Q$ is a hyperplane.

%\begin{figure}[phtb]
%\begin{center}
%\psfrag{pp}{$\hp$} \psfrag{zz}{$\hz$} \psfrag{xx}{$\hx$}
%\psfrag{yy}{$\hy$} \psfrag{p}{$p$} \psfrag{z}{$z$}
%\psfrag{w}{$\hsigma^l(\hz)$} \psfrag{x2}{{\footnotesize $x^{n_j}$}}
%\psfrag{x1}{{\footnotesize $x^{n_j+k}$}}
%\psfrag{x3}{{\footnotesize $y^{n_j+k+l}$}}
%\psfrag{x4}{{\footnotesize $y^{n_j+k+l+m}$}} \psfrag{f1}{{\tiny
%$\hsigma^{k}$}} \psfrag{f2}{{\tiny $\hsigma^{n_j}$}} \psfrag{f3}{{\tiny
%$\hsigma^{l}$}} \psfrag{f4}{{\tiny $\hsigma^{m}$}}
%\includegraphics[height=2in]{figure.eps}
%\caption{\label{figure2} Proof of Theorem~\ref{t.Dirac}: avoiding the kernel of $Q$}
%\end{center}
%\end{figure}

Let us show that we can always reduce the proof to this case.
Recall that $\imath\in\integer$ was chosen so that $\hsigma^\imath(\hz)\in W^s_{\loc}(\hp)$.
Define $\hy, \hw \in\hSigma$ by
$$
\begin{aligned}
\hsigma^{n_j+k}(\hy) \in W^{u}_{\loc}(\hsigma^\imath(\hz))\cap W^{s}_{\loc}(x_{n_j+k})\text{ and }
\hsigma^{\imath}(\hw) \in W^{u}_{\loc}(\hsigma^\imath(\hz))\cap W^{s}_{\loc}(z_{k}).
\end{aligned}
$$
Note that $\hy$ depends on $k$ and $j$ and $\hw$ that depends on $k$. We denote
$y=P(\hy)$ and $w=P(\hw)$. Moreover, $y_n=P(\hsigma^{-n}(\hy))$ and $w_n=P(\hsigma^{-n}(\hw))$ for each $n\ge 0$
Let $m \in \natural$ be fixed,
for the time being. We have that $x_{i}=y_{i}$ with $0\leq i \leq n_{j}+k$. So,
$$
\hsigma^{\imath+m}\big(P(\hsigma^{-n_j-k-\imath-m}(\hy))\big)=y_{n_j+k}=x_{n_j+k}.
$$
Also $\hsigma^{-n_j-k}(y)\to \hsigma^\imath (\hw)$, and so
$\hsigma^{-n_j-k-\imath-m}(y)\to \hsigma^{-m} (\hw)$ when $j\to \infty$.
Therefore, by Propositions~\ref{p.limmesa}
$$
\begin{aligned}
\hm_{\hx,t}
& = \lim_{j\to\infty} A^{n_j+k}\big(x_{n_j+k},t^{x}_{n_j+k}\big)_{\ast} m_{x_{n_j+k},t^{x}_{n_j+k}}\\
& = \lim_{j\to\infty} A^{m_j}\big(y_{m_j},t^{\hy}_{m_j}\big)_{\ast} m_{y_{m_j},t^{\hy}_{m_j}}
\end{aligned}
$$
where $m_{j}=n_j+k+\imath+m$. The last expression may be rewritten as
\begin{equation*}
A^{n_j}\big(x_{n_j},t^{\hx}_{n_j}\big)_{\ast} A^{k+\imath}\big(y_{n_j+k+\imath},t^{\hy}_{y_{n_j+k+\imath}}\big)_{\ast} A^ {m}\big(y_{m_j},t^{\hy}_{m_j}\big)_{\ast} m_{y_{m_j},t^{\hy}_{m_j}}.
\end{equation*}
Making $j\to \infty$,
\begin{equation*}
\begin{aligned}
%&A^{n_j}\left(x_{n_j},t^{\hx}_{n_j}\right)\to Q\\
\big({f^{n_j+k+\imath}_{y_{n_j+k+\imath}}}\big) ^{-1} & \to \big({f^{k+\imath}_{\hw}}\big)^{-1} \circ g\\
A^{k+\imath}\big(y_{n_j+k+\imath},t^{\hy}_{y_{n_j+k+\imath}}\big) & \to A^{k+\imath} \big(w,\big({f^{k+\imath}_w}\big)^{-1}g(t)\big)\\
A^{m}\big(y_{m_j},t^{\hy}_{m_j}\big) & \to A^ {m}\big(w_{m},\big({f^{k+\imath+m}_{w_{m}}}\big)^{-1}g(t)\big)
\end{aligned}
\end{equation*}
and, restricting to a sub-sequence if necessary,
\begin{equation*}
m_{y_{m_j},t^{\hy}_{m_j}}\to m_{w_{m},\big({f^{k+\imath+m}_{w_{m}}}\big)^{-1}g(t)}
\text{ for $\mu^c$-almost every $t$.}
\end{equation*}

%Let $\tilde{Q}=Q\circ A^{k+l} \left(w,{f^{k+l}_w}^{-1}g(t)\right)$.

\begin{lemma}\label{l.reduce}
Denote $\tilde\eta(s) = H^{u}_{(\hp,\tilde{h}(s)),(\hw,s)} E^{I_1}_{\tilde{h}(s)}$ with $\tilde{h}(s) = h^{u}_{\hw,\hp}(s)$.
Then there exists a full $\mu^c$-measure set $\tilde{K}\subset K$ and a sub-sequence $(k_j)_j$
such that for every $t\in\tilde{K}$ there exists a sub-sequence $(n_i'=n_i'(t))_i$ of $(n_j)_j(t)$
such that
$$
A^{n'_i}\big(x_{n'_i},t^{\hx}_{n'_i}\big)\circ A^{k+\imath}\big(y_{n'_i+k_j+\imath},t^{\hy}_{y_{n'_i+k_j+\imath}}\big)
$$
converges to some quasi-projective transformation $\tilde{Q}$. Moreover,
$\tilde\eta\big((f^{k_j+\imath}_{w})^{-1}g(t)\big)$ is not in $\ker\tilde{Q}$ if $j$ is sufficiently large,
depending on $t$.
\end{lemma}

\begin{proof}
As before denote $h=h^u_{\hz,\hp}$ and $h_k=h^u_{\hz_k,\hp}$. We begin by constructing the sub-sequence $(k_j)_j$.
Note that $({f^{k+\imath}_{w}})^{-1}=\big({f^{\imath}_{w}}\big)^{-1}\big({f^{k}_{z_k}}\big)^{-1}$
and $\hw\to\hz$ when $k\to \infty$.
First, take a sub-sequence of values of $k$ such that $\big({f^{k}_{z_k}}\big)^{-1}$ converges uniformly to some $\phi$.
Since $h_k$ converges to the identity map, $(f_p^{k})^{-1}h=h_k(f_{z_k}^{k})^{-1}$  also converges uniformly to $\phi$.
Note that $\phi$ is absolutely continuous with respect to $\mu^c$, by Lemma~\ref{l.abs.cont.limit}. Recall that
$$
\tilde\eta\big((f^{k+\imath}_{w})^{-1}g(t)\big)
= H^{u}_{(\hp,\tilde{h}(({f^{k+\imath}_{w}})^{-1})),(\hw,({f^{k+\imath}_{w}})^{-1})} E^{I_1}_{\tilde{h}(({f^{k+\imath}_{w}})^{-1})}
$$
with $\tilde{h}(s) = h^{u}_{\hw,\hp}(s)$. Up to restricting the sub-sequence of values of $k$,
we may use Lemma~\ref{l.keylemma} to get that
\begin{equation}\label{conv.E}
\begin{aligned}
& \tilde\eta\big((f^{k+\imath}_{w})^{-1}g(t)\big) \to \eta\big(\big({f^{\imath}_{z}}\big)^{-1}\phi g(t)\big)
\text{ and }\\
& E((f_p^{k})^{-1} h g(t))\to E(\phi g(t)) \big)
\end{aligned}
\end{equation}
for every $t$ in some full $\mu^c$-measure set $K_1$.
This defines the sub-sequence $(k_j)_j$ in the statement.
In what follows, all the statements on $k$ are meant restricted to this sub-sequence.

The twisting condition implies that
\begin{equation}\label{eq.twistcond}
{\cH^{u}_{\hz,\hp}}{\cF^{\imath}_{z}} E^{I_1}_t \cap \big(E^{j_{l+1}}_t+\dots+E^{j_{d}}_t\big)=\{0\}
\end{equation}
for any $j_{l+1}, \dots, j_d \in \{1, \dots, d\}$ and a full $\mu^c$-measure set of values of $t\in K$.
In other words, $A^{\imath}(\eta(h(t))$ does not belong to any of the hyperplanes of $\grassl$
determined by the Oseledets decomposition at the point $(\hp,t)$. Since $\phi$ and $g$ are
absolutely continuous, there exists a full $\mu^c$-measure set $K_2$ of values of $t$ such that
\eqref{eq.twistcond} holds with $t$ replaced by $\phi(g(t))$.

Take $\tilde{K} = K_1 \cap K_2$.
Fix any $t\in \tilde{K}$ such that in addition $(\hp,hg(t))$ satisfies the conclusion of the Oseledets theorem.
Consider any sub-sequence $(n'_i)_i$ of $(n_j)_j$ such that $A^{n'_i}(x_{n'_i},t^{\hx}_{n'_i})$ converges
to some quasi-projective transformation $Q$ when $i\to\infty$. Then
$$
A^{n'_i}\big(x_{n'_i},t^{\hx}_{n'_i}\big)\circ A^{k+\imath}\big(y_{n'_i+k+\imath},t^{\hy}_{y_{n'_i+k+\imath}}\big)
$$
converges to $\tilde{Q}=Q\circ A^{k+\imath}\big(w,\big({f^{k+\imath}_w}\big)^{-1}g(t)\big)$ when $i\to\infty$.
Moreover,
$$
\begin{aligned}
\ker \tilde{Q}
& = A^{k+\imath} \big(w,\big({f^{k+\imath}_w}\big)^{-1}g(t)\big)^{-1}\ker Q\\
& = A^{\imath} \big(w,\big({f^{k+\imath}_w}\big)^{-1}g(t)\big)^{-1} A^{k}\big(z_{k},\big({f^{k}_{z_{k}}}\big)^{-1}g(t)\big)^{-1} \ker Q.
\end{aligned}
$$
Next, observe that
\begin{equation}\label{eq.TAT}
A^{k}\big(z_{k},\big({f^{k}_{z_{k}}}\big)^{-1}g(t)\big)^{-1}
= \Theta_k \, A^{-k}(p, hg(t)) \, \Theta
\end{equation}
where $h = h^{u}_{\hp,\hz}$ and
$$
\Theta=H^{u}_{(\hz,g(t)),(\hp,h g(t))} \text{ and }
\Theta_k = H^u_{(\hp,(f_p^k)^{-1}h g(t)),(\hz_k,(f_{z_k}^k)^{-1}(g(t))}.
$$

By Lemma~\ref{l.kernel}, the kernel of $Q$ is contained in some hyperplane $\Hyp v$ of $\grassl$.
Hence, $\Theta(\ker Q)$ is contained in the hyperplane $\Theta(\Hyp v)$, of course.
Since we take $t\in K$ to be such that the Oseledets theorem holds at $(\hp,t)$, the backward iterates
$A^{-k} (p,hg(t)) \Theta(\Hyp v)$ are exponentially asymptotic to some hyperplane section $\Hyp E$
that is defined by a $(d-l)$-dimensional sum $E$ of Oseledets subspaces.
This remains true for $\Theta_k A^{-k} (p,hg(t)) \Theta(\Hyp v)$ because $\Theta_k$ converges
exponentially fast to the identity map, since $\hz_k$ converges to $\hp$ exponentially fast.
In other words, using \eqref{eq.TAT},
\begin{equation*}
\dist_{\grassl} \big(A^{k}\big(z_{k},\big({f^{k}_{z_{k}}}\big)^{-1}g(t)\big)^{-1}\Hyp v, \Hyp E((f_p^{k})^{-1} h g(t)) \big)
\to 0
\end{equation*}
exponentially fast as $k\to \infty$. Then, by \eqref{conv.E}, we have that
$A^{k}\big(z_{k},\big({f^{k}_{z_{k}}}\big)^{-1}g(t)\big)^{-1}\Hyp v$ converges to $E(\phi g(t))$. So,
$$
\ker\tilde{Q}\subset  A^{\imath} \big(z,(f^\imath_z)^{-1}\phi g(t)\big)^{-1}\Hyp E(\phi g(t)).
$$

Keep in mind that $\hz\in W^u_{\loc}(\hp)$ and $\imath\in \natural$ is such that $\hsigma^{\imath}(\hz)\in W^s_{\loc}(\hp)$.
Recall also (from Section~\ref{ss.reduction}) that in the present setting all the local stable holonomies $h^s$ and $H^s$
are trivial. Define
\begin{equation*}
\begin{aligned}
V^i(t)
& = H^u_{(\hz,t_1),(\hp,t)} H^s_{(\hp,t_2),(\hz,t_1)} E^i(t_2)\\
& = H^u_{(\hz,t_1),(\hp,t)} \hA^{-\imath}(\sigma^{\imath}(\hz),s) H^s_{(\hp,s),(\hsigma^{\imath}(\hz),s)} \hA^{\imath}(\hp,t_2) E^i(t_2)\\
& = H^u_{(\hz,t_1),(\hp,t)} \hA^{-\imath}(\sigma^{\imath}(\hz),s) E^i(s)
\end{aligned}
\end{equation*}
with $t_1 = h^u_{\hp,\hz}(t)$, $t_2 = h^s_{\hz,\hp}(t_1)$ and $s = \hf_{\hp}^{\imath}(t_2) = \hf_{\hz}^{\imath}(t_1)$.
Then, by the twisting condition, $\oplus_{j\in J}V^j$ cannot intersect any sum of the form
$\oplus_{i \in I} E^i$ with $\# I + \# J =d$. In particular, the distance between
\begin{equation*}
H^u_{(\hz,t_1),(\hp,t)} \hA^{-\imath}(\hsigma^{\imath}(\hz),s) E(s)
\quand
E^{I_1}(s)
\end{equation*}
is positive. Equivalently, the distance between
\begin{equation*}
\hA^{-\imath}(\hsigma^{\imath}(\hz),s) E(s)
\quand
\eta(t_1) = H^u_{(\hp,t),(\hz,t_1)}E^{I_1}(s)
\end{equation*}
is positive. Then $\eta\big(\big({f^{\imath}_{z}}\big)^{-1}\phi g(t)\big)$ does not intersect
$$
\hA^{-\imath}(\hsigma^{\imath}(\hz),\big(\phi g(t)\big)) E\big(\phi g(t)\big)= A^{\imath} \big(z,(f^\imath_z)^{-1}\phi g(t)\big)^{-1} E(\phi g(t)),
$$
which implies that $\eta\big(\big({f^{\imath}_{z}}\big)^{-1}\phi g(t)\big)\notin\ker\tilde{Q}$.
\end{proof}

Having established Lemma~\ref{l.reduce}, we can now use the same argument as previously,
to conclude that $\hm_{\hx,t}=\delta_{\tilde{Q}\eta}$ at $\mu^c$-almost every point also in this case.
To do this, observe that for every $m$ and $k$ fixed there exist a sub-sequence $(m'_i)_i$
of $(m_j)_j$ such that
\begin{equation}\label{eq.convwm}
m_{y_{m'_i},t^{\hy}_{m'_i}}\to m_{w_{m},\big({f^{k+\imath+m}_{w_{m}}}\big)^{-1}g(t)}
\text{ for $\mu^c$-almost every $t$.}
\end{equation}
Using a diagonal argument, we may choose $(m'_i)_i$ to be independent of $k$ and $m$.
Fix, once and for all, a full $\mu^c$-measure subset $K'$ such that  \eqref{eq.convwm}
and the conclusions of Lemma~\ref{l.reduce} and Proposition~\ref{nova} (more precisely,
Remark~\ref{r.replace}) hold for every $t\in K'$.

For each fixed $t\in K'$, fixing $k$ sufficiently large and making $m'_i$ go to infinity
(along the sub-sequence given by Lemma~\ref{l.reduce}), we find that
$$
\hm_{\hx,t}=\tilde{Q}_{*}\big(A^m(w_m,(f^{k+\imath}_{w_m})^{-1}g(t)\big)_* m_{w_m,(f^{k+\imath}_{w_m})^{-1}g(t)}
$$
Then, making $m\to \infty$ and using Lemma~\ref{l.topology2} and Proposition~\ref{nova},
$$
\hm_{\hx,t}=\delta_{\xi(\hx,t)},
$$
where $\xi(\hx,t)=\tilde{Q}\tilde{\eta}((f^{k+\imath}_{w})^{-1}g(t)$.

Thus we proved that $\hm_{\hx,t}$ is a Dirac measure for $\hmuSigma$-almost every $\hx\in \hSigma$ and
$\hmu^c_{\hx}$-almost every $t\in K$. Note also that the set $\tilde{M}\subset \hM$ of points
$\hxt\in \hM$ such that $\hm_{\hxt}$ is a Dirac measure is measurable, since the map
$\hxt\mapsto \hm_{\hxt}$ is measurable and the set of Dirac measures is closed in the weak$^*$ topology
is closed, then  $\tilde{M}$ is measurable.
Thus we have shown that $\tilde{M}$ has total $\hmu$-measure, which completes the
proof of Theorem~\ref{t.Dirac}.

%%%%%%%%%%%%%%%%%%%%%%%%%%%%%%%%%%%%%%%%%%%%%%%%%%%%%%%%%%%%%%%%%%%%%%%%%%%%%%%%%%%%%%%%%%%%%%%%%%%%%%%%%%%%%%%%%%%%%%%%%%%%%%%%%%%%%%%%%%%%%%%%%%%%%%%%%%%%%%%%%%%%%%%%%%%%%%%%%%%%

\section{Orthogonal complement}\label{s.orthogonal}

\subsection{Eccentricity}\label{ss.eccentricity}

Let $L: \field^d\to\field^d$ be a linear isomorphism and $1\leq l \leq d$.
The $l$-dimensional \emph{eccentricity} of $L$ is defined by
$$
E(l, L) = \sup \Big\{\frac{m(L\mid\xi)}{\norm{L\mid \xi^{\perp}}}: \xi\in \grassl\Big\},
\quad m(L\mid \xi) = \norm{(L \mid \xi)^{-1}}^{-1}.
$$
We call any $l$-subspace $\xi \in \grassl$ that realizes the supremum as \emph{most expanded $l$-subspace}.
These always exist, since the Grassmannian is compact and the expression depends continuously on $\xi$.

These notions may be expressed in terms of
the polar decomposition of $L=K' D K$ with respect to any orthonormal basis: denoting
by $a_1 ,\dots, a_d$ the eigenvalues of the diagonal operator $D$, in non-increasing order,
then $E(l, L) = a_l /a_{l+1}$. The supremum is realized by any subspace $\xi$ whose image
under $K$ is a sum of $l$-eigenspaces of $D$ such that the product of the eigenvalues is
$a_1 \cdots a_l$. It follows that $E(l, L)\geq 1$, and the most expanded $l$-subspace is unique
if and only if the eccentricity is strictly larger than $1$.

\begin{proposition}\label{p.exentricidad}
For every $0<c<1$, there exists a set $\hM_c\subset \hM$ with $\hmu(\hM_c)>c$ such that
$E\big(l,A^{n}(\hf^{-n}\hxt)\big)\to \infty$, and the image of the most expanded subspace
by $A^{n}(\hf^{-n}\hxt)$ converges to $\xi\hxt$, restricted to the iterates such that $\hf^{-n}\hxt\in \hM_c$
\end{proposition}

For the proof, let us recall the following fact, whose proof can be found in \cite{AvV1}:

\begin{proposition}\label{p.eccentricity}
Let $\mathcal{N}$ be a weak$^*$ compact family of probabilities on $\grassl$ such that all $\nu \in \mathcal{N}$
give zero weight to every hyperplane.
Let $L_n : \field^d \to \field^d$ be linear isomorphisms such that $(L_n)_∗ \nu_n$ converges to a Dirac
measure $\delta_{\xi}$ as $n \to\infty$, for some sequence $\nu_n$ in $\mathcal{N}$. Then the eccentricity $E(l, L_n )$
goes to infinity and the image $L_n (\zeta_n )$ of the most expanding $l$-subspace of $L_n$
converges to $\xi$.
\end{proposition}

\begin{proof}[Proof of Proposition~\ref{p.exentricidad}]
 Given $0<c<1$ take $M_c\subset M$ to be a compact set, with $\mu(M_c)>c$ and such that the restriction
 of the map $(x,t)\mapsto m_{(x,t)}$ to $M_c$ is continuous.
 This implies that
 $$
 \mathcal{N} =\{ m_{(x,t)};(x,t)\in M_c\}
 $$
 is a weak$^*$ compact subset of the space of probability measures of $\grassl$,
 and every measure in $\mathcal{N}$ gives zero weight to every hyperplane.
 Moreover,
 $$
 A^{n}(\hf^{-n}\hxt)_* m_{P\times \id(\hf^{-n}\hxt)}=\delta_{\xi\hxt}.
 $$
 Take $\hM_c=(P\times \id)^{-1}(M_c)$.
 Then the claim follows from Proposition~\ref{p.eccentricity}, with $L_n=A^{n}(\hf^{-n}\hxt)$.
\end{proof}

%%%%%%%%%%%%%%%%%%%%%%%%%%%%%%%%%%%%%%%%%%%%%%%%%%%%%%%%%%%%%%%%%%%%%%%%%%%%%%%%%%%%%%%%%%%%%%%%%%%%%%%%%%%%%%%%%%%%%%%%%%%%%%%%%%%%%%%%%%%%%%%%%%%%%%%%%%%%%%%%%%%%%%%%%%%%%%%%%%%%

\subsection{Adjoint cocycle}\label{ss.adjoint}

Fix any continuous Hermitian form $\ip{\cdot\, ,\cdot \,}_{\hxt}$ in $\hM \times \field^d$.
Let $\hF^*:\hM\times \field^d\to \hM\times \field^d$ be the adjoint cocycle,
defined over $\hf^{-1}:\hM\to \hM$ by
$$
\hF^*(\hxt,v) = (\hf^{-1}\hxt,\adj\hxt v)
$$
where $\adj\hxt$ is the adjoint $\hA(\hf^{-1}\hxt)^{\ast}$ of the matrix
$\hA(\hf^{-1}\hxt)$ with respect to the Hermitian form. In other words,
$\adj\hxt$ is characterized by
$$
\ip{u,\adj\hxt v}_{\hf^{-1}\hxt} = \ip{\hA(\hf^{-1}\hxt u),v}_{\hxt}
\quad\text{for any $u, v \in \field^d$}
$$
%\begin{equation*}\label{eq.adjoint}
%\adj^n(\hq)=
%\left\{
%	\begin{array}{ll}
%		\hA(f^{-n}\hxt)^{\ast}\ldots \hA(f^{-1}\hxt)^{\ast}  & \text{if } n>0 \\
%		Id & \text{if } n=0 \\
%		{\hA(f^{n-1}\hxt)^{\ast}}^{-1}\ldots {\hA\hxt^{\ast}}^{-1}& \text{if } n<0 \\
%	\end{array}
%\right.
%\end{equation*}\\
and $\hxt\in \hM$.
We have that
$$
W^{ss}_{\hf^{-1}}\hxt=W^{uu}_{\hf}\hxt
\quand
W^{uu}_{\hf^{-1}}\hxt=W^{ss}_{\hf}\hxt.
$$
It is also easy to see that
\begin{equation*}
H^{u,\adj}_{\hxt,\hys}=(H^{s,\hA}_{\hys,\hxt})^{\ast}
\text{ and }
H^{s,\adj}_{\hxt,\hzr}=(H^{u,\hA}_{(\hzr),\hxt})^{\ast},
\end{equation*}
respectively, for any $\hxt,\hys$ in the same $\hf^{-1}$-unstable set and any
$\hxt,\hzr$ in the same $\hf^{-1}$-stable set.

The following fact is well known (see \cite[Proposition~2.7]{Via07} for a similar result):

%Let $f:M\to M$ be a dynamical system, $A:M\to \GL$ and $F_{A}:M\times \complex^d\to M\times \complex^d$
%the induced linear cocycle.
%
%Let $F_{A_{\ast}}:M\times \complex^d\to M\times \complex^d$ be the adjoint cocycle over
%$f^{-1}:M\to M$, defined in Section~\ref{ss.adjoint}

\begin{proposition}\label{p.lyap_adjoint}
The cocycles $\hF$ and $\hF^*$ have the same Lyapunov exponents. Moreover, if $E^j$, $j=1, \dots, k$
are the Oseledets spaces of $\hF$ then the Oseledets spaces of $\hF^*$ are, respectively,
$$
E^j_{\ast}=\left[E^1\oplus \cdots \oplus E^{j-1}\oplus E^{j+1} \oplus \cdots \oplus E^k\right]^{\perp},
j=1, \dots, k.
$$
\end{proposition}
\begin{proposition}\label{p.twist.adj}
$\hA$ is simple, if and only if, $\adj$ is simple.
\end{proposition}

\begin{proof}
Applying Proposition~\ref{p.lyap_adjoint} to the restriction of $\hF$ to the periodic leaf
$\{\hp\}\times K$ we get that the cocycle $\adj$ is pinching, if and only if, $\hA$ is pinching.
Moreover, the Oseledets decomposition $E^1_* \oplus \cdots \oplus E_*^k$ of $\adj$
is given by the orthogonal complements of the Oseledets subspaces of $A$:
$$
E^j_{\ast}=\left[E^1\oplus \cdots \oplus E^{j-1} \oplus E^{j+1} \oplus \cdots \oplus E^k\right]^{\perp}.
$$
We are going to use this for proving the twisting property, as follows.

Let $\phi_{\hp,\hz}=\cH^{u,\hA}_{\hz,\hp}\circ \cH^{s,\hA}_{\hp,\hz}$ and
$\phi^*_{\hp,\hz}=\cH^{u,\adj}_{\hz,\hp}\circ \cH^{s,\adj}_{\hp,\hz}$.
%We have to prove that, for every sum of $l$ of the Oseledets invariant subspaces (for $A_*$),
%the coefficients of the image by $\phi^*_{\hp,\hz}$ in the Oseledets base grow sub-exponentially.
Denote
$$
\begin{aligned}
& h : K \to K, \quad h(t)=h^u_{\hz,\hp}\circ h^s_{\hp,\hz} \quad\text{and}\\
& H_{t} = H^u_{(\hz,h^u_{\hp,\hz}(t)),(\hp,t)}\circ H^s_{(\hp,h^{-1}(t)),(\hz,h^u_{\hp,\hz}(t))},
\end{aligned}
$$
Then, for any $V\in\sectl$,
\begin{equation*}
 \phi_{\hp,\hz}V(t)=H_{t}\left(V(h^{-1}(t)\right)
\quand
\phi^{\ast}_{\hp,\hz}V(t)={H_{h(t)}}^{\ast}\left(V(h(t)\right).
\end{equation*}

First, we treat the case $l=1$. Define measurably for (almost) every $t\in K$ a linear base
of unit vectors $e^j(t)\in E^{j}(t),\,j=1,\dots,d$. The twisting condition means that if
$$
(\phi_{\hp,\hz}e^k)(t)=\sum_{j=1}^{d} a_{k,j}(t)e^{j}(t),
$$
then
$$
\limlog\lvert a_{k,j}(f_{\hp}^n(t))\rvert=0.
$$
We need to deduce the corresponding fact for the adjoint. For this, write
$$
(\phi^{\ast}_{\hp,\hz}e_{\ast}^k)(t)=\sum_{j=1}^{d} \beta_{k,j}(t)e_{\ast}^{j}(t).
$$
Hence
\begin{eqnarray*}
\beta_{k,j}(t)\ip{e_{\ast}^{j}(t),e^j(t)} &=& \ip{\phi^{\ast}_{\hp,\hz}e_{\ast}^k(t),e^j(t)}\\
									&=& \ip{e_{\ast}^k(h(t)),\phi_{\hp,\hz}e^j(h(t))}\\
									&=& \overline{a_{j,k}(h(t))}\ip{e_{\ast}^{k}(h(t)),e^k(h(t))},
\end{eqnarray*}
by definition $\ip{e_{\ast}^{i}(t),e^i(t)}=\cos(\alpha^i(x))$ for every $1\leq i\leq d$,
where \begin{eqnarray*}
 \alpha^i(t)&=&\measuredangle\left(e^i_{\ast}(t),e^i(t)\right)\\
 &=&\frac{\pi}{2}-\measuredangle\left(e^i(t),E^1_t\oplus \cdots \oplus \hat{E^{i}}_t\oplus \cdots \oplus E^k_t\right),
\end{eqnarray*}
so, by the Oseledets theorem,
$$
\limlog \lvert \ip{e_{\ast}^{i}(f_{\hp}^{n}(t)),e^i(f_{\hp}^{n}(t))}\rvert=0
\text{ $\hmu^c_{\hp}$-almost everywhere.}
$$
Then
\begin{equation*}
 \limlog \lvert\beta_{k,j}(f_{\hp}^{n}(t)) \rvert=\limlog \lvert a_{j,k}(f_{\hp}^{n}(h(t)))\rvert
\end{equation*}
also as $h:K\to K$ preserves $\hmu^{c}_{\hp}$ this is true for $\hmu^c_{\hp}$-almost everywhere.

For $l>1$ the proof is just the same, using the inner product induced on $\exteriorl$ by $\ip{\cdot,\cdot}$,
that is,
\begin{equation*}
\ip{v_1\wedge \cdots \wedge v_{l},w_1\wedge \cdots \wedge w_{l}}_{\exteriorl}=\det\left(\ip{v_i,w_j}\right).
\end{equation*}
Thus, we have shown that $A$ is twisting if and only if $A_{\ast}$ is twisting.
\end{proof}

Applying Proposition~\ref{p.exentricidad} to the adjoint cocycle we get:

\begin{corollary}\label{c.xiast}
There exists a section $\xi^{\ast}:\hM\to Grass\left(l,d\right)$ which is invariant under the cocycle $F_{\adj}$
and the unstable linear holonomies of $\adj$.

Moreover, given any $c>0$ there exists $\hM_c\subset \hM$ with $\hmu(\hM_c)>c$ such that,
restricted to the sub-sequence of iterates $k$ such that $\hf^k(p)$ in $\hM_c$, the eccentricity
$E\big(l,\adj^{k}(\hf^k(p))\big)=E\big( l,A^{k}(p) \big)$ goes to infinity and the
image $\adj^{k}\big(\hf^k(p)\big)\zeta^a_{k}(\hf^k(p))$ of the most expanded $l$-subspace tends to
$\xi^{\ast}(p)$ as $k\to\infty$.
\end{corollary}

The next lemma relates the invariant sections of the two cocycles, $F$ and $F_{\adj}$:

\begin{lemma}\label{l.ortogonal}
For $\hmu$-almost every $\hx,t\subset \hM$, the subspace $\xi(\hx,t)$ is transverse to the
orthogonal complement of $\xi^{\ast}(\hx,t)$.
\end{lemma}

\begin{proof}
Recall that the stable linear holonomies of $\hA$ are trivial. Thus, the same is true for the unstable linear
holonomies of $\adj$. So, the fact that $\xi^{\ast}$ is invariant under unstable linear holonomies means that
it is constant on local stable sets of $\hf$.
Then the same is true about his orthogonal complement $\eta(\hx,t)=\xi^{\ast}(\hx,t)^{\perp}$,
which means that it only depends on $\eta(\hx,t)=\eta(x,t)$, where $x=P(\hx)$.
Recall that the graph of $\eta(x,\cdot)$ over $K$ has zero $m_{x}$-measure, by Proposition~\ref{p.medidazero}:
\begin{eqnarray*}
m_{x}(\graf\,\eta_x)&=&\int \int \delta_{\xi_{\hx,t}}(\eta(x,t))d\mu^c(t)d\mu^{s}_x(\hx)\\
&=&\mu^{c}\times \mu^{s}\left(\lbrace\hx,t:\xi(\hx,t)\in \eta(x,t)\rbrace\right) = 0
\end{eqnarray*}
for $ \muSigma$-almost every $x\in\Sigma$.
Hence $\hmu\left(\lbrace\hx,t:\xi(\hx,t)\in \eta(x,t)\rbrace\right)=0$, which proves the lemma.
\end{proof}

%%%%%%%%%%%%%%%%%%%%%%%%%%%%%%%%%%%%%%%%%%%%%%%%%%%%%%%%%%%%%%%%%%%%%%%%%%%%%%%%%%%%%%%%%%%%%%%%%%%%%%%%%%%%%%%%%%%%%%%%%%%%%%%%%%%%%%%%%%%%%%%%%%%%%%%%%%%%%%%%%%%%%%%%%%%%%%%%%%%%

\section{Proof of Theorem~\ref{teo.B} }\label{s.teoB}

Denote by $\eta\hxt\in \grassd$ the orthogonal complement of $\xi^{\ast}\hxt$ at each $\hxt\in \hM$.
Recall that $\xi^{\ast}$ was defined in Corollary~\ref{c.xiast} and is invariant under $\adj$:
$$
\adj\hxt\xi^{\ast}\hxt=\xi^{\ast}(\hf^{-1}\hxt) \text{ for $\hmu$-almost every $\hxt$.}
$$
Consequently, $\eta$ is invariant under $A$.

According to Lemma~\ref{l.ortogonal}, we have that $\field^{d}=\xi\hxt\oplus \eta\hxt$ at $\hmu$-almost every point.
To prove Theorem~\ref{teo.B} we are going to show that the Lyapunov exponents of $A$ along $\xi$ are strictly greater
than those along $\eta$. For that, let
\begin{equation*}
\xi\hxt=\xi^{1}\hxt\oplus\dots\oplus\xi^{u}\hxt\quand\eta\hxt=\eta^{s}\hxt\oplus\dots\oplus\eta^{1}\hxt
\end{equation*}
be the Oseledets decomposition of $A$ restricted to the two invariant sub-bundles,
where $\xi^{u}$ corresponds to the smallest Lyapunov exponent among $\xi^{i}$ and $\eta^{s}$ the largest among all $\eta^{j}$.

Denote $d_u=\dim \xi^{u}$ and $d_s=\dim\eta^s$, and then let $\lambda_u$ and $\lambda_s$ be the Lyapunov exponents
associated to these two sub-bundles, respectively. Define
\begin{equation*}
\Delta^n\hxt=\dfrac{\det\left(A^n\hxt,\xi^u\hxt\right)^{ \frac{1}{d_u}}}{\det\left(A^n\hxt,W\hxt\right)^{ \frac{1}{d_u+d_s}}},
\end{equation*}
where $W\hxt=\xi^u\hxt\oplus \eta^s\hxt$. By the Oseledets theorem
\begin{equation*}
\lim_{n\to\infty} \frac{1}{n}\log \Delta^n\hxt=\frac{d_s}{d_u+d_s}\left(\lambda_u-\lambda_s\right).
\end{equation*}
The proof of the following proposition is identical to the proof of Proposition 7.3 in \cite{AvV1}:

\begin{proposition}\label{p.infinito}
For every $0<c<1$ there exist a set $\hM_c\subset \hM$ with $\hmu(\hM_c)>c$ such that for $\hmu$-almost every $\hxt\in \hM$
\begin{equation*}
\lim_{n\to\infty} \Delta^n\hxt=\infty
\end{equation*}
restricted to the sub-sequence of values $n$ for which $\hf^n\hxt\in \hM_c$.
\end{proposition}

So now fix some $0<c<1$ and $\hM_c$ given by Proposition~\ref{p.infinito}.
Let $g:\hM_c\to\hM_c$ be the first return map:
\begin{equation*}
g\left(\hx,t\right)=\hf^{r(\hx,t)}\left(\hx,t\right).
\end{equation*}
Then we can define the induced cocycle $G:\hM_c\times \field^{d} \to \hM_c\times \field^{d}$
\begin{equation*}
G\left(\left(\hx,t\right),v\right)=\left(g\left(\hx,t\right),D(\hx,t)v\right),
\end{equation*}
where $D(\hx,t)=\hA^{r\left(\hx,t\right)}\left(\hx,t\right)$. It is well known (see~\cite[Proposition~4.18]{LLE})
that the Lyapunov exponents of $G$ with respect to $\frac{1}{\hmu(\hM_c)}\hmu$ are the products
of the exponents of $\hF_{A}$ by the average return time $1/\hmu(\hM_c)$.
Thus, to show that $\lambda_u>\lambda_s$ it suffices to prove the corresponding fact for $G$.

Define
\begin{equation*}
\tilde{\Delta}^{k}(\hxt)=
\dfrac{\det\left(D^k\hxt,\xi^u\hxt\right)^{ \frac{1}{d_u}}}{\det\left(D^k\hxt,W\hxt\right)^{ \frac{1}{d_u+d_s}}}.
\end{equation*}
Then $\tilde{\Delta}^{k}\hxt$ is a sub-sequence of $\Delta^n\hxt$ such that $\hf^n\hxt\in \hM_c$.
So, using Proposition~\ref{p.infinito} we conclude that
\begin{equation*}
\lim_{n\to\infty} \sum_{j=0}^{k-1}\log\tilde{\Delta}\left(g^{j}\hxt\right)=\lim_{n\to\infty} \log\tilde{\Delta}^k\left(p\right)=\infty
\end{equation*}
for $\hmu$-almost every $\hxt\in \hM_c$.

We need the following classical fact (see \cite[Corollary~6.10]{Kr85}):
\begin{lemma}
Let $T:X\to X$ be a measurable transformation preserving a probability measure $\nu$ in $X$,
and $\varphi:X\to \real$ be a $\nu$ integrable function such that
$\lim_{n\to\infty}\sum_{j=0}^{n-1}\left(\varphi \circ T^{j}\right)=+\infty$ at $\nu$ almost every point.
Then $\int \varphi d\nu >0$.
\end{lemma}

Applying the lemma to $T=g$ and $\varphi=\log \tilde{\Delta}$ we find that
\begin{equation*}
\lim_{k\to\infty}\frac{1}{k}\log \tilde{\Delta}^k\hxt=\lim_{k\to\infty}\frac{1}{k}\sum_{j=0}^{k-1}\log \tilde{\Delta}\left(g^j\hxt\right)=\int \log \tilde{\Delta} \frac{d\hmu } {\hmu(\hM_c)}>0
\end{equation*}
at $\hmu$-almost every point. On the other hand the relation between Lyapunov exponents gives that
\begin{equation*}
\lim_{k\to\infty}\frac{1}{k}\log \tilde{\Delta}^k\hxt=\frac{d_s}{d_u+d_s}\left(\lambda_u-\lambda_s\right)\frac{1}{\hmu(\hM_c)}.
\end{equation*}
this means that $\lambda_u>\lambda_fs$, so there is a gap between the first $l$ Lyapunov exponents and the remaining $d-l$ ones.
Since this applies for every $1\leq l \leq d$, we conclude that the Lyapunov spectrum is simple.

What is left to prove is that a simple cocycle is also a continuity point for the Lyapunov exponents.
\begin{proposition}
If $A$ is simple, then, for every $1\leq i\leq d$, the functions $\lambda_i:H^{\alpha}(\hM)\to \real$ are continuous in $A$.
\end{proposition}
\begin{proof}
Take $\Phi_k:\hM \times \projective \field^d \to \real$
$$
\Phi_k(\hx,v)=\frac{\log\norm{\hA_k(\hx)v}}{\norm{v}},
$$
then for every $k\in\natural$, there exists an $F_{\hA_k}$-invariant $u$-state $\hm^u_k$ such that $\lambda_1(\hA_k)=\int\Phi_k d\hm^u_k$.
Passing to a subsequence if necessary we can suppose that $\hm^u_k$ converges in the week$^*$ topology to some $F_{\hA}$-invariant $u$-state $\hm^u_A$.
 By Theorem~\ref{t.Dirac}
$\hm^u_{\hA}=\int \delta_{E^1_{\hx}}d\hmu(\hx)$, this implies that $$
\lambda_1(\hA_k)\to \int \Phi(\hx,v)d\hm^u_{\hA}=\lambda_1(\hA).
$$

Now, using the same argument for every $i$-dimensional Grassmannian we get that $\lambda_1+\cdots+\lambda_i$ is also continuous, concluding the proof.
\end{proof}

The simplicity plus the continuity implies that there exists a neighbourhood of $A$ with simple Lyapunov spectrum. This completes the proof of Theorem~\ref{teo.B}.

%%%%%%%%%%%%%%%%%%%%%%%%%%%%%%%%%%%%%%%%%%%%%%%%%%%%%%%%%%%%%%%%%%%%%%%%%%%%%%%%%%%%%%%%%%%%%%%%%%%%%%%%%%%%%%%%%%%%%%%%%%%%%%%%%%%%%%%%%%%%%%%%%%%%%%%%%%%%%%%%%%%%%%%%%%%%%%%%%%%%

%%%%%%%%%%%%%%%%%%%%%%%%%%%%%%%%%%%%%%%%%%%%%%%%%%%%%%%%%%%%%%%%%%%%%%%%%%%%%%%%%%%%%5

\appendix

\section{Continuous maps are dense in $L^1(M,N)$ }\label{a.denseL1}

Let $M$ be a normal topological space and $N$ be a geodesically convex separable metric space (Section~\ref{s.L1continuity}).
Denote by $\cF$ the set of measurable maps $f:M\to N$. Given any regular $\sigma$-finite Borel measure $\mu$ on $M$,
fix any point $\0\in N$ and define
$
\L1=\lbrace f\in \cF: \int\dist_N\big(f(x),\0\big) \, d\mu(x)<\infty\rbrace.
$
When $\mu$ is a finite measure, the choice of $\0\in N$ is irrelevant: different choices yield the
same space $\L1$.

The function $\dist_{\L1}:\L1\times \L1 \to \real$ defined by
\begin{equation*}
\dist_{\L1}(f,g)=\int d_N \big(f(x),g(x)\big) d\mu(x)
\end{equation*}
is a distance in $\L1$. The special case $N=\real$ of the next proposition is well known,
but here we need the following more general statement:

\begin{proposition}\label{p.denseL1}
The subset of continuous maps $f:M\to N$ is dense in the space $\L1$.
\end{proposition}

We call $s:M\to N$ a \emph{simple map} if there exist points $v_1, \dots, v_k \in N$ pairwise
disjoint measurable sets $A_1, \dots, A_k \subset M$ with finite $\mu$-measure such that
\begin{equation*}
s(x)=\left\{\begin{array}{ll}
v_i & \text{if } x\in A_i \\
\0 & \text{if } x \notin \cup_{i=1}^k A_i \\
\end{array}\right.
\end{equation*}
Proposition~\ref{p.denseL1} is an immediate consequence of Lemmas~\ref{l.simples}
and~\ref{l.simple.cont} below.

\begin{lemma}\label{l.simples}
The set $\cS$ of simple functions is dense in $\L1$.
\end{lemma}

\begin{proof}
Consider any $f \in \L1$. Given $\epsilon>0$, fix a set $K_0\subset M$ with finite $\mu$-measure
and such that
$$
\int_{M\setminus K_0} \dist_N(f(x),\0) \, d\mu(x) \le \frac{\epsilon}{4}.
$$
Let $\lbrace v_1,\dots,v_i,\dots \rbrace$ be a countable dense subset of $N$. The family
$$
\big\{B\big(v_i,\frac{\epsilon}{\mu(K_0)}\big): i \in \natural\}
$$
covers $N$ and, consequently,
$$
B_i=B\big(v_i,\frac{\epsilon}{2\mu(K_0)}\big) \setminus
\bigcup_{j<i} B\big(v_i,\frac{\epsilon}{2\mu(K_0)}\big),
\quad i\in \natural
$$
is a partition of $N$. Then $A_i=K_0\cap f^{-1}(B_i)$, $i \in \natural$ is a partition of $K_0$
into measurable sets. Fix $k\in\natural$ large enough that
$$
\int_{K_0 \setminus \bigcup_{i=1}^k A_i} \dist_N(f(x),\0) \, d\mu(x) \le \frac{\epsilon}{4}.
$$
Now define $s:M\to N$ by
\begin{equation*}
s(x)=\left\{\begin{array}{ll}
         v_i & \text{if } x\in A_i \text{ for } i = 1, \dots, k\\
         \0 &\text{if } x\notin \cup_{i=1}^k A_i.\\
      \end{array}\right.
\end{equation*}
Then
$$
\int_{M\setminus \cup_{i=1}^k A_i} \dist_{N}(f(x),s(x)) \, d\mu (x)
= \int_{M\setminus \cup_{i=1}^k A_i} \dist_{N}(f(x),\0) \, d\mu (x)
\le \frac{\epsilon}{2}
$$
and
$$
\int_{\cup_{i=1}^k A_i} \dist_{N}(f(x),s(x)) \, d\mu (x)
\le \mu\big(\cup_{i=1}^k A_i)\frac{\epsilon}{\mu(K_0)}
\le \frac{\epsilon}{2}.
$$
Thus $\dist_{\L1}(f,s) < \epsilon$, which proves the lemma.
\end{proof}

\begin{lemma}\label{l.simple.cont}
For every $s\in \cS$ and $\epsilon>0$ there exists a continuous map $f:M\to N$
such that  $\dist_{\L1}(f,s) \le \epsilon$.
\end{lemma}

\begin{proof}
Let $A_i$ and $v_i$, $i=1, \dots, k$ be as in the definition of the simple map $s$ and
$\tau\ge 1$ be as in \eqref{eq.geodesicallyconvex}. Denote
$
L = \max\{d(v_i,\0): i=1, \dots, k\}.
$
For each $i=1, \dots, k$, consider a compact set $K_i \subset A_i$ such that
$\mu(A_i\setminus K_i)< \epsilon/(4k\tau L)$.  Since the $K_i$ are pairwise disjoint, and $M$
is assumed to be normal, there exist pairwise disjoint open sets $B_i \supset K_i$,
$i=1, \dots, k$ with $\mu(B_i\setminus K_i)<\epsilon/(4k\tau L)$.
In particular, we also have $\mu(A_i \setminus B_i)<\epsilon/(4k\tau L)$.

By the Urysohn lemma, there are continuous functions $\psi_i:M\to \real$, $i=1, \dots, k$
such that
\begin{equation}
\psi_i(x)=\left\{\begin{array}{ll}
       1 & \text{if } x\in K_i\\
       0 &\text{if } x\notin B_i.\\
      \end{array} \right.
\end{equation}
Now we use the assumption that $N$ is geodesically convex.
For each $i=1, \dots, k$, fix $\lambda_i:\left[ 0,1 \right] \to N$ with $\lambda_i(1)=v_i$ and
$\lambda_i(0)=\0$. Then define $f:M\to N$ by
 \begin{equation*}
 f(x)=\left\{\begin{array}{ll}
       \lambda_i(\psi_i(x)) & \text{if } x\in B_i \text{ with } i=1, \dots, k\\
       \0                             & \text{if } x\notin \bigcup_{i=1}^k B_i.\\
      \end{array}\right.
\end{equation*}
It is clear that $f$ is continuous, because the $B_i$ are open and pairwise disjoint.
Moreover, $f(x)=s(x)$ if
$$
\text{either}\quad
x \in \bigcup_{i=1}^k K_i
\quad\text{or}\quad
x \in M\setminus\big(\bigcup_{i=1}^{k} A_i\cup \bigcup_{i=1}^k B_i \big).
$$
All the other values of $x$ fall into some of the following cases:
\begin{enumerate}
\item $x\in A_i \cap \big(B_j \setminus K_j\big)$ for some $i$ and $j$: then
$$
\begin{aligned}
d_N(f(x),s(x))
& \le d_N(\lambda_j(\psi_j(x)),\0) + d_N(\0,v_i) \\
& \le \tau d_N(v_j,\0) + d_N(\0,v_i)
\le 2 \tau L.
\end{aligned}
$$
\item $x\in A_i \setminus \cup_{j=1}^k B_j$ for some $i$: then
$
d_N(f(x),s(x)) = d(\0,v_i) \le L.
$
\item $x \in B_j \setminus \cup_{i=1}^k A_i$ for some $j$: then
$$
d_N(f(x),s(x))
= d_N(\lambda_j(\psi_j(x)),\0) \le \tau d_N(v_j,\0) \le \tau L.
$$
\end{enumerate}
In either case, $x$ belongs to the set
$$
\bigcup_{i=1}^k (A_i \setminus B_i) \cup \bigcup_{j=1}^k (B_j \setminus K_j)
$$
which, by construction, has $\mu$-measure bounded by $\epsilon/(2\tau L)$.
So,
$$
\dist_{\L1}(f,s)
\le \frac{\epsilon}{2\tau L}\max\big\{d_N(f(x),s(x)): x \in M\big\}  \le \epsilon,
$$
as claimed.
\end{proof}


\begin{thebibliography}{10}

\bibitem{ASV13}
A.~Avila, J.~Santamaria, and M.~Viana.
\newblock Holonomy invariance: rough regularity and applications to {L}yapunov
  exponents.
\newblock {\em Ast\'erisque}, 358:13--74, 2013.

\bibitem{AvV1}
A.~Avila and M.~Viana.
\newblock Simplicity of {L}yapunov spectra: a sufficient criterion.
\newblock {\em Port. Math.}, 64:311--376, 2007.

\bibitem{Extremal}
A.~Avila and M.~Viana.
\newblock Extremal {L}yapunov exponents: an invariance principle and
  applications.
\newblock {\em Invent. Math.}, 181(1):115--189, 2010.

\bibitem{BaK16}
L.~Backes and A.~Kocsard.
\newblock Cohomology of dominated diffeomorphism-valued cocycles over
  hyperbolic systems.
\newblock {\em Ergodic Theory Dynam. Systems}, 36:1703--1722, 2016.

\bibitem{BPV}
L.~Backes, M.~Poletti, and P.~Varandas.
\newblock Simplicity of {L}yapunov spectrum for linear cocycles over
  non-uniformly hyperbolic systems.
\newblock Preprint https://arxiv.org/pdf/1612.05056.pdf.

\bibitem{BGV03}
C.~Bonatti, X.~G{\'o}mez-Mont, and M.~Viana.
\newblock G\'en\'ericit\'e d'exposants de {L}yapunov non-nuls pour des produits
  d\'eterministes de matrices.
\newblock {\em Ann. Inst. H. Poincar\'e Anal. Non Lin\'eaire}, 20:579--624,
  2003.

\bibitem{BoV04}
C.~Bonatti and M.~Viana.
\newblock Lyapunov exponents with multiplicity 1 for deterministic products of
  matrices.
\newblock {\em Ergod. Th. {\&} Dynam. Sys}, 24:1295--1330, 2004.

\bibitem{BP74}
M.~Brin and Ya. Pesin.
\newblock Partially hyperbolic dynamical systems.
\newblock {\em Izv. Acad. Nauk. SSSR}, 1:177--212, 1974.

\bibitem{CaV77}
C.~Castaing and M.~Valadier.
\newblock {\em Convex analysis and measurable multifunctions}.
\newblock Lecture Notes in Mathematics, Vol. 580. Springer-Verlag, 1977.

\bibitem{DuK}
P.~Duarte and S.~Klein.
\newblock Continuity positivity and simplicity of the {L}yapunov exponents for
  linear quasi-periodic cocycles.
\newblock Preprint https://arxiv.org/pdf/1603.06851.pdf.

\bibitem{Fur63}
H.~Furstenberg.
\newblock Non-commuting random products.
\newblock {\em Trans. Amer. Math. Soc.}, 108:377--428, 1963.

\bibitem{Fur73}
H.~Furstenberg.
\newblock Boundary theory and stochastic processes on homogeneous spaces.
\newblock In {\em Harmonic analysis in homogeneous spaces}, volume XXVI of {\em
  Proc. Sympos. Pure Math. (Williamstown MA, 1972)}, pages 193--229. Amer.
  Math. Soc., 1973.

\bibitem{FK60}
H.~Furstenberg and H.~Kesten.
\newblock Products of random matrices.
\newblock {\em Ann. Math. Statist.}, 31:457--469, 1960.

\bibitem{GM89}
I.~Ya. Gol'dsheid and G.~A. Margulis.
\newblock Lyapunov indices of a product of random matrices.
\newblock {\em Uspekhi Mat. Nauk.}, 44:13--60, 1989.

\bibitem{GR86}
Y.~Guivarc'h and A.~Raugi.
\newblock Products of random matrices : convergence theorems.
\newblock {\em Contemp. Math.}, 50:31--54, 1986.

\bibitem{HPS77}
M.~Hirsch, C.~Pugh, and M.~Shub.
\newblock {\em Invariant manifolds}, volume 583 of {\em Lect. Notes in Math.}
\newblock Springer Verlag, 1977.

\bibitem{Kr85}
U.~Krengel.
\newblock {\em Ergodic theorems}.
\newblock De Gruyter Publ., 1985.

\bibitem{LeS82}
F.~Ledrappier and J.-M. Strelcyn.
\newblock A proof of the estimation from below in {P}esin's entropy formula.
\newblock {\em Ergodic Theory Dynam. Systems}, 2:203--219 (1983), 1982.

\bibitem{Ose68}
V.~I. Oseledets.
\newblock A multiplicative ergodic theorem: {L}yapunov characteristic numbers
  for dynamical systems.
\newblock {\em Trans. Moscow Math. Soc.}, 19:197--231, 1968.

\bibitem{Pol}
M.~Poletti.
\newblock Stably positive lyapunov exponents for symplectic linear cocycles
  over partially hyperbolic diffeomorphisms.
\newblock Preprint https://arxiv.org/pdf/1605.00044.pdf.

\bibitem{Via07}
M.~Viana.
\newblock Lyapunov exponents of {T}eichm\"uller flows.
\newblock In {\em Partially hyperbolic dynamics, laminations, and
  {T}eichm\"uller flow}, volume~51 of {\em Fields Inst. Commun.}, pages
  139--201. Amer. Math. Soc., 2007.

\bibitem{Almost}
M.~Viana.
\newblock Almost all cocycles over any hyperbolic system have nonvanishing
  {L}yapunov exponents.
\newblock {\em Ann. of Math.}, 167:643--680, 2008.

\bibitem{LLE}
M.~Viana.
\newblock {\em Lectures on {L}yapunov exponents}, volume 145 of {\em Cambridge
  Studies in Advanced Mathematics}.
\newblock Cambridge University Press, 2014.

\bibitem{FET}
M.~Viana and K.~Oliveira.
\newblock {\em Foundations of Ergodic Theory}.
\newblock Cambridge University Press, 2015.

\bibitem{WaY}
Y.~Wang and J.~You.
\newblock Quasi-periodic {S}chrödinger cocycles with positive {L}yapunov
  exponent are not open in the smooth topology.
\newblock Preprint arXiv:1501.05380, 2015.

\end{thebibliography}
\end{document}